\documentclass{amsart}

\usepackage[utf8]{inputenc}
\usepackage[normalem]{ulem} 	
\usepackage{xparse}         	
\usepackage{comment}

\usepackage{mathtools,amssymb,amsthm,empheq}
\usepackage{bm} 	   	
\usepackage{relsize}   	
\usepackage{mathrsfs}  	
\usepackage{cancel}   	
\usepackage{esint}    	
\usepackage{xcolor} 	
\usepackage{cases}

\usepackage{pgf,tikz}
\usetikzlibrary{cd,arrows,babel,shapes,positioning}

\usepackage{hyperref}
\hypersetup{
	colorlinks,
	citecolor=blue,
	filecolor=blue,
	linkcolor=blue,
	urlcolor=blue
}

\DeclareMathOperator{\Dist}{dist}
\DeclareMathOperator{\Div}{div}

\DeclareMathOperator{\Diam}{diam}
\DeclareMathOperator{\BMO}{BMO}
\DeclareMathOperator{\Loc}{loc}
\DeclareMathOperator*{\Osc}{osc}

\newcommand{\R}{\mathbb{R}}

\newcommand{\CO}{\mathcal{C}}
\newcommand{\EPS}{\varepsilon}
\newcommand{\EMP}{\emptyset}

\newcommand{\Ann}{\mathfrak{A}}
\newcommand{\TT}[1]{\tilde{\tilde{#1}}}
\newcommand{\WORD}[1]{\quad \text{#1} \quad}
\newcommand{\ColorWord}[2]{\color{#1} #2 \color{black} }

\numberwithin{equation}{section}

\theoremstyle{plain}

\newtheorem{thm}[equation]{Theorem}
\newcommand{\refthm}[1]{\emph{\ColorWord{blue}{Theorem} \ref{#1}}}

\newtheorem{lemma}[equation]{Lemma}
\newcommand{\reflemma}[1]{\emph{\ColorWord{blue}{Lemma} \ref{#1}}}

\newtheorem{prop}[equation]{Proposition}
\newcommand{\refprop}[1]{\emph{\ColorWord{blue}{Proposition} \ref{#1}}}

\newtheorem{cor}[equation]{Corollary}
\newcommand{\refcor}[1]{\emph{\ColorWord{blue}{Corollary} \ref{#1}}}

\theoremstyle{definition}

\newtheorem{defin}[equation]{Definition}
\newcommand{\refdef}[1]{\emph{Definition \ref{#1}}}

\theoremstyle{remark}

\newtheorem{rem}[equation]{Remark}
\newcommand{\refrem}[1]{\textit{Remark \ref{#1}}}

\newtheoremstyle{named}{}{}{\itshape}{}{\bfseries}{}{.5em}{#1 #3}
\theoremstyle{named}
\title[Perturbation Theory with BMO Antisymmetric Part]{Perturbation Theory for Second Order Elliptic Operators with BMO Antisymmetric Part}
\author{Martin Dindoš, Erika Nystr\"om, Martin Ulmer}
\address{School of Mathematics, \\
The University of Edinburgh and Maxwell Institute of Mathematical Sciences, Edinburgh UK}
\email{M.Dindos@ed.ac.uk, esatterq@exseed.ed.ac.uk, M.Ulmer@sms.ed.ac.uk}

\date{}

\begin{document}

\maketitle

\begin{abstract}
\noindent
In the present paper we study perturbation theory for the $L^p$ Dirichlet problem on bounded chord arc domains for elliptic operators in divergence form with potentially unbounded antisymmetric part in BMO. Specifically, given elliptic operators $L_0 = \mbox{div}(A_0\nabla)$ and $L_1 = \mbox{div}(A_1\nabla)$ such that the $L^p$ Dirichlet problem for  $L_0$ is solvable for some $p>1$; we show that if $A_0 - A_1$ satisfies certain Carleson condition, then the $ L^q$ Dirichlet problem for $L_1$ is solvable for some $q \geq p$. Moreover if the Carleson norm is small then we may take $q=p$. We use the approach first introduced in Fefferman-Kenig-Pipher '91 on the unit ball, and build on Milakis-Pipher-Toro '11 where the large norm case was shown for symmetric matrices on bounded chord arc domains. We then apply this to solve the $L^p$ Dirichlet problem on a bounded Lipschitz domain for an operator $L = \mbox{div}(A\nabla)$, where $A$ satisfies a Carleson condition similar to the one assumed in Kenig-Pipher '01 and Dindo\v{s}-Petermichl-Pipher '07 but with unbounded antisymmetric part.
\end{abstract}

\tableofcontents 
\bigskip

\section{Introduction}

The study of perturbations of elliptic operators in divergence form \(L:=\mathrm{div}(A\nabla\cdot)\) 
    goes back to a result of Dahlberg \cite{dahlberg_absolute_1986}. 
Specifically, given elliptic operators \( L_0 = \Div(A_0\nabla) \) and \( L_1 = \Div(A_1\nabla) \),
    where we know that the \( L^p \) Dirichlet problem for \( L_0 \) is solvable,
    he considered the discrepancy function \(\EPS(X):=|A_0(X)-A_1(X)|\), 
    and showed that if the measure
\begin{equation}\label{eqq1}
d\mu(Z)=\sup_{X\in B(Z,\delta(Z)/2))}\frac{\EPS(X)^2}{\delta(X)}dZ,
    \qquad\textrm{with}\quad \delta(Z):=\mathrm{dist}(Z,\partial\Omega),
\end{equation}    
    is a Carleson measure with vanishing Carleson norm, 
    then the solvability of the \( L^p \) Dirichlet problem 
    is transferred to \(L_1:=\mathrm{div}(A_1\nabla\cdot)\) with the same exponent \( p \). 
\\

Actually this was formulated in terms of properties 
    of the corresponding elliptic measures \( \omega_0 \) and \( \omega_1 \) since 
we know that the \( L^p \) Dirichlet problem for \( L = \Div(A\nabla\cdot) \) 
    is solvable iff the elliptic measure \( \omega \) associated with \( L \) 
    belongs to the reverse H\"older space \( B_p(d\sigma) \),
    where \( d\sigma \) is surface measure on \( \partial\Omega \).
In this language Dahlberg has shows that if the Carleson norm of \( \mu \) is small, 
    then \(\omega_0\in B_p(d\sigma)\) implies \(\omega_1\in B_p(d\sigma)\). 
    
A natural question that arose was whether the condition on \( \mu \) could be relaxed 
    to draw the weaker conclusion that \(\omega_0\in A_\infty(d\sigma)\) 
        implies \( \omega_1\in A_{\infty}(d\sigma) \), 
        where \( A_{\infty}(d\sigma) = \bigcup_{q>1} B_q(d\sigma) \);
    i.e. transferring solvability to \( L_1 \) but not necessarily with the same exponent.
After some progress was made in \cite{fefferman_criterion_1989} 
    this important results was finally established in \cite{fefferman_theory_1991}. 

To summarize two different types of results were established
\begin{itemize}
    \item[(L)] If the Carleson norm of \( \mu \) is bounded 
        then \(\omega_0\in A_\infty(d\sigma)\) implies \( \omega_1\in A_{\infty}(d\sigma) \).
    \item[(S)] If the Carleson norm of \( \mu \) is small
        then \(\omega_0\in B_p(d\sigma)\) implies \(d\omega_1\in B_p(d\sigma)\).
\end{itemize}
In \cite{dahlberg_absolute_1986} and \cite{fefferman_theory_1991} 
    the results were only proved for symmetric matrices in 
    the case \( \Omega = \mathbb{B}^n \subset \R^n \). The symmetry assumption in these papers is not really necessary and the new techniques established there works with some modifications also in the non-symmetric case. Regarding the assumptions on the domain, there has been some work to extend this result to more general domains.
In \cite{milakis_harmonic_2011} the authors extend (L) to the case 
    where \( \Omega \) is a bounded chord-arc domain (see \refdef{def:CAD}).

These results were recently generalized to 1-sided chord-arc domains in \cite{cavero_perturbations_2019} 
    where the authors show both type (L) and (S) results.
In the second part \cite{cavero_perturbations_2020} 
    they also prove a type (L) results for non-symmetric bounded matrices.
Finally an (S) type result for bounded matrices was obtained in \cite{akman_perturbation_2021}.
\\

In this paper we relax the boundedness hypothesis on the coefficients and assume that an elliptic matrix $A$
has (potentially unbounded) antisymmetric part. These operators where first studied in the elliptic case by Li and Pipher in \cite{li_lp_2019},   where they have shown that under the assumption that the antisymmetric part of the matrix $A$ belongs to BMO space (and the symmetric part is bounded)
then the usual elliptic theory holds for such operators and in particular we have the usual Harnack's inequality, interior and boundary H\"older continuity, etc.

We note that in our approach we need to assume that \( \Omega \) is a chord-arc domain as we currently require the exterior cone condition (see \refrem{rem:ExteriorConeCondition}) to hold. There is an opportunity that new techniques as in \cite{cavero_perturbations_2020,akman_perturbation_2021} will remove this assumption in the future.
 \\
Recall that a matrix \( A \) is \emph{elliptic} if there exists \(\lambda_0\) such that
\begin{equation}\label{eq:elliptic}
    \lambda_0 |\xi|^2 \leq \xi^T A(X) \xi \leq  \lambda_0^{-1} |\xi|^2, 
    \quad \forall \xi \in \R^n, \; a.e. \; X \in \Omega.
\end{equation}
Note that even in the case where the matrix \( A \) is not symmetric,
    ellipticity is only a condition on the symmetric part of the matrix \( A^s \).
    
 For the antisymmetric part \(A^a\) we ask that \(\| A^a \|_{\BMO(\Omega)} \leq \Lambda_0\), i.e.
 \begin{align}\sup_{\substack{Q\subset \Omega\\ Q \textrm{ cubes}}}\fint_Q |A^a(Y)-(A^a)_Q|dY\leq \Lambda_0.\label{eq:A^ainBMO}\end{align}

Our main results are as follows. We generalize  (L) and (S) type results to operators as above in  \refthm{thm:NormBig} and \refthm{thm:NormSmall}.
Instead of using the Carleson measure \( \mu \) defined as in \eqref{eqq1} which uses the $L^\infty$ norm,
we introduce a more generalized version 
    that allows \(A\) to be unbounded, namely that
\[ d\mu'(Z):=\frac{\beta_r(Z)^2}{\delta(Z)}dZ\]
    where  
\begin{align}\beta_r(Z) \coloneqq \left( \fint_{B(Z,\delta(Z)/2)} |A_1 - A_0|^r \right)^{1/r},\label{def:beta_r} \end{align}
	for some large fixed \( 1\leq r<\infty \) which only depends on \(n,\lambda_0\) and \(\Lambda_0\). Recall that by the John-Nirenberg inequality a function in BMO belongs to all $L^r$ spaces $r<\infty$. 

Observe that even if we restrict ourselves to  bounded matrices \(A\) this new Carleson measure \(\mu'\) 
    has smaller Carleson norm than the original measure \(\mu\). In particular, it follows that perturbation result \cite[Theorem 8.1]{milakis_harmonic_2011} is a special case of \refthm{thm:NormBig}. 
\\
We are ready to state  two perturbation results:

\begin{thm}\label{thm:NormBig}
	Let \(\Omega\subset\mathbb{R}^n\) be a bounded chord arc domain and \(L_0=\mathrm{div}(A_0\nabla\cdot)\) 
	    and \(L_1=\mathrm{div}(A_1\nabla\cdot)\) two elliptic operators 
	    that satisfy \eqref{eq:elliptic} and \eqref{eq:A^ainBMO}. 
    Let \(\omega_0\) and \(\omega_1\) be the corresponding elliptic measures. 
    Then there exists \(1\leq r=r(n,\lambda_0,\Lambda_0)<\infty\) such that 
        if \(d\mu'(Z):=\frac{\beta_r(Z)^2}{\delta(Z)}dZ \) is a Carleson measure 
        then \( \omega_0 \in A_\infty(d\sigma) \) implies \( \omega_1 \in A_\infty(d\sigma)\).
    Thus if  the \(L^p\) Dirichlet problem for \(L_0\) is solvable this implies 
        solvability of the \(L^q\) Dirichlet problem for \(L_1\), for some \( q \geq p \).
\end{thm}\smallskip

\begin{thm}\label{thm:NormSmall}
	Let \(\Omega\subset\mathbb{R}^n\) be a bounded chord arc domain and \(L_0=\mathrm{div}(A_0\nabla\cdot)\) 
	    and \(L_1=\mathrm{div}(A_1\nabla\cdot)\) two elliptic operators 
	    that satisfy \eqref{eq:elliptic} and \eqref{eq:A^ainBMO}. 
    Let \(\omega_0\) and \(\omega_1\) be the corresponding elliptic measures. Let $1<p<\infty$ and assume that 
    \( \omega_0 \in B_p(d\sigma) \). Then there exists \(1\leq r=r(n,\lambda_0,\Lambda_0)<\infty\) and \(\gamma=\gamma(n,p,[\omega_0]_{B_p},\lambda_0,\Lambda_0) > 0 \) 
    such that if \( \|\mu'\|_{\CO} \leq \gamma \)
	then \( \omega_1 \in B_p(d\sigma) \).
	Thus if  the \(L^p\) Dirichlet problem for \(L_0\) is solvable 
	    this implies solvability of the \(L^p\) Dirichlet problem for \(L_1\),
        for the same exponent \( p \).
\end{thm}\smallskip
    
(The Carleson norm \(\Vert\cdot\Vert_{\CO}\) is defined below in \refdef{def:CarlesonNorm}.) \\

When we started to develop the above perturbation theory for unbounded operators we had in mind one particular application in the spirit of papers by Kenig and Pipher \cite{kenig_dirichlet_2001} and Dindo\v{s}, Petermichl and Pipher \cite{dindos_lp_2007} and extend such results to unbounded matrices.

From  \cite{kenig_dirichlet_2001,dindos_lp_2007}, it follows that if  $\Omega$ is a Lipschitz domain and $A:\Omega\to\mathbb R^{n\times n}$ is a bounded elliptic matrix such that 
\[ d\tilde{\mu}(X):= \delta(X)^{-1} \Osc_{B(X, \delta(X)/2)} \sup_{ij} |a_{ij}(Z)|^2\]
is a Carleson measure then the \(L^p\) Dirichlet problem is solvable for some large $p<\infty$. Additionally, if $p\in(1,\infty)$ is given and both the Lipschitz character of our domain and the Carleson norm of $\tilde\mu$ is sufficiently small then we can conclude solvability of the \(L^p\) Dirichlet problem for this given value of $p$. So again we have one large-Carleson and one small-Carleson type result.

To obtain this one needs perturbation results since in the paper \cite{dindos_lp_2007} mollification procedure is used to replace above Carleson condition with 
\[d\hat{\mu}(X):= \sup_{B(X, \delta(X)/2)}|\nabla A(Z)|^2\delta(Z).\]
This gives the authors better matrix to work with and get the conclusions. To deduce the same for the original matrix we apply our Theorems \refthm{thm:NormBig} and \refthm{thm:NormSmall} and improve conclusions of \cite{kenig_dirichlet_2001,dindos_lp_2007} to unbounded matrices.\\

We note that under a different assumption of so-called \( t \)-independence of the coefficients on $A$ the solvability of the $L^p$  Dirichlet problem for matrices with BMO antisymmetric part was shown in \cite{hofmann_dirichlet_2022}.

\begin{thm}\label{thm:ApplicationSmallNorm}
	Let \(\Omega\) be a bounded Lipschitz domain with Lipschitz character \(K>0\) (that is the Lipschitz constant of graphs describing $\partial\Omega$ is bounded by $K$).
	    Let \(L_0=\mathrm{div}(A\nabla\cdot)\) be an elliptic operator satisfying \eqref{eq:elliptic} and \eqref{eq:A^ainBMO} 
	    let \(\alpha_r\) be
\begin{align}
    \alpha_r(Z):=\left(\fint_{B(Z,\delta(Z)/2)}|A-(A)_{B(Z,\delta(Z)/2)}|^rdY\right)^{1/r}
    \label{eq:defofalpha_r}.
\end{align}
Then for every \(1<p<\infty\) there exists \(r=r(n,\lambda_0,\Lambda_0)>1\) and \( \EPS = \EPS(p) > 0 \) 
	    such that if
	\[ \| \alpha_r(Z)^2 \delta(Z)^{-1} dZ \|_{\CO} < \EPS \WORD{and} K<\EPS, \] 
	    then \(\omega\in B_p(\sigma)\), i.e. the \( L^p \) Dirichlet problem is solvable for the operator $L_0$ in $\Omega$.
\end{thm}

Similarly,  \cite{kenig_dirichlet_2001} can be improved as follows:

\begin{thm}\label{thm:ApplicationBigNorm}
	Let \(\Omega\) be a bounded Lipschitz domain and \(L_0=\mathrm{div}(A\nabla\cdot)\) 
	    an elliptic operator satisfying \eqref{eq:elliptic} and \eqref{eq:A^ainBMO}. Consider \(\alpha_r\) defined as above in \eqref{eq:defofalpha_r}. 
	Then there exists \(r=r(n,\lambda_0,\Lambda_0)>1\) such that if
	\[  \| \alpha_r(Z)^2 \delta(Z)^{-1} dZ \|_{\CO} < \infty \] 
	     then the corresponding elliptic measure of $L_0$ belongs to \(\omega\in A_\infty(d\sigma)\), 
	     i.e. the \( L^p \) Dirichlet problem for $L_0$ is solvable for all $p\in (p_0,\infty)$ where \emph{some} \(  p_0 >1 \) is sufficiently large.
\end{thm}

It follows the we now have a larger class of elliptic operators 
    that solve the \( L^p \) Dirichlet problem on bounded Lipschitz domains than was previously known, since we are
    replacing the oscillation of \( A \) measured in $L^\infty$ norm by an an \( L^p \) mean oscillation for some large \( p > 2 \).
    
It is worth noting that the study of boundary value problems for scalar elliptic operators has a long history. The reader might be interested to read more in the survey paper \cite{DP22}. \\
    
The paper is organized as follows: 
We start with Section \ref{S:Preliminaries} containing definitions and other preliminaries. 
In Section \ref{S:LargeNormProof}, we outline the proof of \refthm{thm:NormBig},
    which closely follows that of \cite{milakis_harmonic_2011} and \cite{fefferman_theory_1991};
    this contains a subsection with results needed to prove the key identity
\begin{equation}\label{eqF}
 F(X) \coloneqq u_1(X) - u_0(X) = \int_{\Omega} \EPS \nabla u_1(Y) \nabla_Y G_0(X,Y) dY.
 \end{equation}
 We note that the formula  \eqref{eqF} for $F$ comes from the paper  \cite{fefferman_theory_1991} and was also carried over to \cite{milakis_harmonic_2011}. We have realized that the proof that \eqref{eqF} holds in our case does not follow immediately from the properties of the Green's function since the difference $u_1-u_0$ does not have sufficient regularity (c.f. Proposition \ref{prop:GreenExist}). This observation also applies to \cite{fefferman_theory_1991} and \cite{milakis_harmonic_2011} but a rather simple approximation argument can close this gap in the case of bounded coefficients. In our case we had to work a bit harder to  show that \eqref{eqF} is indeed true.
\\

The meat of the proof of \refthm{thm:NormBig} consists of proving \reflemma{lem:2.9} and \reflemma{lemma:2.10/2.16}, 
    which is done in Sections \ref{S:Lemma2.9Proof} and \ref{S:Lemma2.10Proof} respectively.
With these results established, \refthm{thm:NormSmall} follows (Section \ref{S:SmallNormProof}).
Finally, in Section \ref{S:Application} we prove \refthm{thm:ApplicationBigNorm} and \refthm{thm:ApplicationSmallNorm}.

\section{Preliminaries} \label{S:Preliminaries}
Here and in the following sections 
    we implicitly allow all constants to depend on \( n,\lambda_0 \) and \( \Lambda_0 \).

\begin{defin}\label{def:CorcscrewCondition}
	\( \Omega\subset\mathbb R^n \) satisfies the \emph{corkscrew condition} with parameters \( M > 1, r_0 > 0 \) if, 
	    for each boundary ball \( \Delta \coloneqq \Delta(Q,r) \) with \( Q \in \partial \Omega \) 
	    and \( 0 < r < r_0 \),
		there exists a point \( A(Q,r) \in \Omega  \), called a \emph{corkscrew point relative to} \( \Delta \),
		    such that \( B(A(Q,r),M^{-1}r) \subset T(Q,r) \). 
\end{defin}

\begin{defin}
	\( \Omega \) is said to satisfy the \emph{Harnack chain condition} if there is a constant \( c_0 \)
	such that for each \( \rho > 0, \; \Lambda \geq 1, \; X_1,X_2 \in \Omega \) 
	    with \( \delta(X_j) \geq \rho \) and \( |X_1 - X_2| \leq \Lambda \rho \),
	    there exists a chain of open balls \( B_1,\dots,B_N \subset \Omega \) 
	    with \( N \lesssim_{\Lambda} 1, \;  X_1 \in B_1, \; X_2 \in B_N, \; B_i \cap B_{i+1} \neq \EMP \)
	    and \( c_0^{-1} r(B_i) \leq \Dist(B_i,\partial \Omega) \leq c_0 r(B_i) \).
	The chain of balls is called a \emph{Harnack chain}.
\end{defin}

\begin{defin}[NTA]
	\( \Omega \) is an \emph{Non-Tangentially Accessible domain}  
		if it satisfies the Harnack chain condition and 
		\( \Omega, \; \R^n \setminus \Bar{\Omega} \) both satisfy the corkscrew condition.
	If only \( \Omega \) satisfied the corkscrew condition 
	then it is called a \emph{1-sided NTA domain} or \emph{uniform domain}.
\end{defin}

\begin{defin}[CAD]\label{def:CAD}
Let \(\Omega\subset\R^n \). \(\Omega\) is called \emph{chord arc domain} (CAD) 
    if \(\Omega\) is a NTA set of locally finite perimeter and Ahlfors regular boundary, 
    i.e. there exists  \(C\geq 1\) so that for \(r\in (0,\mathrm{diam}(\Omega))\) and \(Q\in \partial \Omega\)
\[C^{-1}r^{n-1}\leq \sigma(B(Q,r))\leq Cr^{n-1}.\]
Here \(B(Q,r)\) denotes the $n$-dimensional ball with radius r 
    and center Q and \(\sigma\) denotes the surface measure. 
The best constant C in the condition above is called the Ahlfors regularity constant.  \\

If we replace NTA domain with 1-sided NTA domain in the above definition 
    then \( \Omega \) is called a \emph{1-sided chord arc domain} (1-sided CAD).

\end{defin}\smallskip

Throughout this paper \( \Omega \) will denote a bounded CAD.

\begin{defin}\label{def:CarlesonNorm} For a measure \(\mu\) on \(\Omega\) if the quantity
\[\Vert \mu \Vert_{\CO}:= \sup_{\Delta\subset\partial\Omega}
    \frac{1}{\sigma(\Delta)}\int_{T(\Delta)}d\mu,\]
is finite then $\mu$ is said to be the Carleson measure and $\|\mu\|_{\CO}$  its Carleson norm. Here
 the Carleson region $T(\Delta)$ of a boundary ball \(\Delta=\Delta(Q,r) \coloneqq B(Q,r) \cap \partial\Omega \) is defined as \(T(\Delta(Q,r))=\overline{B(Q,r)}\cap\Omega\).
\end{defin}

\begin{prop}\label{Prop:ContantAntiSymmMatrix}
    Let \( b \) be a constant anti-symmetric matrix
        and let \( u \in W^{1,2}(E) \) and \( v \in W_0^{1,2}(E) \), with \( E \subset \Omega \) measurable.
    Then
    \[ \int_{E} b\nabla u \cdot \nabla v = 0. \]
\end{prop}

\begin{proof}
Note that if \( b \) is a constant anti-symmetric matrix and \( E \subset \Omega \),
    then for \( u \in W^{1,2}(E) \) and \( \phi \in C_c^\infty(E) \) we have
\begin{align*}
    \int_{E} b\nabla u \cdot \nabla \phi 
    &= \int_{E} b_{ij} \partial_i u  \partial_j \phi
    = \int_{E} u \partial_{i}(b_{ij} \partial_j \phi)
    = \int_{E} u  b_{ij} \partial_{ij} \phi
    = \int_{E} u b_{ji}^T \partial_{ij} \phi
    \\
    &= - \int_{E} u  b_{ji} \partial_{ji} \phi
    = - \int_{E} u  b_{ij} \partial_{ij} \phi
    = - \int_{E} \partial_i (b_{ij} u) \partial_{j} \phi
    \\
    &= - \int_{E}  b_{ij} \partial_i u \partial_{i}  \partial_{j} \phi
    = - \int_{E} b\nabla u \cdot \nabla \phi. 
\end{align*}
\end{proof}

Denoting by \((A_i^a)_E\) the constant matrix of 
component-wise means of $A$ on $E$.  It follows that for $u,v$ as above
\begin{align}\label{prop:TheAvarageIs0}
    \int_{E} A_i \nabla u \cdot \nabla v = \int_{E} (A_i-(A_i^a)_E)\nabla u \cdot \nabla v. 
\end{align}

\subsection{Muckenhoupt and Reverse H\"older spaces}

Let \( \mu \) be a doubling measure on \( \partial \Omega \)
    and let \( w:\partial\Omega\to [0,\infty) \).
Furthermore, let \(1<p<\infty\) and let \( p' \) denote its H\"older conjugate 
    i.e. \( \frac{1}{p} + \frac{1}{p'} = 1 \).

\begin{defin}[Muckenhoupt spaces]
We define the \emph{Muckenhoupt spaces} $A_1(\mu)$, $A_p(\mu)$, $A_\infty(\mu)$ by:
\begin{itemize}
\item \(w \in A_p(\mu) \) iff there exists \(C>0\) such that for all balls \(B\subset \mathbb{R}^n\)
\[\left(\frac{1}{\mu(B)}\int_B wd\mu \right)\left(\frac{1}{\mu(B)}\int_B w^{1-p'}d\mu \right)^{p-1}\leq C<\infty.\]

\item \(w \in A_1(\mu) \) iff there exists \(C>0\) such that 
    for \( \mu \)-a.e. \(x\in \mathbb{R}^n\) and balls \(B=B(x)\subset \mathbb{R}^n\) centered at \(x\)
\[\frac{1}{\mu(B)}\int_{B(x)} wd\mu \leq Cw(x).\]

\item Finally, we set \(A_\infty(\mu) :=\bigcup_{1\leq p<\infty}A_p(\mu) \).
\end{itemize}
\end{defin}

\begin{defin}\label{DefRH}[Reverse H\"older spaces]
We define the \emph{Reverse H\"older spaces} \(B_p(\mu), B_\infty(\mu) \) by:
\begin{itemize}
    \item \(w \in B_p(\mu) \) iff there exists \(C>0\) 
        such that for all balls \(B\subset \mathbb{R}^n\)
        \[\left(\frac{1}{\mu(B)}\int_B w^p d\mu \right)^{1/p}\leq\frac{C}{\mu(B)}\int_B wd\mu.\]
        The best constant in the above estimate we shall denote by $[w]_{B_p}$.
     \item \(w \in B_\infty(\mu) \) iff there exists \(c>0\) such that 
    for a.e. \(x\in \mathbb{R}^n\) and balls \(B=B(x)\subset \mathbb{R}^n\) centered at \(x\)
\[ cw(x)\leq \frac{1}{\mu(B)}\int_{B(x)}wd\mu.\]
\end{itemize}
\end{defin}

It is easy to see that the following hold:
\begin{itemize}
    \item \(A_1(\mu)\subset A_p(\mu)\subset A_q(\mu)\subset A_\infty(\mu)\) for \(1\leq p<q<\infty\),
    \item \(B_q(\mu)\subset B_p(\mu)\) for \(1<p<q \leq \infty\), and
    \item \(A_\infty(\mu)=\bigcup_{p>1}B_p(\mu)\).
\end{itemize}
For more properties of these spaces we refer the reader to \cite{grafakos_modern_2009}.
\\

Suppose now that \( \nu \) is another doubling measure on \( \partial \Omega \).
We say that \( \nu \in A_p(\mu) \; [B_q(\mu)] \) if  \( \nu \ll \mu \) 
    and the Radon-Nikodym \( w \coloneqq \frac{d\nu}{d\mu} \in A_p(\mu) \; [B_q(\mu)] \).

\subsection{The \(L^p\) Dirichlet boundary value problem}

\begin{defin}
    Let \(u:\Omega\to \mathbb{R}\). The
    \emph{nontangential maximal function} \( N[u] : \partial\Omega \to \R \) is defined as
    \[N_\alpha[u](Q)\coloneqq\sup_{X\in\Gamma_\alpha(Q)}|u(X)|, \]
    where
    \[\Gamma_\alpha(Q) \coloneqq \{Y\in \Omega; |Y-Q|<\alpha\delta(Y)\},  \] 
    is the cone of aperture $\alpha$ (for $\alpha>1$). Here \( \delta \) denotes the distance function to the boundary \( \partial\Omega \).
\end{defin}

Let \( L = \Div(A\nabla \cdot) \), where \(A(X)\in \mathbb{R}^{n\times n}\) is a matrix,
    satisfying \eqref{eq:elliptic} and \eqref{eq:A^ainBMO}. 
We say that \(u\in W_{\Loc}^{1,2}(\Omega)\) is weak solution to the equation \(Lu=0\) in \(\Omega\) if
$$
    \int_\Omega A \nabla u \nabla \varphi dx = 0, \quad \forall \varphi \in C^\infty_0(\Omega),
$$
    where \(C^\infty_0(\Omega) \) denotes the space of all smooth functions with compact support.
We know (see e.g. \cite{li_lp_2019}) that if \( f \in C^0(\partial\Omega) \) 
    then there exists a \( u \in W^{1,2}(\Omega) \cap C^0(\bar{\Omega}) \) such that
\[
\begin{cases}
Lu=0, & \quad \text{in } \Omega,
\\
u=f, & \quad \text{on } \partial \Omega.
\end{cases}
\]

\begin{defin}
Let \( \alpha > 0 \).
We say the \(L^p\) Dirichlet problem for the operator \(L\) is \emph{solvable}, 
    if for all boundary data \( f\in L^p(\partial\Omega)\cap C(\partial\Omega)\) 
    the solution  \(u \) as defined above satisfies the estimate
\[
    \Vert N_\alpha(u)\Vert_{L^p(\partial\Omega)}\lesssim_{\alpha} \Vert f\Vert_{L^p(\partial\Omega)}.     
\]

\end{defin}

\subsection{Elliptic measure}
Recall that by Riesz theorem there exists a measure \( \omega^X \) such that that for \( u \) as above 
    \[ u(X) = \int_{\Omega} f d\omega^X. \]
This is called the \emph{elliptic measure with pole at} \( X \).
As noted in the introduction  the \( L^p \) Dirichlet problem is solvable 
    iff \( \omega \in B_{p'}(d\sigma) \).
    For a proof see e.g. \cite{kenig_harmonic_1994} and the references therein.

\subsection{Properties of solutions}
In this sections we include some important results from Li's thesis \cite{li_lp_2019}
    that will be used later.
These results hold for solutions on NTA domains.
First we have reverse Holder's and Caccioppoli's inequalities for the gradient:
\begin{prop}[Lemma 3.1.2]\label{prop:GradRevHol}
    Let \( u \in W_{\Loc}^{1,2}(\Omega) \) be a weak solution.
    Let \( X \in \Omega \) and let \( B_R = B_R(X) \) be such that \( \overline{B}_R \subset \Omega \)
        and let \( 0 < \sigma < 1 \).
    Then there exists  \( p>2 \) such that
    \[ \left( \fint_{B_{\sigma R}} |\nabla u|^p \right)^{1/p} 
        \lesssim_{\sigma} \left( \fint_{B_{R}} |\nabla u|^2 \right)^{1/2}. \]
\end{prop}

\begin{prop}\label{prop:Caccioppoli}
    For a \(C=C(n,\lambda,\Lambda_0)<\infty\) we have for a solution \(u\) and \(B(X,2R)\subset \Omega\)
    \[\fint_{B(X,R)} |\nabla u(Z)|^2dZ\leq \frac{C}{R^2}\fint_{B(X,2R)} |u(Z)|^2 dZ.\]
\end{prop}

\begin{prop}[Lemma 3.1.4]\label{prop:DiGNMestimate}
    Let \(u\in W^{1,2}_{loc}(\Omega)\) be a weak solution of $Lu=0$ and \(\overline{B(X,2R)}\subset \Omega\). Then
    \[\sup_{B(X,R)} |u| \leq C(n,\lambda,\Lambda_0) \left(\fint_{B(X,2R)} |u|^2\right)^{\frac{1}{2}}.\]
\end{prop}

Harnack's inequality also does hold:
\begin{prop}[Lemma 3.1.8]\label{prop:Harnack}
Let \(u\in W^{1,2}_{loc}(\Omega)\) be a nonnegative weak solution and \(\overline{B(X,2R)}\subset \Omega\). Then
    \[\sup_{B(X,R)} |u| \leq C(n,\lambda,\Lambda_0) \inf_{B(X,R)} |u|.\]
\end{prop}

We also have the comparison principle.
\begin{prop}[Proposition 4.3.6]\label{prop:CompPrinc}
    Let \( u,v \in W^{1,2}(T_{2r}(Q)) \cap C^0(\overline{(T_{2r}(Q))}) \) be non-negative such that \( Lu=Lv=0 \) in \( T_{2r}(Q) \) and \( u,v \equiv 0 \) on \( \Delta(Q,2r) \).
    Then
    \[ \frac{u(X)}{v(X)} \approx \frac{u(A_r(Q))}{v(A_r(Q))}, \quad X \in T_r(Q). \]
\end{prop}

And the boundary H\"older estimate also holds.
\begin{prop}[Lemma 3.2.5]\label{prop:BoundaryHolder}
    Let \( u \in W^{1,2}(\Omega) \) be a solution in \( \Omega \) and \( P \in \partial \Omega \). 
    Suppose that \( u \) vanishes on \( \Delta(P,R) \).
    Then for \( 0 < r \leq R \) we have
    \[ \Osc_{T(P,r)} u \lesssim_{\Omega} \left( \frac{r}{R} \right)^{\alpha} 
        \sup_{T(P,R)} |u|. \]
\end{prop}

\begin{rem}\label{rem:ExteriorConeCondition}
    The proof of this result uses the exterior corkscrew condition, 
        i.e., that \( \R^n \setminus  \Bar{\Omega} \) satisfies \refdef{def:CorcscrewCondition} and it is the reason why in the paper we assume that \( \Omega \) is a CAD rather than a 1-sided CAD domain.
 \end{rem}

An important corollary of the result above is the following lemma:
\begin{prop}\label{prop:BoundaryHarnack}
    Let \( u \geq 0 \) be a solution in \( \Omega \) 
        that vanishes on \( \Delta(Q,2r) \).
    Then
    \[ \sup_{T(\Delta(Q,r))} u \lesssim u(A(Q,r)). \]
Here $A(Q,r)$ is a corkscrew point inside $\Omega$ w.r.t $Q$ and $r$ as defined by Definition \ref{def:CorcscrewCondition}.
\end{prop}

This is \emph{Lemma 4.4} of \cite{jerison_boundary_1982},
    the only difference in our setting is that equation (4.5) in \cite{jerison_boundary_1982} follows from \refprop{prop:BoundaryHolder}.
After combining \refprop{prop:BoundaryHarnack} with \refprop{prop:BoundaryHolder} we have:
\begin{prop}\label{prop:BoundaryHoelderContinuity}
    Let \(u\geq 0\) be a solution that vanishes on \(\Delta(Q,R)\). Then there are \(C>0,1>\alpha>0\) such that
    \[\sup_{T(\Delta(Q,r))} u \leq C\left(\frac{r}{R}\right)^\alpha u(A(Q,R)).\]
\end{prop}

\subsection{Properties of the Green's function}

The paper \cite{li_lp_2019} also gives us information on some properties of the Green's function.
\begin{prop}[Theorem 4.1.1]\label{prop:GreenExist}
    There exists a unique function (the Green's function) 
    \( G : \Omega \times \Omega \to \R \cup \{ \infty \} \),
    such that
    \[ G(\cdot,Z) \in W^{1,2}(\Omega \setminus B(Z,r)) \cap W_0^{1,1}(\Omega), 
        \quad Z \in \Omega, \; r > 0, \]
    and
    \begin{align} \int_{\Omega} A(Y) \nabla_{Y} G(Y,Z) \nabla \phi(Y) dY = \phi(Z), 
        \quad \phi \in W_0^{1,p}(\Omega) \cap C^0(\Omega), \quad p>n. \label{eq:DefiningPropertyofGreensfct}\end{align}
\end{prop}

\begin{prop}\label{prop:GreenBounds}
    \[ G(X,Y) \lesssim |X-Y|^{2-n}, \quad X,Y \in \Omega, \]
        and for any \( 0<\theta <1 \) we have
    \[ G(X,Y) \gtrsim_{n,\lambda_0,\Lambda_0,\theta} |X-Y|^{2-n}, \quad X,Y \in \Omega : |X-Y| < \theta\delta(Y). \]
\end{prop}

\begin{prop}\label{prop:GreenAdj}
    Let \( L^* \) be the adjoint operator to \( L \) and let \( G^* \) 
        be its Green's function.
    Then
    \[ G^*(X,Y) = G(Y,X), \quad X,Y \in \Omega. \]
\end{prop}

Finally we have the following relation between the Green's function and the elliptic measure \( \omega^X \)
    which gives us that the elliptic measure must be doubling.
\begin{prop}[Corollary 4.3.1]\label{prop:GreenToOmega}
    \[ \omega^X(\Delta(Q,r)) \approx r^{n-2} G(X,A(Q,r)), \quad X \in \Omega \setminus B(Q,2r). \]
\end{prop}

\begin{prop}[Corollary 4.3.2]\label{prop:DoublingPropertyOfomega}
    \[ \omega^X(\Delta(Q,2r)) \lesssim \omega^X(\Delta(Q,r)), \quad X \in \Omega \setminus B(Q,2r). \]
\end{prop}

As $\delta(X)$ is a continuous function on $\overline{\Omega}$, without loss of generality assume that $0\in\Omega$ and that \(\delta(0)\geq \delta(X)\) for all \(X\in \Omega\). Let  \(\omega^0=\omega\). 
\begin{lemma} Then
    \[ \omega(\Delta(X^*,\delta(X))) \approx \delta(X)^{n-2} G(0,X), \quad X \in \Omega \setminus B(0, \tfrac{1}{2}\delta(0)). \]
\end{lemma}

\textit{\textbf{Proof:}}
Let \( X \in \Omega \setminus B(0, \tfrac{1}{2}\delta(0)) \)
To begin with note that if \(\delta(X) < \frac{1}{2}\delta(0) \), then \(0\notin B(X^*,2\delta(X))\) 
    and hence the result immediately follows from \refprop{prop:GreenToOmega}.
Assume therefore that \( \delta(X) \geq \frac{1}{2}\delta(0) \).
Let \( Z \) be the point given by \( \partial B(X^*,\frac{1}{4}\delta(0)) \cap [X^*,X] \).
Then \( \delta(Z) = |Z-X^*| = \frac{1}{4}\delta(0) \) and we may choose \( Z^* = X^* \). 
Thus \( 0 \notin T(Z^*,2\delta(Z)) \) so  \refprop{prop:GreenToOmega} applies. 
    We get that 
\begin{equation}\label{eq:GreenToOmegaForZ}
    \omega(\Delta(Z^*,\delta(Z))) \approx \delta(Z)^{n-2} G(0,Z) 
\end{equation}

Next as our domain is CAD, there clearly exists a finite Harnack chain, 
from \( X \) to \( Y \) in \( B(X,\delta(X)) \setminus B(0,\delta(0)/4) \) whose length is independent of \( X \).
Thus by \refprop{prop:Harnack} we deduce that
\begin{equation}\label{eq:G(Z)=G(X)}
    G(0,X) \approx G(0,Z).
\end{equation}
Finally we note that since \( \omega \) is doubling and \( 4\delta(Z) = \delta(X) \leq \delta(0) \)
    we have
\begin{equation}\label{eq:OmegaDoubling}
    \omega(\Delta(X^*,\delta(X))) \approx \omega(\Delta(Z^*,\delta(Z))).
\end{equation}
Thus combining \eqref{eq:GreenToOmegaForZ}, \eqref{eq:G(Z)=G(X)} and \eqref{eq:OmegaDoubling}
    yields the desired result.\qed\\

Throughout this work \(G_i\) will denote the Green's function of \(L_i\) for \( i=0,1 \). 
Furthermore, as above, we assume \(0\in \Omega\) and declare this to be the special point that is the \lq\lq center of the domain" \(\Omega\) in the sense that \(\delta(0)=\max\{\delta(X);X\in \Omega\}\). 
We shorten notation and set \(G_0(Y):=G_0(0,Y)\).

\subsection{Nontangential behaviour and the square function in chord arc domains}

Recall that the nontangential maximal function is given by

\[N_\alpha[u](Q)\coloneqq\sup_{X\in\Gamma_\alpha(Q)}|u(X)|, \]

and the mean-valued nontangential maximal function is defined by

\[ \Tilde{N}^\eta_\alpha[u](Q) \coloneqq \sup_{X \in \Gamma_\alpha(Q)} 
    \left( \fint_{B(X,\eta\delta(X)/2)} |u(Z)|^2 dZ \right)^{1/2}. \]

It is immediately clear that
$$
    \Tilde{N}^\eta_\alpha[u](Q) \leq N_{\alpha+\eta/2}[u](Q).
$$
We write \(\Tilde{N}_\alpha[u]=\Tilde{N}^1_\alpha[u]\) and drop the aperture \(\alpha\) when it is clear from the context.\\

\begin{lemma}\label{lemma:NontanMaxFctWithDiffConesComparable}[Remark 7.2 in \cite{milakis_harmonic_2011}]
    Let \( \mu \) be a doubling measure on \( \partial \Omega \), where \( \Omega \) is a NTA domain.
    Let \( v : \Omega \to \R \) and let \( 0 < p < \infty,\alpha,\beta>0,2>\eta>0 \).
    Then 
    \[ \| \tilde{N}_\alpha [v] \|_{L^p(d\mu)} \approx \| \tilde{N}_\beta [v] \|_{L^p(d\mu)} \approx \| \tilde{N}^\eta_\alpha [v] \|_{L^p(d\mu)}, \]
    where the implied constant depends on the character of \( \Omega \), 
        the doubling constant of \( \mu \) and \( \beta/\alpha \).
\end{lemma}

Now for \( L_i = \Div(A_i\nabla), \; i=0,1 \), 
    where \( A_i \) satisfies \eqref{eq:elliptic} and \eqref{eq:A^ainBMO},
    let \( u_i \) be the solution to the Dirichlet problem for \( L_i \) 
    with boundary data \( f \in L^p(\partial\Omega)\cap C(\partial\Omega) \).
We set \(F\coloneqq u_1-u_0\) and note that clearly \( F = 0 \) on \( \partial\Omega \).

\begin{lemma}\label{lemma:NFleqtildeNF}
Let \( 0 < \eta < 2 \).
For any solution \( u \) of an elliptic PDE we have
$$
    N_\alpha[u](Q)\lesssim \tilde{N}^{\eta}_\alpha[u](Q)     
$$
    and furthermore
\begin{equation}\label{eq:NFLeqTNF}
    N_\alpha[F](Q)\lesssim
        \tilde{N}^{\eta}_{\alpha}[F](Q) + \tilde{N}^{\eta}_\alpha[u_0](Q).    
\end{equation}

\end{lemma}

\begin{proof}
First, note that by \refprop{prop:DiGNMestimate} we have, for any solution \( u \):
\[ |u(X)| \leq \sup_{B(X,\frac{\eta}{2}\delta(X)/4)} |u| \lesssim \left( \fint_{B(X,\frac{\eta}{2}\delta(X)/2)} |u|^2 \right)^{1/2}\lesssim \left( \fint_{B(X,\eta\delta(X)/2)} |u|^2 \right)^{1/2}, \]
    and hence 
\[ N_\alpha[u](Q) \lesssim \tilde{N}^{\eta}_\alpha[u](Q). \]
Using this and the triangle inequality we then have
\begin{align*}
    N_\alpha[F](Q) 
    &\leq N_\alpha[u_1](Q) + N_\alpha[u_0](Q) 
    \lesssim \tilde{N}^{\eta}_\alpha[u_1](Q) + \tilde{N}^{\eta}_\alpha[u_0](Q)
    \\
    &\lesssim \tilde{N}^{\eta}_\alpha[F](Q) + \tilde{N}^{\eta}_\alpha[u_0](Q). 
\end{align*}
\end{proof}

We also have an \lq\lq almost Caccioppoli inequality" for \( F \).
\begin{lemma}\label{lemma:CaccioppoliForF}
    \[ \int_{B(X,R)} |\nabla F |^2dZ 
        \lesssim \frac{1}{R^2}\int_{B(X,2R)} (F^2 + u_0^2) dZ.\]
\end{lemma}

To see this one simply uses Caccioppoli for \( u_0 \) and \( u_1 \), and the triangle inequality
\begin{align*}
    \int_{B(X,R)} |\nabla F |^2dZ 
    &\lesssim \int_{B(X,R)} (|\nabla u_1^2| + |\nabla u_0|^2) dZ
    \lesssim \frac{1}{R^2}\int_{B(X,2R)} (u_1^2 + u_0^2) dZ
    \\
    &\lesssim \frac{1}{R^2}\int_{B(X,2R)} (F^2 + u_0^2) dZ.
\end{align*}

As is customary, we can define the square function of a function \(u\in W^{1,2}_{loc}(\Omega)\) by
\[S_\alpha[u](Q):=\left(\int_{\Gamma_\alpha(Q)}|\nabla u(X)|^2\delta(X)^{2-n}dX\right)^{1/2}.\]

If we set \( f \coloneqq |\nabla u|\delta \),  the square function can be considered to be a restriction of a more general operator 
\[A^{(\alpha)}[f](Q):=\left(\int_{\Gamma_\alpha(Q)}\frac{|f|^2}{\delta(X)^{n}}dX\right)^{1/2}.\]

More results on this operator can be found in \cite{milakis_harmonic_2011}, in particular their Proposition 4.5:

\begin{prop}\label{prop:SquareFctWithDiffApertures}
    We have for \(0<p<\infty\) and two apertures \(\alpha,\beta\geq 1\)
    \[\Vert A^{(\alpha)}[f]\Vert_{L^p(d\sigma)}\approx \Vert A^{(\beta)}[f]\Vert_{L^p(d\sigma)}.\]
\end{prop}

This holds for any doubling measure \( \mu \) and in our case it means that  the \(L^p\) norms of square functions for cones of different aperture are comparable.

\subsection{Dyadic decomposition of \(\partial\Omega\) and definition of decomposition \((\partial\Omega,4R_0)\)} \label{subsection:dyadic decomposition properties}

Recall that \( 0\in \Omega \) and set \(R_0=\min(\frac{1}{2^{30}}\delta(0), 1)\). 
As in \cite{milakis_harmonic_2011} we consider a matrix \(A'\) 
    with \(A'=A_1\) on \((\partial\Omega,R_0/2) \coloneqq \{Y\in \Omega : \delta(Y)<R_0/2\} \)
    and \(A'=A_0\) on \(\Omega\setminus (\partial\Omega,2R_0)\). Then the following holds
\begin{lemma}[cf. Lemma 7.5 in \cite{milakis_harmonic_2011}]
    If \(\omega'\) denotes the elliptic measure associated to \(L'=\mathrm{div}(A'\nabla \cdot)\). Then \(\omega_1\in B_p(\omega_0)\) iff \(\omega'\in B_p(\omega_0)\). 
\end{lemma}
Thus without loss of generality we may assume that \(\beta_r(Y)=0\) 
    for \(Y\in \Omega, \delta(Y)> 4R_0\). 
\\

Recall the famous decomposition of M. Christ. 
By \cite{christ_tb_1990} there exists a family of \lq\lq cubes"
    \(\{Q_\alpha^k\subset\partial \Omega; k\in \mathbb{Z},\alpha\in I_k \subset \mathbb{N} \} \) 
    where each scale $k$ decomposes \(\partial \Omega \) such that for every \(k\in \mathbb{Z}\):
\[\sigma\left(\partial \Omega\setminus\bigcup_{\alpha\in I_k}Q_\alpha^k\right)=0
    \qquad\textrm{and}\qquad 
    \omega_0\left(\partial \Omega\setminus\bigcup_{\alpha\in I_k}Q_\alpha^k\right)=0. \]
Furthermore, the following properties hold:
\begin{enumerate}
    \item If \(l\geq k\) then either \(Q_\beta^l\subset Q_\alpha^k\) or \(Q_\beta^l\cap Q_\alpha^k=\emptyset\).
    \item For each \((k,\alpha)\) and \(l<k\) there is a unique \(\beta\) so that \(Q_\alpha^k\subset Q_\beta^l\).
        \item Each \(Q_\alpha^k\) contains a ball \(\Delta(Z_\alpha^k,8^{-k-1})\).
    \item There exists a constant \(C_0>0\) such that \(8^{-k-1}\leq \mathrm{diam}(Q_\alpha^k)\leq C_08^{-k}\).
\end{enumerate}
The last listed property implies together with the Ahlfors regularity of the surface measure that \(\sigma(Q_\alpha^k)\approx 8^{-k(n-1)}\). 
Similarly, the doubling property \refprop{prop:DoublingPropertyOfomega} 
    of the elliptic measure guarantees us \(\omega_0(Q_\alpha^k)\approx \omega_0(B(Z_\alpha^k, 8^{-k-1})).\)
\\

Now we can define a  decomposition of \((\partial\Omega,4R_0) \). 
For \(k\in \mathbb{Z},\alpha\in I_k\), set
\[I_\alpha^k \coloneqq \{Y\in \Omega;  8^{-k-1}<\delta(Y)<8^{-k+1}, 
    \exists P\in Q_\alpha^k:  8^{-k-1}<|P-Y|< 8^{-k+1}\},\]
Note that this collection has finite overlap.
Next choose the scale \( k_0 \) such that
\[(\partial\Omega,4R_0)\subset\bigcup_{\alpha,k\geq k_0}I^k_\alpha.\]
Additionally, for \(\EPS>0\) we set the scale \(k_\EPS\) as the smallest integer 
    such that \(I^k_\alpha\subset (\partial\Omega,\EPS)\) for all \(k\geq k_\EPS\).
The choices of \(k_0\) and \(k_\EPS\) guarantee that
\begin{align} 
    (\partial\Omega,4R_0)\setminus (\partial\Omega,\EPS)\subset \bigcup_{\alpha, k_0\leq k\leq k_\EPS} I_\alpha^k 
        \subset (\partial\Omega,32R_0)\setminus (\partial\Omega,\EPS/8)\label{eq:Prelimk_0leqkleqk_EPS}  
\end{align}

Furthermore, we define the following enlarged decomposition
\[\hat{I}_\alpha^k=\{Y\in \Omega;  8^{-k-2}<\delta(Y)<8^{-k+2}, 
    \exists P\in Q_\alpha^k:  8^{-k-2}<|P-Y|< 8^{-k+2}\}, \]
    and it is clear that that \(I_\alpha^k\subset\hat{I}_\alpha^k\) 
    and that the \(\hat{I}_\alpha^k\) have finite overlap.
Observe for \( l\geq 0 \) that we can cover \(I_\alpha^k\) by balls \(\{B(X_i, 8^{-k-l})\}_{1\leq i\leq N_l}\) 
    with \(X_i\in I_\alpha^k\) such that \(|X_i-X_l|\geq 8^{-k-3}/2\). 
Note here that $N_l$ is independent of $k$ and \(\alpha\). 
Furthermore, we have for each \(Z\in B(X_i,2 8^{-k-3})\) that 

\begin{align}\label{eq:Prelimdelta(Z)}
     8^{-k+2}>\delta(Z)>  8^{-k-1}-2 8^{-k-3}=\frac{62}{8} 8^{-k-2},
\end{align}

  and hence

\begin{align}\label{eq:PrelimB(Xi)SubsetB(Z)}
    B(X_i, 8^{-k-3})\subset B(Z,3 8^{-k-3})\subset B(Z,\frac{62}{8} 8^{-k-3})
    \subset B(Z,\delta(Z)/8),
\end{align}

\begin{align}\label{eq:PrelimSameSizeOfIalphaB(Z)}
    |B(X_i, 8^{-k-3})|\approx |I_\alpha^k|\approx B(Z,\delta(Z)/2),
\end{align}

and

\begin{align}\label{eq:PrelimIalphaSubsetB(Xi)SubsethatI}
    I_\alpha^k\subset\bigcup_{i=1}^NB(X_i, 8^{-k-3})\subset \bigcup_{i=1}^NB(X_i,2 8^{-k-3})
    \subset \hat{I}_\alpha^k.
\end{align}

Note also that for \(Z\in (\partial\Omega,4R_0)\) there exists \(Z^*\in\partial \Omega\) 
    with \(|Z^*-Z|=\delta(Z)\) and an \(I_\alpha^k\) such that \(Z^*\in Q^k_\alpha\) and
\[2 8^{-k-1}\leq\delta(Z)\leq\frac{1}{2}8^{k+1}.\] 
Thus, for every \(X\in B(Z,\delta(Z)/4)\)
\[ 8^{-k-1}\leq \delta(X)\leq  8^{-k+1}\qquad\textrm{and}\qquad 8^{-k-1}\leq |Z^*-X| \leq  8^{k+1}. \]
Hence
\begin{align}\label{eq:PrelimB(Z,delta(Z)/4)SubsetIalphak}
    B(Z,\delta(Z)/4)\subset I_\alpha^k.
\end{align}

Next, we note that
\begin{align}\label{eq:PrelimISubsetCarlesonregion}
    \hat{I}_\alpha^k\subset T(\Delta(Z_\alpha^k, C_0+16)8^{-k}).
\end{align}

Lastly, we also observe that if \(P\in Q_\alpha^k\) and \(Z\in \hat{I}_\alpha^k\) then, 
    for \(M=8^4\),

\begin{align}\label{eq:PrelimIalphaSubsetGammaM}
    Z\in \Gamma_M(P).
\end{align}

We are also going to need an intermediate decomposition in cubes \( \tilde{I}_\alpha^k\) defined by
\begin{align*}
    \tilde{I}_\alpha^k:= \bigg\{Y\in \Omega: & \frac{14}{16}8^{-k-1}<\delta(Y)< \frac{18}{16}8^{-k+1}, 
    \\
        &\exists P\in Q_\alpha^k:  \frac{14}{16}8^{-k-1}<|P-Y|< \frac{18}{16}8^{-k+1}\bigg\}.    
\end{align*}

The enlargement compared to \(I_\alpha^k\) here is chosen such that

\begin{align}\label{eq:PrelimIalphasubsettildeIalphasubsethatIalpha}
    I_\alpha^k\subset\bigcup_{i=1}^NB(X_i, 8^{-k-3})\subset  \bigcup_{i=1}^NB(X_i,2 8^{-k-3})\subset \tilde{I}_\alpha^k\subset \bigcup_{i=1}^NB(X_i,4 8^{-k-3})
    \subset \hat{I}_\alpha^k.
\end{align}

All of this will be used to prove the following important proposition:

\begin{prop}\label{prop:EPSleqEPS_0}
    With $I_\alpha^k$, $\tilde I_\alpha^k$ as above we have:
    \[ \left(\fint_{I_\alpha^k}|\EPS(Y)|^rdY\right)^{1/r} \lesssim \frac{1}{\omega_0(Q_\alpha^k)}\left(\int_{\hat{I}^k_\alpha}\delta(Z)^{-2}G_0(Z)\beta_r(Z)^2dZ\right)^{1/2}.\]
   In particular, if \(\Vert \frac{\beta_r(Z)^2G_0(Z)}{\delta(Z)^2}\Vert_\mathcal{C}\leq \EPS_0^2 \), then
    \[\left(\fint_{I_\alpha^k}|\EPS(Y)|^rdY\right)^{1/r}\lesssim \EPS_0.\]
\end{prop}

Proof: Using the covering of $I_\alpha^k$ by balls defined above we have

\begin{align*}
     \left(\fint_{I_\alpha^k}|\EPS(Y)|^rdY\right)^{1/r}
     &\lesssim \left(\sum_{i=1}^N  \fint_{B(X_i, 8^{-k-3})}|\EPS(Y)|^rdY\right)^{1/r}
     \\
     &= \sum_{i=1}^N \fint_{B(X_i, 8^{-k-3})}\left(\fint_{B(X_i, 8^{-k-3})}|\EPS(Y)|^rdY\right)^{1/r}dZ
     \\
     &\lesssim \sum_{i=1}^N \fint_{B(X_i, 8^{-k-3})}\left(\fint_{B(Z,\delta(Z)/2)}|\EPS(Y)|^rdY\right)^{1/r}dZ
     \\
     &\leq \sum_{i=1}^N \fint_{B(X_i, 8^{-k-3})}\beta_r(Z)dZ
\\     &\leq \sum_{i=1}^N \left(\fint_{B(X_i, 8^{-k-3})}\beta_r(Z)^2dZ\right)^{1/2}
     \\
     &\lesssim \sum_{i=1}^N \left(\int_{B(X_i, 8^{-k-3})}\delta(Z)^{-n}\beta_r(Z)^2dZ\right)^{1/2}
     \\
     &\lesssim \sum_{i=1}^N \left(\frac{1}{\omega_0(Q_\alpha^k)} 
        \int_{B(X_i, 8^{-k-3})}\delta(Z)^{-2}G_0(Z)\beta_r(Z)^2dZ\right)^{1/2}
     \\
     &\lesssim \left(\frac{1}{\omega_0(Q_\alpha^k)} \int_{\hat{I}^k_\alpha}\delta(Z)^{-2}G_0(Z)\beta_r(Z)^2dZ\right)^{1/2}.
\end{align*}

Furthermore, for the second part under the assumption that
\(\Vert
\frac{\beta_r(Z)^2G_0(Z)}{\delta(Z)^2}\Vert_\mathcal{C}\leq \EPS_0^2 \), we have that
\begin{align*}
    &\left(\frac{1}{\omega_0(Q_\alpha^k)} \int_{\hat{I}^k_\alpha}\delta(Z)^{-2}G_0(Z)\beta_r(Z)^2dZ\right)^{1/2}
    \\
     &\qquad\lesssim \left(\frac{1}{\omega_0(\Delta(Z_\alpha^k, C_0+16)8^{-k})}
        \int_{T(\Delta(Z_\alpha^k, (C_0+16)8^{-k}))}\frac{\beta_r(Z)^2G_0(Z)}{\delta(Z)^{2}}dZ\right)^{1/2}
    \\
    &\qquad\lesssim \EPS_0.
\end{align*}

\section{Proof of \refthm{thm:NormBig}} \label{S:LargeNormProof}

We use a method inspired by \cite{milakis_harmonic_2011}. 
Recall that \(A_0\) and \(A_1\) satisfy \eqref{eq:elliptic} and \eqref{eq:A^ainBMO}, 
    and that \(\omega_0\) and \(\omega_1\) denote their elliptic measures. 
Furthermore,  \(F=u_1-u_0\) and \(\beta_r\) are as in \eqref{def:beta_r}.
First, we need to prove

\begin{thm}\label{thm:theorem2.3}
Let \(\Omega\) be a bounded CAD. 
There exists \(\EPS_0=\EPS_0(n,\lambda_0,\Lambda_0)>0\) and \(r>0\), such that if
\[\sup_{\Delta\subset \partial\Omega} \left(\frac{1}{\omega_0(\Delta)}\int_{T(\Delta)}\beta_r^2(X)\frac{G_0(X)}{\delta(X)^2}dX\right)^{1/2}<\EPS_0,\]
then \(\omega_1\in B_2(\omega_0)\).
\end{thm}

This theorem corresponds to Theorem 2.9 in \cite{milakis_harmonic_2011} but
    with the discrepancy function \(\alpha\) instead of \(\beta_r\). 
Using this the authors of \cite{milakis_harmonic_2011} first prove

\begin{thm}[Theorem 8.2]\label{Thm:8.2}
Let \(\Omega\) be a bounded CAD and let
\[A(a)(Q)=\left(\int_{\Gamma(Q)}\frac{\alpha^2(X)}{\delta(X)^n}dX\right)^{1/2}.\]
If \(\Vert A(a) \Vert_{L^{\infty}}\leq C<\infty\) and \(\omega_0\in A_\infty(\sigma)\) then \(\omega_1\in A_\infty(\sigma)\).
\end{thm}

After that, they establish \refthm{thm:NormBig} (which is Theorem 8.1 in their notation). 
If we replace \(\alpha\) by \(\beta_r\), 
    we can conclude \refthm{thm:NormBig} from \refthm{Thm:8.2} 
    and \refthm{Thm:8.2} from \refthm{thm:theorem2.3} in the same way as in \cite{milakis_harmonic_2011}. 
The only modification  is the substitution of their discrepancy function \(\alpha\) by \(\beta_r\).
\\

Hence it remains to prove \refthm{thm:theorem2.3}. We first establish the following two lemmas.

\begin{lemma}\label{lem:2.9}
    Let \( \mu \) be a doubling measure on \( \partial\Omega \). 
    Under the assumptions of \refthm{thm:theorem2.3} we have for every \(0<p<\infty\)
    \[ \int_{\partial\Omega}\Tilde{N}_{\alpha}[F]^pd\mu 
    \lesssim_{p} \EPS_0\int_{\partial \Omega} M_{\omega_0}[S_{\bar{M}}u_1]^pd\mu, \]
    where the aperture \( \bar{M}=2^8 \) is at least twice as large as \( \alpha \).
\end{lemma}

In particular this lemma holds with \( \mu \in \{ \omega_0,\sigma \} \).

\begin{lemma}\label{lem:2.10}
Under the assumptions of \refthm{thm:theorem2.3} we have for any aperture \(\alpha>0\)
    \[ \int_{\partial\Omega} S_{\bar{M}}[F]^2 d\omega_0 \lesssim 
        \int_{\partial\Omega} \left( \Tilde{N}_{\alpha}[F]^2 + f^2 \right) d\omega_0. \]
\end{lemma}

These lemmas are versions of Lemma 2.9 and Lemma 2.10 in \cite{fefferman_theory_1991} or Lemma 7.7 and Lemma 7.8 in \cite{milakis_harmonic_2011}.\\

\textit{Proof of }\refthm{thm:theorem2.3}: Assume that \reflemma{lem:2.9} and \reflemma{lem:2.10} hold. Since
\[ \| S_{\alpha}[u_0] \|_{L^2(d\omega_0)} \lesssim \| f \|_{L^2(d\omega_0)}, \]
we have that
\begin{align*}
    \int_{\partial\Omega} \tilde{N}_{\alpha}[F]^2 d\omega_0
    &\lesssim \int_{\partial\Omega} \EPS_0^2 M_{\omega_0}[S_{\bar{M}}u_1]^2 d\omega_0
    \lesssim \EPS_0^2 \int_{\partial\Omega}  S_{\bar{M}}[u_1]^2 d\omega_0
    \\
    & \leq \EPS_0^2 \int_{\partial\Omega}  S_{\bar{M}}[F]^2 d\omega_0 + \EPS_0^2 \int_{\partial\Omega} S_{\bar{M}}[u_0]^2 d\omega_0 
    \\
    &\lesssim \EPS_0^2 \int_{\partial\Omega} S_{\bar{M}}[F]^2 d\omega_0 + \EPS_0^2 \int_{\partial\Omega} f^2 d\omega_0
    \\
    &\lesssim \EPS_0^2\int_{\partial\Omega}   \Tilde{N}_{\alpha}[F]^2  d\omega_0
        + \EPS_0^2 \int_{\partial\Omega} f^2 d\omega_0.
\end{align*}
Thus, with \( \EPS_0 \) small enough we can hide the term \( \Tilde{N}_\alpha[F]^2 \)  and absorb it by the lefthand side. We get
\[ \| \tilde{N}_{\alpha}[F] \|_{L^2(d\omega_0)} \lesssim \| f \|_{L^2(d\omega_0)}. \]

By \reflemma{lemma:NFleqtildeNF} we obtain
\begin{align*}
    \int_{\partial\Omega} N_{\alpha}[u_1]^2 d\omega_0\lesssim \int_{\partial\Omega} \tilde{N}_{\alpha}[F]^2 d\omega_0+\int_{\partial\Omega} \tilde{N}_{\alpha}[u_0]^2 d\omega_0\lesssim \int_{\partial\Omega} f^2 d\omega_0,    
\end{align*}

 Theorefore \( \omega_1 \in B_2(\omega_0) \) as desired. \qed\\

\subsection{The difference function F}
Our first proposition is a generalization of Lemma 3.12 in \cite{cavero_perturbations_2019}
   using the same strategy for its proof.

\begin{prop}\label{prop:DefF}
    Suppose that \(\Omega\subset \R^{n}\) is a bounded CAD 
        and let \(L_0,L_1\) be two elliptic operators. 
    Let \(u_0\in W^{1,2}(\Omega)\) be a weak solution of \(L_0u_0=0\) in \(\Omega\), 
        and let \(G_1\) be the Green's function of \(L_1\). Then
\[\int_\Omega A_0(Y)\nabla_Y G_1(Y,X) \cdot \nabla u_0(Y)dY=0
    \qquad \textrm{for a.e. } X\in \Omega.\]
\end{prop}

\begin{proof}

We begin by fixing a point \(X_0\in \Omega\) and considering a cut-off function \(\varphi\in C_c([-2,2])\) 
    such that \(0\leq \varphi\leq 1\) and \(\varphi\equiv 1\) on \([-1,1]\). 
For each \(0<\EPS<\delta(X_0)/16\) we set \(\varphi_\EPS(X)=\varphi(|X-X_0|/\EPS)\) 
    and \(\psi_\EPS=1-\varphi_\EPS\). 
Furthermore let \( G_1^{X_0} = G_1(Y,X_0) \).
We see that
\[ \int_\Omega  A_0 \nabla G_1^{X_0} \cdot \nabla u_0 dY 
    = \int_\Omega A_0\nabla (G_1^{X_0} \psi_\EPS) \cdot \nabla u_0 dY
    + \int_\Omega A_0\nabla (G_1^{X_0} \varphi_\EPS) \cdot \nabla u_0dY.\]

Thanks to \refprop{prop:GreenExist} we have that \(G_1(\cdot,X_0)\psi_\EPS\in W_0^{1,2}(\Omega)\), 
    which gives us (as $L_0u_0=0$) that
    \[\int_\Omega A_0(Y)\nabla (G_1(Y,X_0)\psi_\EPS(Y)) \cdot \nabla u_0(Y)dY=0.\]
Next we note that 
\begin{align*}
    \int_\Omega A_0 \nabla (G_1^{X_0} \varphi_\EPS) \cdot \nabla u_0dY
    &=\int_\Omega A_0\nabla G_1^{X_0} \cdot \nabla u_0 \varphi_\EPS dY
        +\int_\Omega A_0 \nabla \varphi_\EPS\cdot G_1^{X_0} \nabla u_0dY
    \\
    &\eqqcolon I_\EPS(X_0) +II_\EPS(X_0).
\end{align*}
Thus if we can show that \( I_\EPS(X_0) + II_\EPS^k(X_0) \to \) 0 as \( \EPS \to 0 \) 
    for a.e. \( X_0 \in \Omega \), then we have shown our claim.
We start by considering the first term. Clearly,
\begin{align*}
    |I_\EPS(X_0)|&\leq \int_{B(X_0,2\EPS)}|A_0| |\nabla G_1^{X_0}| |\nabla u_0| dY
    \\
    &\leq \left(\int_{B(X_0,\delta(X_0/8))}|A_0|^r dY\right)^{1/r}\left(\int_{B(X_0,2\EPS)}
        \left(|\nabla G_1^{X_0}||\nabla u_0|\right)^{r'}dY\right)^{1/r'},
\end{align*}
    for some \( r>2 \) to be determined later.
Notice that the first term is bounded since \(A\in L^r_{loc}(\Omega)\).
To deal with the second term we decompose the ball \( B(X_0,2\EPS) \) into family of annuli
    \(C_j(X_0,\EPS)=B(X_0,2^{-j+1}\EPS)\setminus B(X_0,2^{-j}\EPS), j\geq 0\).
This gives us

\begin{align*}
    &|I_\EPS(X_0)|\lesssim
        \left(\int_{B(X_0,2\EPS)}
        \left(|\nabla G_1^{X_0}||\nabla u_0|\right)^{r'}dY\right)^{1/r'}
    \\
    &\quad\leq \left(\sum_{j=0}^\infty (2^{-j}\EPS)^n\fint_{C_j(X_0,\EPS)}
            \left(|\nabla G_1^{X_0}||\nabla u_0|\right)^{r'}dY
        \right)^{1/r'}
    \\
    &\quad\lesssim \sum_{j=0}^\infty (2^{-j}\EPS)^{n/r'}
        \left(\fint_{C_j(X_0,\EPS)}|\nabla G_1^{X_0}|^2dY\right)^{1/2}
        \left(\fint_{B(X_0,2^{-j})}|\nabla u_0|^{\tfrac{2r}{r-2}}dY\right)^{\tfrac{r-2}{2r}}.
\end{align*}

Using \refprop{prop:GradRevHol}, Caccioppoli's inequality and \refprop{prop:GreenBounds} 
    we get that for \(r\) sufficiently large we have:
\begin{align*}
    \left(\fint_{B(X_0,2^{-j})}|\nabla u_0|^{\frac{2r}{r-2}}dY\right)^{\frac{r-2}{2r}}
    &\lesssim \left(\fint_{B(X_0,2^{-j+1})}|\nabla u_0|^{2}dY\right)^{1/2}
    \\
    &\leq M[|\nabla u_0|^2\chi_\Omega](X_0)^{1/2}
\end{align*}
    and
\begin{align*}
    \left(\fint_{C_j(X_0,\EPS)}|\nabla G_1^{X_0}|^2dY\right)^{1/2}
    &\leq\frac{1}{2^{-j}\EPS}\left(\fint_{\bigcup_{l=j-1}^{j+1}C_l(X_0,\EPS)} |G_1^{X_0}|^2dY\right)^{1/2}
    \\
    &\lesssim \frac{1}{2^{-j}\EPS}\left(\fint_{\bigcup_{l=j-1}^{j+1}C_l(X_0,\EPS)}
        (2^{-j}\EPS)^{2(2-n)}dY\right)^{1/2}
    \\
    &= (2^{-j}\EPS)^{1-n}.
\end{align*}
Hence 
\begin{align*}
    |I_\EPS(X_0)|
    &\lesssim \sum_{j=0}^\infty (2^{-j}\EPS)^{1-(n-n/r')}M[|\nabla u_0|^2\chi_\Omega](X_0)^{1/2}
    \\
    &= \sum_{j=0}^\infty (2^{-j}\EPS)^{1-n/r}M[|\nabla u_0|^2\chi_\Omega](X_0)^{1/2}.
\end{align*}
Choosing \( r > 2n \) we get  that
\[|I_\EPS(X_0)|\lesssim \sum_{j=0}^\infty 2^{-j/2}\sqrt{\EPS}M[|\nabla u_0|^2\chi_\Omega](X_0)^{1/2}
    \lesssim \sqrt{\EPS}M[|\nabla u_0|^2\chi_\Omega](X_0)^{1/2}.\]
    
For the second term we note that \(\Vert\nabla \varphi_\EPS\Vert_\infty\approx \EPS^{-1}\). Then \refprop{prop:GreenBounds} and H\"older's inequality give us:
    \begin{align*}
        |II_\EPS(X_0)|&\leq \EPS^{-1}\int_{B(X_0,2\EPS)} |A_0| |\nabla u_0| |G_1^{X_0}| dY
        \\
        &\lesssim \EPS^{-n+1} \int_{B(X_0,2\EPS)} |A_0| |\nabla u_0| dY
\end{align*}
\begin{align*}
        &\lesssim \EPS^{1-n/r} \left(\int_{B(X_0,\delta(X_0)/8)} |A_0|^r\right)^{1/r} 
            \left(\fint_{B(X_0,2\EPS)}|\nabla u_0|^{r'} dY\right)^{1/r'}
        \\
        &\lesssim \sqrt{\EPS} \left(\fint_{B(X_0,2\EPS)}|\nabla u_0|^{2} dY\right)^{1/2}
        \\
        &\leq \sqrt{\EPS} M[|\nabla u_0|^2\chi_\Omega](X_0)^{1/2}.
    \end{align*}
    
Combining both integrals we have 
    \begin{align*}
        \int_\Omega A_0\nabla (G_1^{X_0}\varphi_\EPS) \cdot \nabla u_0dY\leq |I_\EPS(X_0)|+|II_\EPS(X_0)|\lesssim \sqrt{\EPS}M[|\nabla u_0|^2\chi_\Omega](X_0)^{1/2}
    \end{align*}
    for all \(\EPS>0\). Since \(\nabla u_0\in L^2(\Omega)\) 
    we have \(M[|\nabla u_0|^2\chi_\Omega]\in L^{1,\infty}(\Omega)\) 
    and thus \(M[|\nabla u_0|^2\chi_\Omega] <\infty\) a.e. on \( \Omega \). 
Thus letting \(\EPS\to 0+\) finishes the proof.
    
\end{proof}

Next we prove a result similar to Lemma 3.18 in \cite{cavero_perturbations_2019}.
However, we note that the proof of this result is not actually given in \cite{cavero_perturbations_2019},
    the paper instead cites \cite{workinprogress} 
    which was not available to us as it has not yet appeared anywhere.

\begin{prop}\label{prop:DefF_(u_0-u_1)}
    Suppose that \(\Omega\subset \R^{n}\) is a bounded CAD. 
    Let \(L_0,L_1\) be two elliptic operators, 
        \(u_0,u_1\in W^{1,2}(\Omega)\) be a pair of weak solutions 
        of \(L_0u_0=0,\,L_1u_1=0\) in \(\Omega\), with \(u_0-u_1\in W_0^{1,2}(\Omega)\)
        and \(G_0\) be the Green's function of \(L_0\). Then for a.e. \( X\in \Omega\) we have
    \[F(X) \coloneqq u_0(X)-u_1(X)
        =\int_\Omega A_0(Y)\nabla_Y G_0(Y,X) \cdot \nabla (u_0-u_1)(Y)dY.\]
\end{prop}

\begin{proof}

Fix \(X_0\in \Omega\) and consider a cut-off function \(\vartheta\in C_c([-2,2])\) 
    such that \(0\leq \vartheta\leq 1\) and \(\vartheta\equiv 1\) on \([-1,1]\). 
For each \(0<\EPS<\delta(X_0)/16\) we set \(\vartheta_\EPS(X)=\vartheta(|X-X_0|/\EPS)\) 
    and \(\psi_\EPS=1-\vartheta_\EPS\). 
Consider a sequence of functions \(\varphi_k\in C_0^\infty(\Omega)\)     
    such that \( \varphi_k \to u_0-u_1\) in \(W^{1,2}(\Omega)\) with
\begin{equation}\label{eq:SpeedOfConvergence}
    \Vert\varphi_{k+1}-( u_0-u_1)\Vert_{W^{1,2}(\Omega)}
    \leq \frac{1}{2}\Vert\varphi_k-(u_0-u_1)\Vert_{W^{1,2}(\Omega)}, \quad \forall k\geq 1.     
\end{equation}

By \refprop{prop:GreenExist}
\[ \varphi_k(X_0) = \int_\Omega A_0(Y)\nabla_Y G_0(Y,X_0) \cdot \nabla \varphi_k(Y)dY. \]
Notice that, by taking a subsequence, we may assume that \( \varphi_k \to u_0-u_1\) a.e.
It follows that
\begin{align*}
    & \int_\Omega A_0(Y) \nabla_Y G_0(Y,X_0) \nabla (\varphi_k-(u_0-u_1))(Y) dY
    \\
    &=\int_\Omega A_0(Y) \nabla_Y (G_0(Y,X_0)\psi_\EPS(Y)) \nabla (\varphi_k-(u_0-u_1))(Y)dY
    \\
    &\qquad +\int_\Omega A_0(Y) \nabla_Y G_0(Y,X_0) \nabla (\varphi_k-(u_0-u_1))(Y)\vartheta_\EPS(Y) dY
    \\
    &\qquad +\int_\Omega A_0(Y) \nabla \vartheta_\EPS(Y) \nabla (\varphi_k-(u_0-u_1))(Y)G_0(Y,X_0)dY
    \\
    &\eqqcolon I_\EPS^k(X_0) +II_\EPS^k(X_0) +III_\EPS^k(X_0).
\end{align*}

We aim to show that 
\begin{equation}\label{eq:I+II+III}
     I_\EPS^{k}(X_0) + II_\EPS^k(X_0) + III_\EPS^k(X_0) \to 0, \quad \EPS \to 0,
\end{equation}
for a well chosen sequence $k=k(\varepsilon)$ with property that $k(\varepsilon)\to\infty$ as $\EPS \to 0$.
\\

Analogous to the proof of \refprop{prop:DefF} we get that for a large \(k\)
\[|I_\EPS^k|
    \lesssim \Vert G_0(\cdot,X_0)\psi_\EPS\Vert_{W^{1,2}(\Omega)}     
        \Vert\varphi_k-(u_0-u_1)\Vert_{W^{1,2}(\Omega)} 
    \leq \sqrt{\EPS}, \]
and
\[|II_\EPS^k|,|II_\EPS^k|\lesssim \sqrt{\EPS} M[g_k](X_0)^{1/2},\]
    where \( g_k \coloneqq |\nabla(\varphi_k-(u_0-u_1))|^2\chi_\Omega \).
Thus
\begin{equation}\label{eq:3parts}
    |I_\EPS^k|+|II_\EPS^k|+|III_\EPS^k|
    \lesssim \sqrt{\EPS} \max\left\{1, \sup_{k} g_k \right\}
\end{equation}
Now, since \( g_k \in L^1(\R^n) \)
    we have that  \(M[g_k]\in L^{1,\infty}(\R^n)\) with the bound
\begin{align*}
    \| M[g_k] \|_{L^{1,\infty}(\R^n)}
    \lesssim \Vert g_k \Vert_{L^1(\Omega)} 
    \leq \Vert\varphi_k-(u_0-u_1)\Vert_{W^{1,2}(\Omega)}^2 \lesssim 2^{-k},
\end{align*}
Thus \( \sup_{k} M[g_k]^{1/2} <\infty \) a.e. allowing us to let \( \EPS \to 0 \) in \eqref{eq:3parts} and yielding \eqref{eq:I+II+III} for a.e. $X_0$ as desired.
\end{proof}

\section{Proof of \reflemma{lem:2.9}} \label{S:Lemma2.9Proof}

Let \(Q\in \partial \Omega\) and \(X\in \Gamma_\alpha(Q)\). 
Let \( G_0 \) be the Green's function corresponding to \( L_0 \)
    and \( G_0^* \) its adjoint.
As before let \( F = u_1 - u_0 \) and note that by 
    \refprop{prop:DefF_(u_0-u_1)}, \refprop{prop:DefF} and \refprop{prop:GreenAdj} we have
\begin{align*}
    F(X) 
    &= u_1(X) - u_0(X) 
    = \int_{\Omega} A_0^T \nabla_Y G_0^*(Y,X) \nabla(u_1-u_0)(Y) dY
    \\
    &= \int_{\Omega} A_0 \nabla u_1(Y) \nabla_Y G_0(X,Y) dY
        - \int_{\Omega} A_0 \nabla u_0(Y) \nabla_Y G_0(X,Y) dY
    \\
    &= \int_{\Omega} A_0 \nabla u_1(Y) \nabla_Y G_0(X,Y) dY 
        - \int_{\Omega} A_1 \nabla u_1(Y) \nabla_Y G_0(X,Y) dY
    \\
    &= \int_{\Omega} \EPS \nabla u_1(Y) \nabla_Y G_0(X,Y) dY.
\end{align*}

\begin{rem}\label{rem:GreenFunctionProperty} Note that the Green's function property \eqref{eq:DefiningPropertyofGreensfct} only holds for \( \varphi \in W_{0}^{1,p}(\Omega) \)
    with \( p > n \geq 2 \) and so cannot be applied directly unless we have shown statements such as \refprop{prop:DefF} and \refprop{prop:DefF_(u_0-u_1)}. This is true even if  \( A_0,A_1 \) are bounded and symmetric. This was perhaps overlooked in  \cite{milakis_harmonic_2011} and \cite{fefferman_theory_1991} but can be fortunately fixed via an approximation argument similar to the one we have given in the previous section.
\end{rem}

We split \( F \) into two terms (first of which is the near part close to $X$).
\[ F = F_1 + F_2, \quad F_1(Z) = \int_{B(X)} \nabla_{Y} G_0(Z,Y)\cdot \EPS(Y) \nabla u_1(Y) dY, \]
    and then split \( F_1 \) further and write it as
\[ F_1 = \Tilde{F}_1 + \Tilde{\Tilde{F}}_1, \quad 
    \Tilde{F}_1(Z) = \int_{B(X)} \nabla_Y \Tilde{G}_0(Z,Y) \cdot \EPS(Y) \nabla u_1(Y) dY. \]
Here 
\(B(X):=B(X,\delta(X)/4)\) and \( \Tilde{G}_0 \) denotes the \lq\lq local Green function" for \( L_0 \) on \( 2B(X) \). 
We also set \( K(Z,Y) \coloneqq G_0(Z,Y) - \Tilde{G}_0(Z,Y) \).
Since \( \mu \) is a doubling measure we have by
    \reflemma{lemma:NontanMaxFctWithDiffConesComparable} that
\begin{align*}
\Vert \tilde{N}_\alpha[F]\Vert_{L^p(d\mu)}
    &\leq \Vert \tilde{N}_\alpha[\tilde{F_1}]\Vert_{L^p(d\mu)}
        +\Vert \tilde{N}_\alpha[\tilde{\tilde{F}}_1]\Vert_{L^p(d\mu)}
        +\Vert \tilde{N}_\alpha[F_2]\Vert_{L^p(d\mu)}
\\
&\lesssim \Vert \tilde{N}_\alpha^{1/2}[\tilde{F_1}]\Vert_{L^p(d\mu)}
    +\Vert \tilde{N}_\alpha^{1/2}[\tilde{\tilde{F}}_1]\Vert_{L^p(d\mu)}
    +\Vert \tilde{N}_\alpha^{1/4}[F_2]\Vert_{L^p(d\mu)}.
\end{align*}

Hence to conclude that \reflemma{lem:2.9} holds
it is enough to show that the pointwise bound
\[ \left(\fint_{B(X)} |\tilde{F}_1|^2\right)^{1/2}
        +\left(\fint_{B(X)} |\tilde{\tilde{F}}_1|^2\right)^{1/2}
        +\left(\fint_{B(X,\delta(X)/8)} |F_2|^2\right)^{1/2}
    \lesssim \EPS_0 S_{\bar{M}}[u_1](Q) \]
is true for almost every \(Q\in \partial \Omega\) and \(X\in\Gamma_\alpha(Q)\). We shall consider each of the terms above separately.

\subsection{The \lq\lq local term" \( F_1\)}\label{subsection:Lemma2.9TildeF1}
Consider \( \phi \in C_c^\infty(\mathbb{B}^n) \) non-negative with \( \int_{\R^n} \phi = 1 \)
    and set \( \phi_m = (\frac{2m}{\delta(X)})^{n}\phi(2mX/\delta(X)) \).
Define 
\[ \hat{A}^m \coloneqq A_{1} * \phi_m, \quad \hat{\EPS}^m \coloneqq \EPS * \phi_m, 
    \quad \hat{L}_m \coloneqq \Div(\hat{A}^m \nabla), \]
    and let \( r >1 \) be large enough that \refprop{prop:GradRevHol} holds with \( p = \frac{2r}{r-2} \). 
Let \( \hat{u}_m \) be the weak solution to the Dirichlet problem
\[ \hat{L}_m v = 0, \text{ in } 2B(X), \quad u_1-v \in W_0^{1,2}(2B(X)).  \]
We know that \( A_{0},A_{1} \in L^{r}(2B(X)) \), wherefore
\[ \hat{A}^m \to A_{1}, \quad \hat{\EPS}^m \to \EPS, \quad \text{in }  L^{r}(2B(X)). \]
Moreover \( A_{1}^s \in L^{\infty}(2B(X)) \), 
    hence by the ellipticity of our original $A_1$ for large  \( m \) there exists \( \lambda_{0,m} > 0 \) such that
\[ \lambda_{0,m} |\xi|^2 \leq \xi^T \hat{A}^m \xi \leq  \lambda_{0,m}^{-1} |\xi|^2, 
    \quad \forall \xi \in \R^n, \; a.e. \; X \in 2B(X), \]
    with \( \lambda_{0,m} \to \lambda_{0} \) as \( m \to \infty \) and so the ellipticity constant does not depend on $m$. Next we will show that
\begin{align}\label{e42}
    \| \nabla(\hat{u}_m - u_1) \|_{L^2(2B(X))} \lesssim \| \hat{A}^m - A_1 \|_{L^r(2B(X))}^2,     
\end{align}
    and in particular it follows that \( \nabla \hat{u}_m \to \nabla u_1 \) in \( L^2(2B(X)) \).
Clearly for \( m \) large enough  we have that
\begin{align*}
    \int_{2B(X)}|\nabla(\hat{u}_m - u_1)|^2
    &\lesssim 2\lambda_0^{-1} \int_{2B(X)}\hat{A}^m\nabla(\hat{u}_m - u_1)\cdot \nabla(\hat{u}_m - u_1)
    \\
    &= - 2\lambda_0^{-1} \int_{2B(X)}(\hat{A}^m)\nabla u_1\cdot \nabla(\hat{u}_m - u_1)
    \\
    &= 2\lambda_0^{-1} \int_{2B(X)}(A_1 - \hat{A}^m)\nabla u_1\cdot \nabla(\hat{u}_m - u_1)
    \\
    &\leq 2\lambda_0^{-1} \| \hat{A}^m - A_1 \|_{L^r(2B(X))}
        \| \nabla u_1 \|_{L^{\frac{2r}{r-2}}(2B(X))}
        \\
        &\hspace{3em} \cdot \left(\int_{2B(X)} |\nabla \hat{u}_m-\nabla u_1|^2\right)^{1/2}
    \\[1 ex]
    &\leq 2\lambda_0^{-2} \| \hat{A}^m - A_1 \|_{L^r(2B(X))}^2
        \| \nabla u_1 \|_{L^{\frac{2r}{r-2}}(2B(X))}^2
        \\
        & \hspace{3em} + \frac{1}{2} \left(\int_{2B(X)} |\nabla \hat{u}_m-\nabla u_1|^2\right)
    \\[1 ex]
    & \lesssim \| \hat{A}^m - A_1 \|_{L^r(2B(X))}^2 \| \nabla u_1 \|_{L^{\frac{2r}{r-2}}(2B(X))}^2.
\end{align*}
By \refprop{prop:GradRevHol}  we have that
\[ \| \nabla u_1 \|_{L^{\frac{2r}{r-2}}(2B(X))}^2 \lesssim_X \| \nabla u_1 \|_{L^{2}(3B(X))}^2 
    \leq \| \nabla u_1 \|_{L^{2}(\Omega)}^2 < \infty, \]
    and hence \eqref{e42} follows.\\

The reason we consider approximations of \(\tilde{F}_1\), namely
\[\hat{F}_m(Z) \coloneqq \int_{B(X)} \nabla_Y \tilde{G}_0(Z,Y) \cdot \hat{\EPS}^m(Y) \nabla \hat{u}_m(Y) dY,\]
and
\begin{align}
    \hat{F}_m^{\rho}(Z) &\coloneqq \int_{B(X)} \nabla_Y \tilde{G}^{\rho}_0(Z,Y) 
        \cdot \hat{\EPS}^m(Y) \nabla \hat{u}_m(Y) dY\nonumber
    \\  
    &= - \int_{B(X)} \Div (\hat{\EPS}^m \nabla \hat{u}_m) \tilde{G}_{0}^{\rho}(Z,Y)dY
    + \int_{\partial B(X)} \hat{\EPS}^m \nabla \hat{u}_m \cdot \nu \tilde{G}_{0}^{\rho}(Z,Y)dY\label{eq:defhatF^rho}.
\end{align}
is that for the term \(\tilde{F}_1\) we have few situations where derivative hit terms that do not have required regularity. This is not true for mollified coefficients as those are smooth. 
Here, in the formula above \(\tilde{G}_0^\rho(Z,Y)\in W^{1,2}_0(2B(X))\) is the unique function that satisfies
\[\int_{2B(X)}A_0(Y)\nabla \tilde{G}_0^{\rho}(Z,Y)\nabla \phi(Y) dY=\fint_{B(Z,\rho)}\phi(Y) dY
    \quad \forall \phi\in W^{1,2}_0(2B(X)), \]
    which exists by the Lax-Milgram theorem.
From \eqref{eq:defhatF^rho}, it is clear that \( \hat{F}_m^{\rho} \) is continuous on \( 2B(X) \setminus \frac{3}{2}B(X) \) 
    with \( \hat{F}_m^{\rho} = 0 \) on \( \partial 2B(X) \).
Next note that by \cite{gruter_green_1982}
\[ \|\tilde{G}_0^{\rho}\|_{L^{\frac{n}{n-2},\infty}(2B(X))} \lesssim 1, \]
    where the implied constant is independent of \( \rho \).
Hence \( \|\tilde{G}_0^{\rho}\|_{L^{1}(2B(X))} \lesssim 1 \) and therefore
\begin{align}
    |\hat{F}_m^{\rho}| 
    \lesssim \| \Div (\hat{\EPS}^m \nabla \hat{u}_m) \|_{\infty} + \| \hat{\EPS}^m \nabla \hat{u}_m \|_{\infty}
    \lesssim_m 1. \label{eq:FhatrhoBounded}
\end{align}
  This allows us to claim that \( \hat{F}_m^{\rho} \in L^{\infty}(2B(X)) \).
Finally by Minkowski's integral inequality
\begin{align*}
    &\left\|  \int_{B(X)} |\Div (\hat{\EPS}^m \nabla \hat{u}_m)| |\nabla_Z \tilde{G}_{0}^{\rho}(Z,Y)|dY \right\|_{L_Z^2(2B(X))}
    \\
    &\leq \int_{B(X)} |\Div (\hat{\EPS}^m \nabla \hat{u}_m)| \| \nabla_Z\tilde{G}_{0}^{\rho}(\cdot,Y) \|_{L^2(2B(X))} dY
    \\
    &\lesssim_{\rho} \int_{B(X)} |\Div (\hat{\EPS}^m \nabla \hat{u}_m)| dY
    \lesssim_m |B(X)| \lesssim_X 1,
\end{align*}
    and similarly
\[ \left\| \int_{\partial B(X)} |\hat{\EPS}^m \nabla \hat{u}_m| 
    |\nabla_Z \tilde{G}_{0}^{\rho}(Z,Y)|dY \right\|_{L_Z^2(2B(X))}
    \lesssim_{\rho,X} 1. \]
Thus  \( \| \nabla \hat{F}_m^{\rho} \|_{L^2(2B(X))} \lesssim_{\rho,X} 1 \) 
    and we can conclude that \( \hat{F}_m^{\rho} \in W_0^{1,2}(2B(X)) \).
Next,
\begin{align*}
     \fint_{2B(X)} |\nabla \hat{F}^\rho_m|^2 
    &\lesssim \fint_{2B(X)} A_0 \nabla \hat{F}^\rho_m \cdot \nabla \hat{F}^\rho_m
    \\
    &= -   \int_{B(X)} \Div (\hat{\EPS}^m \nabla \hat{u}_m)
    \left( \int_{2B(X)} A_0  \nabla_Z \tilde{G}^\rho_{0}(Z,Y) \cdot \nabla \hat{F}^\rho_m
        dZ \right)dY 
        \\
        &\quad+   \int_{\partial B(X)} \nu \cdot \hat{\EPS}^m \nabla \hat{u}_m
        \left( \int_{2B(X)} A_0  \nabla_Z \tilde{G}^\rho_{0}(Z,Y) \cdot \nabla \hat{F}^\rho_m 
        dZ \right)dY 
    \\[1 ex]
    &= -   \int_{B(X)} \Div (\hat{\EPS}^m \nabla \hat{u}_m) 
    \left(\fint_{B(Y,\rho)}\hat{F}_m^{\rho} \right) dY
        \\
        &\quad+   \int_{\partial B(X)} \nu \cdot \hat{\EPS}^m \nabla \hat{u}_m
    \left(\fint_{B(Y,\rho)}\hat{F}^\rho_m\right) dY
    \\[1 ex]
    &= \int_{B(X)} \hat{\EPS}^m \nabla \hat{u}_m \cdot \nabla \left(\fint_{B(Y,\rho)}\hat{F}^\rho_m\right) dY
    \\
    &= \int_{B(X)} \hat{\EPS}^m \nabla \hat{u}_m \cdot   \left(\fint_{B(Y,\rho)}\nabla\hat{F}^\rho_m\right)  dY.
\end{align*}
Therefore
\begin{align*}
    \fint_{2B(X)} |\nabla \hat{F}^\rho_m|^2 
    &\leq C\fint_{B(X)} \hat{\EPS}^m \nabla \hat{u}_m \cdot 
        \left(\fint_{B(Y,\rho)}\nabla\hat{F}^\rho_m\right) 
    \\
    &\leq \frac{1}{2} \fint_{B(X)} \left|\left(\fint_{B(Y,\rho)}\nabla\hat{F}^\rho_m\right)\right|^2
        + C \fint_{B(X)} |\hat{\EPS}^m|^2 |\nabla \hat{u}_m |^2. 
\end{align*}

Since the term
\[\fint_{B(X)} \left|\left(\fint_{B(Y,\rho)}\nabla\hat{F}^\rho_m\right)\right|^2\leq \fint_{B(X)} \fint_{B(Y,\rho)}|\nabla\hat{F}^\rho_m|^2\lesssim \fint_{2B(X)}|\nabla\hat{F}^\rho_m|^2,\]
 it can be absorbed by the left side of the expression above and thus giving us that
 
 \begin{align*}
    \fint_{2B(X)} |\nabla \hat{F}^\rho_m|^2&\lesssim \fint_{B(X)} |\hat{\EPS}^m|^2 |\nabla \hat{u}_m|^2 
    \leq \left(\fint_{B(X)} |\hat{\EPS}^m|^r\right)^{2/r} \left(\fint_{B(X)}|\nabla \hat{u}_m|^\frac{2r}{r-2}\right)^\frac{r-2}{r}.
\end{align*}

We also note that
\[\left(\fint_{B(X)} |\hat{\EPS}^m|^r\right)^{2/r}
    \to \left(\fint_{B(X)} |\EPS|^r\right)^{2/r}, \quad m \to \infty \]
We know that if \(\delta(X)\geq 4R_0\), then \(\EPS=0\) on \(B(X)\) nd there is nothing to show. 
Hence may assume that \(\delta(X)\leq 4R_0\)  and use \eqref{eq:PrelimB(Z,delta(Z)/4)SubsetIalphak} 
    and \refprop{prop:EPSleqEPS_0}, to obtain 
\[\left(\fint_{B(X)} |\EPS|^r\right)^{2/r}
    \leq \left(\fint_{I_\alpha^k} |\EPS|^r\right)^{2/r}\leq \EPS_0^2. \]
For the remaining term we have by \refprop{prop:GradRevHol} 
\begin{align*}
    \left(\fint_{B(X)}|\nabla \hat{u}_m|^\frac{2r}{r-2}\right)^\frac{r-2}{r}
    \lesssim \fint_{2B(X)}|\nabla \hat{u}_m|^2\to \fint_{2B(X)}|\nabla u_1|^2, \quad m \to \infty,
\end{align*}

and by the \emph{Poincare inequality} for \( \hat{F}^\rho_m \) we have
\begin{align}
    \fint_{B(X)} |\hat{F}^\rho_m|^2 
    \lesssim\fint_{2B(X)} |\hat{F}^\rho_m|^2 
    \lesssim \delta(X)^{2-n} \int_{2B(X)} |\nabla \hat{F}^\rho_m|^2.\label{eq:hatFmrhobound1}
\end{align}
Collecting all terms together then yields:
\begin{align}
    \lim_{m\to \infty}\left(\delta(X)^{2-n} \int_{2B(X)} |\nabla \hat{F}^\rho_m|^2\right)
    &\lesssim \EPS_0^2  \int_{2B(X)}  |\nabla \hat{u}_1|^2 \delta(Z)^{2-n} dZ \nonumber
    \\
    &\leq \EPS_0^2 S_{\bar{M}}[u_1](Q_0)^2.\label{eq:hatFmrhobound2}
\end{align}

To get statement on the original $\tilde{F}_1$ we let $\rho\to 0$ and then $m\to\infty$. We claim that
\begin{align}
    \fint_{B(X)} |\hat{F}_m^{\rho}|^2 \to \fint_{B(X)} |\hat{F}_m|^2, \quad \rho \to 0,\label{eq:F_rhogoestoF_m}     
\end{align}
    and
\begin{align} \fint_{B(X)} |\hat{F}_m|^2 \to \fint_{B(X)} |\tilde{F_1}|^2, \quad m \to \infty.\label{eq:F_mgoestotildeF_1} \end{align}
Having this \eqref{eq:hatFmrhobound1} and \eqref{eq:hatFmrhobound2} combined give us 
\[ \left(\fint_{B(X)} |\tilde{F}_1|^2\right)^{1/2} \lesssim \EPS_0 S_{\bar{M}}[u_1](Q), \]
    as desired.\\
    
We start by proving \eqref{eq:F_rhogoestoF_m}. As this is a claim at every $X$ we allow the implicit constants below to depend on \( X \).
We know from \cite{li_lp_2019} that for a fixed \( Z \), 
    \( \tilde{G}_{0}^{\rho}(Z,\cdot) \to \tilde{G}_{0}(Z,\cdot) \) 
    weakly in \( W_0^{1,1+\eta}(2B(X)) \) for small \( \eta > 0 \). 
Hence we also have weak convergence in \( W^{1,1+\eta}(B(X)) \).
For any \( v \in W^{1,1+\eta}(B(X)) \)
\[ \left| \int_{B(X)} \nabla v \cdot \hat{\EPS}^m \nabla \hat{u}_m dY \right| 
    \leq \| \hat{\EPS}^m \nabla \hat{u}_m \|_{L^\infty(B(X))} \| \nabla v \|_{L^1(B(X))}
    \lesssim_{m,X} \| v \|_{W^{1,1+\eta}(B(X))}, \]
    and in particular, this means that
\[ \int_{B(X)} \nabla \tilde{G}_{0}^{\rho}(Z,\cdot) \cdot \hat{\EPS}^m \nabla \hat{u}_m dY
    \to \int_{B(X)} \nabla \tilde{G}_{0}(Z,\cdot) \cdot \hat{\EPS}^m \nabla \hat{u}_m dY, \quad \rho \to 0, \]
    i.e., \( \hat{F}_m^{\rho}(Z) \to \hat{F}_m(Z) \).
Thus using \eqref{eq:FhatrhoBounded} and the dominated convergence theorem we conclude that
\begin{align} \fint_{B(X)} |\hat{F}_m^{\rho}|^2 \to \fint_{B(X)} |\hat{F}_m|^2, \quad \rho \to 0. \end{align}
\\

To establish \eqref{eq:F_mgoestotildeF_1} we consider the pointwise  difference of the two functions at \(Z\in B(X)\).
We have that
\begin{align*}
    (\hat{F}_m-\tilde{F}_1)(Z)
    &= \int_{B(X)} \nabla_Y \tilde{G}_0(Z,Y) 
        \cdot (\hat{\EPS}^m-\EPS)(Y) \nabla \hat{u}_m(Y) dY
    \\
    &\qquad+ \int_{B(X)} \nabla_Y \tilde{G}_0(Z,Y) 
        \cdot \EPS(Y) \nabla (\hat{u}_m-u_1)(Y) dY
    \\
    &\qquad \eqqcolon I^m(Z)+II^m(Z).
\end{align*}
       
We proceed as in the proof of \refprop{prop:DefF_(u_0-u_1)} and consider  a cut-off function \(\vartheta\in C_c([-2,2])\) 
    such that \(0\leq \vartheta\leq 1\) and \(\vartheta\equiv 1\) on \([-1,1]\). 
For each \(0< s <\delta(X)/16\) we set \(\vartheta_s(Y) \coloneqq \vartheta(|Y-Z|/s)\) 
    and \(\psi_s \coloneqq 1-\vartheta_s\). 
This allows us to write
\begin{align*}
    I^m(Z)&=\int_{B(X)} \nabla_Y (\tilde{G}_0(Z,Y)\psi_s(Y)) \cdot (\hat{\EPS}^m-\EPS)(Y) \nabla \hat{u}_m(Y) dY
    \\
    &\qquad +\int_{B(X)} \nabla_Y \tilde{G}_0(Z,Y)\vartheta_s(Y) \cdot (\hat{\EPS}^m-\EPS)(Y) \nabla \hat{u}_m(Y) dY
    \\
    &\qquad +\int_{B(X)} \tilde{G}_0(Z,Y)\nabla\vartheta_s(Y) \cdot (\hat{\EPS}^m-\EPS)(Y) \nabla \hat{u}_m(Y) dY
    \\
    &\qquad \eqqcolon \tilde{I}_s^m(Z)+\hat{I}_s^m(Z)+\bar{I}_s^m(Z),
\end{align*}

and 
\begin{align*}
    II^m(Z)&=\int_{B(X)} \nabla_Y (\tilde{G}_0(Z,Y)\psi_s(Y)) \cdot \EPS(Y) \nabla (\hat{u}_m-u_1)(Y) dY
    \\
    &\qquad +\int_{B(X)} \nabla_Y \tilde{G}_0(Z,Y)\vartheta_s(Y) \cdot \EPS(Y) \nabla (\hat{u}_m-u_1)(Y) dY
    \\
    &\qquad +\int_{B(X)}  \tilde{G}_0(Z,Y)\nabla\vartheta_s(Y) \cdot \EPS(Y) \nabla (\hat{u}_m-u_1)(Y) dY
    \\
    &\qquad \eqqcolon \tilde{II}^m_s(Z) +\hat{II}_s^m(Z)+\bar{II}_s^m(Z),
\end{align*}

For \(\tilde{I}_s^m\) and \(\tilde{II}_s^m\) we can use H\"older's inequality to get
\begin{align*}
    \tilde{I}_s^m(Z)
    &\lesssim \Vert \nabla(\tilde{G}_0(Z,\cdot)\psi_s)\Vert_{L^{\frac{2r}{2-r}}(B(X)\setminus B(Z,s))}
        \Vert\hat{\EPS}_m-\EPS\Vert_{L^r(B(X))}
        \Vert\nabla \hat{u}_m\Vert_{L^{2}(B(X))}
    \\
    & \lesssim \Vert \nabla(\tilde{G}_0(Z,\cdot)\psi_s)\Vert_{L^{\frac{2r}{2-r}}(B(X)\setminus B(Z,s))}
        \Vert\hat{\EPS}_m-\EPS\Vert_{L^r(B(X))},
\end{align*}
    and
\begin{align*}
    \tilde{II}_s^m(Z)
    &\lesssim \Vert\nabla(\tilde{G}_0(Z,\cdot)\psi_s)\Vert_{L^{\frac{2r}{2-r}}(B(X)\setminus B(Z,s))}
        \Vert\EPS\Vert_{L^r(B(X))}
        \Vert\nabla (\hat{u}_m-u_1)\Vert_{L^{2}(B(X))}
    \\
    &\lesssim
        \Vert\nabla(\tilde{G}_0(Z,\cdot)\psi_s)\Vert_{L^{\frac{2r}{2-r}}(B(X)\setminus B(Z,s))}
        \Vert\nabla (\hat{u}_m-u_1)\Vert_{L^{2}(B(X))}.
\end{align*}

By the chain rule we see that
\begin{align*}
    \| \nabla(\tilde{G}_0(Z,\cdot)\psi_s)\|_{L^{\frac{2r}{2-r}}(B(X)\setminus B(Z,s))}
    \leq & \Vert \nabla(\tilde{G}_0(Z,\cdot))\Vert_{L^{\frac{2r}{2-r}}(B(X)\setminus B(Z,s))} 
    \\
    &+  \| \nabla \psi_s \|_{\infty} \Vert \tilde{G}_0(Z,\cdot) \Vert_{L^{\frac{2r}{2-r}}(B(X)\setminus B(Z,s))}. 
\end{align*}
Since \( \| \nabla \psi_s \|_{\infty} \lesssim \frac{1}{s} \) by \refprop{prop:GreenBounds}
$\Vert \tilde{G}_0(Z,\cdot) \Vert_{L^{\frac{2r}{2-r}}(B(X)\setminus B(Z,s))}
    \lesssim_{s} 1$ and therefore by \refprop{prop:GradRevHol} and \refprop{prop:GreenExist}
\begin{align*}
    \Vert \nabla\tilde{G}_0(Z,\cdot)\Vert_{L^{\frac{2r}{2-r}}(2B(X)\setminus B(Z,s))} 
    &\lesssim_{s} \Vert \nabla\tilde{G}_0(Z,\cdot)\Vert_{L^{2}(\frac{3}{2}B(X)\setminus \frac{1}{2}B(Z,s))} 
    \\
    &\lesssim_{s} \Vert \tilde{G}_0(Z,\cdot)\Vert_{L^{2}(2B(X)\setminus \frac{1}{3}B(Z,s))} 
    \lesssim_{s} 1.
\end{align*}
Since \( \Vert\hat{\EPS}_m-\EPS\Vert_{L^r(B(X))}, \Vert\nabla (\hat{u}_m-u_1)\Vert_{L^{2}(B(X))} \to 0 \)   for a fixed $s>0$ we may therefore choose  \( m=m(s) \) so that
\[|\tilde{I}_s^{m(s)}|+|\tilde{II}_s^{m(s)}|\lesssim \sqrt{s}.\]

For the remaining terms, estimates similar to the ones in the proofs of \refprop{prop:DefF_(u_0-u_1)} and \refprop{prop:DefF} give us:

\[|\hat{I}_s^m(Z)|,|\bar{I}_s^m(Z)| 
    \lesssim \sqrt{s}M[|\nabla\hat{u}_m|^2\chi_{\frac{3}{2}B(X)}](Z)^{1/2},\]
and
\begin{align*}
|\hat{II}_s^m(Z)|,|\bar{II}_s^m| 
    &\lesssim \sqrt{s}M[|\nabla(\hat{u}_m-u_1)|^2\chi_{\frac{3}{2}B(X)}](Z)^{1/2}
    \\
    &\lesssim \sqrt{s}M[|\nabla\hat{u}_m|^2\chi_{\frac{3}{2}B(X)}](Z)^{1/2}
        + \sqrt{s}M[|\nabla u_1|^2\chi_{\frac{3}{2}B(X)}](Z)^{1/2}.   
\end{align*}

Putting everything together, we get for some \(m=m(s)\)
\begin{align*}
    &\fint_{B(X)}|\hat{F}_m-\tilde{F}_1(Z)| 
    \leq \fint_{B(X)} (\tilde{I}_s^m+\tilde{II}_s^m+\hat{I}_s^m+\hat{II}_s^m+\bar{I}_s^m+\bar{II}_s^m) 
    \\
    &\qquad\lesssim \sqrt{s} 
    + \fint_{B(X)}\sqrt{s} M[|\nabla\hat{u}_m|^2\chi_{\frac{3}{2}B(X)}]^{1/2} 
    + \fint_{B(X)}\sqrt{s} M[|\nabla u_1|^2 \chi_{\frac{3}{2}B(X)}]^{1/2}.
\end{align*}

Since \(\nabla u_1,\nabla \hat{u}_m\in L^{\frac{2r}{r-2}}(\frac{3}{2}B(X))\), 
    by H\"older and the fact that the maximal function \( \| M \|_{L^p \to L^p} < \infty \) for \(p>1\) is bounded we may conclude that

\begin{align*}
    \fint_{B(X)} & \sqrt{s}M[|\nabla u_1|^2\chi_{\frac{3}{2}B(X)}]^{1/2}
    \leq \sqrt{s}\left(\fint_{B(X)} M[|\nabla u_1|^2\chi_{\frac{3}{2}B(X)}]^{r/(r-2)}\right)^{\frac{r-2}{2r}}
    \\
    &\lesssim \sqrt{s}\left(\fint_{\frac{3}{2}B(X)} |\nabla u_1|^{\frac{2r}{r-2}}\right)^{\frac{r-2}{2r}}
    \lesssim \sqrt{s}\left(\fint_{\frac{5}{3}B(X)} |\nabla u_1|^{2}\right)^{1/2}
    \\
    &\lesssim \sqrt{s} \Vert \nabla u_1\Vert_{L^2(2B(X))},
\end{align*}

    and similarly
\[\fint_{B(X)}\sqrt{s} M[|\nabla\hat{u}_m|^2\chi_{\frac{3}{2}B(X)}]^{1/2}\lesssim \sqrt{s}. \]
    
Hence 
\[\fint_{B(X)}|\hat{F}_m-\tilde{F}_1|\lesssim \sqrt{s}.\]
Because the implicit constant in this inequality depends on  $m$ or \(s\)  we conclude that \eqref{eq:F_mgoestotildeF_1} must hold.
Thus
\[\left(\fint_{B(X)}|\tilde{F}_1|^2\right)^{1/2}\lesssim \EPS_0 S_M[u_1](Q),\]
    as desired.

\begin{rem}\label{rem:TildeF}
This involved approximation argument is not shown
    in either \cite{milakis_harmonic_2011} nor \cite{fefferman_theory_1991}.
Instead they argue that \( \tilde{F}_1 \) satisfies the equation
\begin{align}\label{eq:FalseIdentity}
\begin{cases}
    L_0 \tilde{F}_1 = \Div[\EPS \nabla u_1 \chi_{B(X)}], \quad &\text{in } 2B(X),\\
    \tilde{F}_1 = 0, &\text{on } 2B(X).
\end{cases}
\end{align} 
There are two problems with this.
The first one is that it is not clear if the weak derivative \( \nabla \tilde{F}_1 \) 
    even exists in \( L_{\Loc}^1(2B(X)) \) due to the low regularity of the Green's function.
The second issue is that  even if we could claim that \( \tilde{F}_1 \in W_{0}^{1,2}(2B(X)) \),
    the Green's function property \eqref{eq:DefiningPropertyofGreensfct} only holds for \( \varphi \in W_{0}^{1,p}(2B(X)) \)
    with \( p > n \geq 2 \) and so it would not apply for this case. 
\end{rem}

Next we consider bounds for \( \TT{F}_1(Z)  \).
For a large \(r>1\) and \(Z\in B(X)\) we have that
\begin{align*}
    \TT{F}_1(Z) 
    &=  \int_{B(X)} \nabla_Y K(Z,Y) \cdot \EPS(Y) \nabla u_1(Y) dY
    \\
    &\leq \delta(X)^n\left(\fint_{B(X)}|\EPS|^r dY \right)^{1/r} \left( \fint_{B(X)} |\nabla u_1(Y)|^\frac{2r}{r-2} dY \right)^{\frac{r-2}{2r}} 
    \\
    &\hspace{3 em} \cdot \left( \fint_{B(X)} |\nabla_Y K(Z,Y)|^2 dY \right)^{1/2}.
\end{align*}

The first two terms are handled as we did above for \( \tilde{F}_1 \). 
Note that \( L_0K(Z,\cdot)=0 \) in \( 2B(X) \) and \(K(Z,\cdot)\geq 0\),
    so we may use Caccioppoli (\refprop{prop:Caccioppoli}) 
    and Harnack (\refprop{prop:Harnack}) to deduce that
\begin{align*}
    \fint_{B(X)} | \nabla_Y K(Z,Y)|^2 dY 
    &\lesssim \delta(X)^{-1} \left( \fint_{\frac{3}{2}B(X)} |K(Z,Y)|^2 dY \right)^{1/2}
    \\
    &\lesssim \delta(X)^{-1} \sup_{Y \in \frac{3}{2}B(X)} |K(Z,Y)|
    \\
    &\lesssim \delta(X)^{-1} \inf_{Y \in \frac{3}{2}B(X)} |K(Z,Y)|
    \\
    &\lesssim \delta(X)^{-n-1} \int_{\frac{3}{2}B(X)} |K(Z,Y)| dY.
\end{align*}
Bounds in \refprop{prop:GreenBounds} apply on $K$ as it is the sum of two Green's functions. Hence
\[ \int_{\frac{3}{2}B(X)} |K(Z,Y)| \lesssim \int_{\frac{3}{2}B(X)} |Z-Y|^{2-n}dY
    = \int_{0}^{\frac{3}{2}B(X)} t dt 
    \approx \delta(X)^2/2. \]
Combining this with the previous estimate we obtain
\begin{align*}
    |\TT{F}_1(Z) | 
    &\lesssim \delta(X)^{n} \EPS_0 \left( \fint_{B(X)} |\nabla u_1(Y)|^2 dY \right)^{1/2}\delta(X)^{-n+1}
    \\
    &\lesssim \EPS_0 \left( \int_{B(X)} |\nabla u_1(Y)|^2 \delta(Y)^{2-n} dY \right)^{1/2} 
     \leq \EPS_0 S_{\bar{M}}[u_1](Q_0).
\end{align*}

\subsection{The \lq\lq away" term \(F_2\)}

The aim is to consider  a fixed point \(Z\in B(X,\delta(X)/8)\). We integrate over \(Y\in \Omega\setminus B(X)\) and by triangle inequality we therefore must have  \(|Z-Y|\geq\delta(X)/8\) for all such points $Y$.
We would like to obtain a pointwise bound
\[F_2(Z)\lesssim C\EPS_0 M_{\omega_0}S_{\bar{M}}(u_1)(Q).\]
Let \( X^*\in \partial \Omega \)  be a point such that \(|X^*-X|=\delta(X)\). We consider the following decompositions of the boundary and the domain:
\begin{align*}
 &\Delta_j\coloneqq \Delta(X^*, 2^{j-1}\delta(X)), 
    \quad \Omega_j \coloneqq \Omega \cap B(X^*,\delta(X)2^{j-1}),
    \quad R_j\coloneqq \Omega_j \setminus (\Omega_{j-1}\cup B(X)),
    \\
    & A^j\coloneqq A(X^*, 2^{j-1}\delta(X)).
\end{align*}
for \(j=-1,0,1,...,N\) where \(N\) is chosen so that \(2^{14}R_0\leq 2^{N-1}\delta(X)<2^{15}R_0\).
Let
\[ \begin{cases}
    F_2^0(Z) \coloneqq \displaystyle\int_{\Omega_0} \EPS(Y) \nabla_{Y} G_0(Z,Y) \cdot \nabla u_1(Y) dY,
    \\
    F_2^j(Z) \coloneqq \displaystyle\int_{R_j} \EPS(Y) \nabla_{Y} G_0(Z,Y) \cdot \nabla u_1(Y) dY. 
\end{cases} \]
This decomposes $F_2$  into the following terms:
\begin{align*}
    |F_2(Z)| &= \left|\int_{\Omega\setminus B(X)} \EPS(Y) \nabla_{Y} G_0(Z,Y) \cdot \nabla u_1(Y) dY\right|
    \\
    &  \leq \left|\int_{\Omega_0} \EPS(Y) \nabla_{Y} G_0(Z,Y) \cdot \nabla u_1(Y) dY\right| 
    \\
    &\qquad +\sum_{j=1}^N\left|\int_{R_j} \EPS(Y) \nabla_{Y} G_0(Z,Y) \cdot \nabla u_1(Y) dY\right|
    \\
    &\qquad + \int_{(\partial\Omega,4R_0)\setminus (B(X)\cup B(X^*,2^{15}R_0))} |\EPS(Y) ||\nabla_{Y} G_0(Z,Y)|| \nabla u_1(Y)| dY
    \\
    &=|F_2^0(Z)| + \sum_{j=1}^N|F_2^j(Z)|+J.
\end{align*}

Starting with estimating $F_2^0(Z)$ we have that
\begin{align*}
    |F_2^0(Z)|&\leq \int_{\Omega_0\cap (\partial\Omega,4R_0)}|\EPS(Y)||\nabla_Y G_0(Z,Y)||\nabla u_1(Y)|dY =
    \\
    &\lim_{\EPS\to 0}\int_{(\Omega_0\cap (\partial\Omega,4R_0))\setminus (\partial\Omega,\EPS)}|\EPS(Y)||\nabla_Y G_0(Z,Y)||\nabla u_1(Y)|dY.
\end{align*}

Since we can cover \((\partial\Omega,4R_0)\setminus (\partial\Omega,\EPS)\) by the decomposition introduced in Subsection \ref{subsection:dyadic decomposition properties}, by \eqref{eq:Prelimk_0leqkleqk_EPS}, we can write
\begin{align}
    &\int_{(\Omega_0\cap (\partial\Omega,4R_0))\setminus (\partial\Omega,\EPS)}|\EPS(Y)||\nabla_Y G_0(Z,Y)||\nabla u_1(Y)|dY\nonumber
    \\
    &\qquad\leq\sum_{\substack{Q_\alpha^k\subset \partial\Omega \\ k_0\leq k\leq k_{\EPS}}}\int_{I_\alpha^k\cap ((\Omega_0\cap (\partial\Omega,4R_0))\setminus (\partial\Omega,\EPS))}|\EPS(Y)||\nabla_Y G_0(Z,Y)||\nabla u_1(Y)| dY\nonumber
    \\
    &\qquad\leq\sum_{\substack{Q_\alpha^k\subset 3\Delta_0 \\ k_0\leq k\leq k_{\EPS}}}
        \int_{I_\alpha^k}|\EPS(Y)||\nabla_Y G_0(Z,Y)||\nabla u_1(Y)| dY\nonumber
    \\
    &\qquad\leq\sum_{\substack{Q_\alpha^k\subset 3\Delta_0 \\ k_0\leq k\leq k_{\EPS}}}
        \mathrm{diam}(Q_\alpha^k)^n\left(\fint_{I_\alpha^k}|\EPS(Y)|^r dY\right)^{1/r}\left(\fint_{I_\alpha^k}|\nabla_Y G_0(Z,Y)|^2dY\right)^{1/2}\nonumber
        \\[-2 ex]
        & \hspace{11 em} \cdot \left(\fint_{I_\alpha^k}|\nabla u_1(Y)|^{\frac{2r}{r-2}} dY\right)^{\frac{r-2}{2r}}\label{eq:StoppingTimeArgument_FirstStep}
\end{align}

Using the ball covering that we have introduced in subsection \ref{subsection:dyadic decomposition properties} 
    and its properties \eqref{eq:PrelimSameSizeOfIalphaB(Z)} 
    and \eqref{eq:PrelimB(Xi)SubsetB(Z)} together with \refprop{prop:GradRevHol} we obtain
\begin{align}
   \left(\fint_{I_\alpha^k}|\nabla u_1|^{\frac{2r}{r-2}} \right)^{\frac{r-2}{2r}}&\leq \sum_{i=1}^N\left(\fint_{B(X_i,\lambda 8^{-k-3})}|\nabla u_1|^\frac{r-2}{2r} \right)^{\frac{r-2}{2r}} \nonumber
    \\
    &\lesssim\sum_{i=1}^N \left(\fint_{B(X_i,2\lambda 8^{-k-3})}|\nabla u_1|^2 \right)^{1/2}\nonumber
    \\
    &\lesssim \left(\fint_{\hat{I}_\alpha^k}|\nabla u_1|^2 \right)^{1/2}.\label{eq:StoppingTimeArgument_RevHoelderGRadu}
\end{align}
We estimate the Green's function using Caccioppoli's inequality
\begin{align*}
    \left(\fint_{I_\alpha^k}|\nabla_Y G_0(Z,Y)|^2dY\right)^{1/2}&\lesssim \mathrm{diam}(Q_\alpha^k)^{-1}\left(\fint_{\hat{I}_\alpha^k}|G_0(Z,Y)|^2dY\right)^{1/2}
    \\
    &\lesssim \mathrm{diam}(Q_\alpha^k)^{-1}\left(\fint_{\hat{I}_\alpha^k}\frac{|G_0(Z,Y)|^2}{|G_0(Y)|^2}|G_0(Y)|^2dY\right)^{1/2}
    \\
    &\lesssim \mathrm{diam}(Q_\alpha^k)^{-1}\left(\inf_{Y\in \hat{I}_\alpha^k} \frac{|G_0(Z,Y)|^2}{|G_0(Y)|^2}\right)\left(\fint_{\hat{I}_\alpha^k}|G_0(Y)|^2dY\right)^{1/2}.
\end{align*}

For the last term we further use the comparison principle and the doubling property of the elliptic measure:
\begin{align*}
\mathrm{diam}(Q_\alpha^k)^{-1}\left(\fint_{\hat{I}_\alpha^k}|G_0(Y)|^2dY\right)^{1/2}&\approx \mathrm{diam}(Q_\alpha^k)^{-1}\left(\fint_{\hat{I}_\alpha^k}\frac{\omega_0(Q_\alpha^k)^2}{\mathrm{diam}(Q_\alpha^k)^{2n-4}}dY\right)^{1/2}
\\
&\lesssim \omega_0(Q_\alpha^k)\mathrm{diam}(Q_\alpha^k)^{-n+1}.
\end{align*}

Since we can cover \(5\Delta_0\) with \(N\) balls \(B(Q_i,\delta(X)/4)\) such that \(|Q_i-Q_j|<\delta(X)/4, Q_i\in 5\Delta_0\) we see that \(\Omega_0\cap (\partial\Omega,\delta(X)/8)\subset\bigcup_i B(Q_i,\delta(X)/4)\). Note that \(N\) here is independent of \(X\) and \(\delta(X)\). Let \(\tilde{A}_i=A(Q_i, \delta(X)/4)\). By the comparison principle  for \(Y\in B(Q_i,\delta(X)/4)\) we have that
\[\frac{|G_0(Z,Y)|}{|G_0(Y)|}\approx \frac{|G_0(Z,\tilde{A}_i)|}{|G_0(\tilde{A}_i)|}.\]
By the Harnack's inequality for all \(Y\in \Omega_0\setminus (\partial\Omega,\delta(X)/16)\)
\[\frac{|G_0(Z,Y)|}{|G_0(Y)|}\approx \frac{|G_0(Z,A_0)|}{|G_0(A_0)|}.\]
Since \(\tilde{A}_i\in \Omega_0\setminus (\partial\Omega,\delta(X)/16)\) also have the same estimates for all \(Y\in \Omega_0\cap (\partial\Omega,\delta(X)/16)\), that is
\[\frac{|G_0(Z,Y)|}{|G_0(Y)|}\approx\frac{|G_0(Z,\tilde{A}_i)|}{|G_0(\tilde{A}_i)|}\approx \frac{|G_0(Z,A_0)|}{|G_0(A_0)|}.\]
Hence we may use the comparison principle for the Green's function to obtain
\begin{align}
    \frac{|G_0(Z,A_0)|}{|G_0(A_0)|}
    \lesssim \frac{\omega^Z_0(\Delta(   X^*, \delta(X)/2))}{\omega_0(\Delta(X^*, \delta(X)/2))}
    \lesssim \frac{1}{\omega_0(\Delta_0)}\label{eq:GreensfctFractionboundedbyomega(Delta0)}.
\end{align}
After we put all pieces together we finally have for the gradient of $G_0$:
\begin{align}
    \left(\fint_{I_\alpha^k}|\nabla_Y G_0(Z,Y)|^2dY\right)^{1/2}
    \lesssim \frac{\omega_0(Q_\alpha^k)\mathrm{diam}(Q_\alpha^k)^{-n+1}}{\omega_0(\Delta_0)}.\label{eq:StoppingTimeArgument_CacciopolliOnGreens}
\end{align}

Next, we consider the term of \eqref{eq:StoppingTimeArgument_FirstStep} containing $\varepsilon$ function. By \refprop{prop:EPSleqEPS_0} we have
\begin{align*}
    \omega_0(Q_\alpha^k)\left(\fint_{I_\alpha^k}|\EPS|^{r} dY\right)^{1/r}&\leq \omega_0(Q_\alpha^k)^{1/2}\left(\int_{\hat{I}^k_\alpha}\frac{G_0(Z)\beta_r(Z)^2}{\delta(Z)^2}dZ\right)^{1/2},
\end{align*}
    and hence \eqref{eq:StoppingTimeArgument_FirstStep} can be further estimated by
\begin{align}\nonumber
    &\sum_{\substack{Q_\alpha^k\subset 3\Delta_0 \\ k_0\leq k\leq k_{\EPS}}}\mathrm{diam}(Q_\alpha^k)^n\left(\fint_{I_\alpha^k}|\EPS(Y)|^r dY\right)^{1/r}\left(\fint_{I_\alpha^k}|\nabla_Y G_0(Z,Y)|^2dY\right)^{1/2}
        \\[-2 ex]\nonumber
        & \hspace{8 em} \cdot \left(\fint_{I_\alpha^k}|\nabla u_1(Y)|^{\frac{2r}{r-2}} dY\right)^{\frac{r-2}{2r}}
    \\[1 ex]\label{e4.16}
    &\lesssim \frac{1}{\omega_0(\Delta_0)}\sum_{\substack{Q_\alpha^k\subset 3\Delta_0 \\ k_0\leq k\leq k_{\EPS}}} \omega_0(Q_\alpha^k)^{1/2}\left(\int_{\hat{I}^k_\alpha}\frac{G_0(Z)\beta_r(Z)^2}{\delta(Z)^2}dZ\right)^{1/2}\left(\int_{\hat{I}_\alpha^k}|\nabla u_1|^2\delta^{2-n} \right)^{1/2}.
\end{align}

For the purposes of the stopping time argument below we define
\[T u_1 (Z)=|\nabla u_1(Z)|^2\delta(Z)^{2-n}\]
and the super-level sets
\[O_j=\{P\in 3\Delta_0; T_\varepsilon u_1(P)=\left(\int_{(\Gamma_M(P)\setminus B_{\varepsilon}(0))\cap(\partial\Omega,4R_0)}T u_1(Z)dZ\right)^{1/2}>2^j\}.\]
We say a dyadic boundary cube \(Q_\alpha^k,\, k_{R_0}\leq k\leq k_\EPS\) belongs to \(J_j\), if
\[\omega_0(O_j\cap Q_\alpha^k)\geq \frac{1}{2}\omega_o(Q_\alpha^k)\qquad\textrm{and}\qquad \omega_0(O_{j+1}\cap Q_\alpha^k)< \frac{1}{2}\omega_o(Q_\alpha^k),\]
and belongs to \(J_\infty\), if
\[\omega_0(Q_\alpha^k\cap\{T_\varepsilon u_1=0\})\geq \frac{1}{2}\omega_o(Q_\alpha^k).\]
Furthermore, let \(M_{\omega_0}\) be the uncentered Hardy-Littlewood maximal function and let
\[\tilde{O}_j=\{M_{\omega_0}(\chi_{O_j})>1/2\}\] 
and observe that
for \(Z\in Q_\alpha^k\in J_j\)
\[M_{\omega_0}(\chi_{O_j})(Z)\geq \frac{\omega_0(Q_\alpha^k\cap O_j)}{\omega_0(Q_\alpha^k)}\geq \frac{1}{2}\]
and hence \(Q_\alpha^k\subset \tilde{O}_j\). Thus, we also have
\[\omega_0(Q_\alpha^k\cap\tilde{O}_j\setminus O_{j+1})=\omega_0(Q_\alpha^k\setminus O_{j+1})\geq \frac{1}{2}\omega_o(Q_\alpha^k).\]
The weak \(L^1\) boundedness of the maximal function therefore implies that
\[\omega_0(\tilde{O}_j\setminus O_{j+1})\leq \omega_0(\tilde{O}_j)=\omega_0(\{M_{\omega_0}(\chi_{O_j})>1/2\})\lesssim \Vert \chi_{O_j}\Vert_{L^1(d\omega_0)}=\omega_0(O_j).\]

We use this to further estimate \eqref{e4.16}. Applying the above decomposition and the Cauchy-Schwarz inequality we get
\begin{align*}
    &\sum_{\substack{Q_\alpha^k\subset 3\Delta_0 \\ k_0\leq k\leq k_{\EPS}}} \omega_0(Q_\alpha^k)^{1/2}\left(\int_{\hat{I}^k_\alpha}\frac{G_0(Z)\beta_r(Z)^2}{\delta(Z)^2}dZ\right)^{1/2}\left(\int_{\hat{I}_\alpha^k}|\nabla u_1|^2\delta^{2-n} dY\right)^{1/2}
    \\
    &\lesssim \frac{1}{\omega_0(\Delta_0)}\sum_{j}\left(\sum_{\substack{Q_\alpha^k\in J_j \\ k_0\leq k\leq k_{\EPS}}}\int_{\hat{I}^k_\alpha}\frac{G_0(Z)\beta_r(Z)^2}{\delta(Z)^2}dZ\right)^{1/2}
        \\
        & \hspace{6 em} \cdot \left(\sum_{\substack{Q_\alpha^k\in J_j \\ k_0\leq k\leq k_{\EPS}}} \omega_0(Q_\alpha^k)\int_{\hat{I}_\alpha^k}|\nabla u_1|^2\delta^{2-n} dY\right)^{1/2}.
\end{align*}

Since for every two cubes \(Q_\alpha^k,Q_\beta^l,\, l\leq k\) either contain each other, i.e., \(Q_\beta^l\subset Q_\alpha^{k}\) or are disjoint \(Q_\beta^l\cap Q_\alpha^{k}=\emptyset\), there is a disjoint collection of cubes in \(J_j\) such that their union covers all the other cubes. We call them the top cubes. We observe that for any such top cube \(Q_\alpha^k\) and its subcube \(Q_\beta^l\subset Q_\alpha^k\) we have \(\hat{I}_\beta^l\subset T(\Delta(Z_\alpha^k, (C_0+16\lambda)8^{-k}))\) and the overlap of these Carleson regions of different top cubes \(Q_\alpha^k\) is finite. We also know that the overlap of the \(\hat{I}_\beta^l\) is finite. Hence

\begin{align*}
    \sum_{Q_\alpha^k\in J_j}\int_{\hat{I}^k_\alpha}\frac{G_0(Z)\beta_r(Z)^2}{\delta(Z)^2}dZ&=\sum_{\substack{Q_\alpha^k\in J_j\\\textrm{disj. top cubes}}}\sum_{Q_\beta^l\in J_j, Q_\beta^l\subset Q_\alpha^k} \int_{\hat{I}^l_\beta}\frac{G_0(Z)\beta_r(Z)^2}{\delta(Z)^2}dZ
    \\
    &\lesssim \sum_{\substack{Q_\alpha^k\in J_j\\\textrm{disj. top cubes}}} \int_{T(\Delta(Z_\alpha^k, (C_0+16\lambda)8^{-k}))}\frac{G_0(Z)\beta_r(Z)^2}{\delta(Z)^2}dZ
    \\
    &\lesssim \sum_{\substack{Q_\alpha^k\in J_j\\\textrm{disj. top cubes}}} \EPS_0^2 \omega_0(Q_\alpha^{k})\leq 2\EPS_0^2\omega_0(O_j). 
\end{align*}
Here we have used \refprop{prop:EPSleqEPS_0} in the penultimate step and the property of cubes in \(J_j\) in the last step. Denote \(S_M^{\varepsilon}(Q)=\Gamma_M(Q)\setminus B_{2\varepsilon(Q)}\cap (\partial\Omega,4R_0)\). Then
\begin{align*}
    \sum_{Q_\alpha^k\in J_j} & \omega_0(Q_\alpha^k)\int_{\hat{I}_\alpha^k}Tu_1(Z)dZ
    =\sum_{\substack{Q_\alpha^k\in J_j\\\textrm{disj. top cubes}}}\sum_{Q_\beta^l\in J_j, Q_\beta^l\subset Q_\alpha^k} \omega_0(Q_\beta^l)\int_{\hat{I}_\beta^l}Tu_1(Z)dZ
  \end{align*}
\begin{align*}
    &\lesssim\sum_{\substack{Q_\alpha^k\in J_j\\\textrm{disj. top cubes}}}\sum_{Q_\beta^l\in J_j, Q_\beta^l\subset Q_\alpha^k} \omega_0((\tilde{O}_j\setminus O_{j+1})\cap Q_\beta^l)\int_{\hat{I}_\beta^l}Tu_1(Z)dZ
    \\
    &\lesssim\sum_{\substack{Q_\alpha^k\in J_j\\\textrm{disj. top cubes}}}\sum_{Q_\beta^l\in J_j, Q_\beta^l\subset Q_\alpha^k} \int_{(\tilde{O}_j\setminus O_{j+1})\cap Q_\beta^l}\int_{\hat{I}_\beta^l}Tu_1(Z)dZd\omega_0(P)
\\    &\lesssim\sum_{\substack{Q_\alpha^k\in J_j\\\textrm{disj. top cubes}}}\sum_{Q_\beta^l\in J_j, Q_\beta^l\subset Q_\alpha^k} \int_{(\tilde{O}_j\setminus O_{j+1})\cap Q_\alpha^k}\int_{\hat{I}_\beta^l}Tu_1(Z)dZd\omega_0(P)
    \\
    &\lesssim\sum_{\substack{Q_\alpha^k\in J_j\\\textrm{disj. top cubes}}} \int_{(\tilde{O}_j\setminus O_{j+1})\cap Q_\alpha^k}\sum_{Q_\beta^l\in J_j, Q_\beta^l\subset Q_\alpha^k}\int_{\hat{I}_\beta^l}Tu_1(Z)dZd\omega_0(P)
    \\
    &\lesssim\sum_{\substack{Q_\alpha^k\in J_j\\\textrm{disj. top cubes}}} \int_{(\tilde{O}_j\setminus O_{j+1})\cap Q_\alpha^k}\int_{S_M^\varepsilon(P)}Tu_1(Z)dZd\omega_0(P)
    \\
    &\leq \int_{(\tilde{O}_j\setminus O_{j+1})}\int_{S_M^\varepsilon(P)}Tu_1(Z)dZd\omega_0(P)
    \\
    &\leq \int_{(\tilde{O}_j\setminus O_{j+1})}T_\varepsilon u_1(P)^2d\omega_0(P).
\end{align*}

We put everything together to finally get the following estimate for  \eqref{eq:StoppingTimeArgument_FirstStep}:
\begin{align*}
    \int_{\Omega_0\setminus (\partial\Omega,\EPS)} &|\EPS||G_0||\nabla F||\nabla u_1|dY
    \lesssim \frac{1}{\omega_0(\Delta_0)}\sum_j \EPS_0 \omega_0(O_j)^{1/2} \left(\int_{(\tilde{O}_j\setminus O_{j+1})}T_\varepsilon u_1^2d\omega_0 \right)^{1/2}
    \\
    &\leq \EPS_0\frac{1}{\omega_0(\Delta_0)}\sum_j 2^{j+1}\omega_0(O_j)^{1/2}\omega_0(\tilde{O}_j\setminus O_{j+1})^{1/2}
    \\
    &\leq \EPS_0\frac{1}{\omega_0(\Delta_0)}\sum_j 2^{j+1}\omega_0(O_j)
    \\
    &\leq \EPS_0\frac{1}{\omega_0(\Delta_0)}\int_{3\Delta_0} T_\varepsilon u_1 d\omega_0
    \\
    &= \EPS_0\frac{1}{\omega_0(\Delta_0)}\int_{3\Delta_0} \left(\int_{(\Gamma_M(P)\setminus B_\varepsilon(P))\cap (\partial\Omega,4R_0)}|\nabla u_1(Z)|^2\delta(Z)^{2-n}dY\right)^{1/2} d\omega_0(P).
\end{align*}

Taking \(\varepsilon\to 0\) finally yields
\begin{align*}
    F_2^0(Z)=\int_{\Omega_0}|\EPS||\nabla G_0||\nabla u_1|dY
    &\lesssim \EPS_0\frac{1}{\omega_0(\Delta_0)}\int_{3\Delta_0} S_M u_1 \omega_0 \leq M_{\omega_0}[S_M(u)](Q_0).
\end{align*}

Our next objective are the terms $F_2^j$ for \(j\geq 1\). We split the region $R_j$ and write it as a union of two subregions:
\[R_j=(R_j\cap (\partial\Omega,2^{j-6}\delta(X)))\cup (R_j\setminus (\partial\Omega,2^{j-6}\delta(X)))\coloneqq V_j \cup W_j.\]

For \(Y\in R_j\), we observe that \(|A_{-1}-Y|\geq \delta(X)/4\) since \(A_{-1}\in\overline{\Omega_{-1}}\). Now, by the Harnack's inequality  \(G_0(Z,Y)\approx G_0(A_{-1},Y)\) where the implicit constants in the Harnack's inequality are independent of the points \(Z\) and \(X\). Using the boundary H\"older continuity of the nonnegative solution  \(G_0(\cdot, Y)\) in \(\Omega_{j-2}\) in \refprop{prop:BoundaryHoelderContinuity} we get that

\begin{align}
G_0(Z,Y)&\approx G_0(A_{-1},Y)\leq \sup_{\tilde{Z}\in T(\Omega_{-1})} G_0(\tilde{Z},Y)
\lesssim \left(\frac{2^{-2}\delta(X)}{2^{j-3}\delta(X)}\right)^\beta G_0(A_{j-2}, Y)\nonumber
\\
&\lesssim 2^{-j\alpha}G_0(A_{j-2}, Y).\label{eq:2alphaestimateinF_2^j}
\end{align}

\hfill\\

Assume now that \(Y\in V_j\). We proceed similarly to the above case for \(\Omega_0\).
We have a bound  analogous to \eqref{eq:StoppingTimeArgument_FirstStep}, namely that
\begin{align*}
    &\int_{V_j\setminus (\partial\Omega,\EPS)}|\EPS(Y)||\nabla_Y G_0(Z,Y)||\nabla u_1(Y)|dY
    \\
    &\qquad\leq\sum_{\substack{Q_\alpha^k\subset \frac{3}{2}\Delta_j\setminus \frac{3}{2}\Delta_{j-2} \\ k_0\leq k\leq k_{\EPS}}}\mathrm{diam}(Q_\alpha^k)^n\left(\fint_{I_\alpha^k}|\EPS(Y)|^r dY\right)^{1/r}\left(\fint_{I_\alpha^k}|\nabla_Y G_0(Z,Y)|^2dY\right)^{1/2}
        \\[-2 ex]
        & \hspace{14 em}\cdot\left(\fint_{I_\alpha^k}|\nabla u_1(Y)|^{\frac{2r}{r-2}} dY\right)^{\frac{r-2}{2r}}
\end{align*}

Instead of \eqref{eq:StoppingTimeArgument_CacciopolliOnGreens} we have a similar argument. By Caccioppoli's inequality we have
\begin{align*}
    \left(\fint_{I_\alpha^k}|\nabla_Y G_0(Z,Y)|^2dY\right)^{1/2}&\lesssim \mathrm{diam}(Q_\alpha^k)^{-1}\left(\fint_{\hat{I}_\alpha^k}|G_0(Z,Y)|^2dY\right)^{1/2}
    \\
    &\lesssim \mathrm{diam}(Q_\alpha^k)^{-1}\left(\fint_{\hat{I}_\alpha^k}\frac{|G_0(Z,Y)|^2}{|G_0(Y)|^2}|G_0(Y)|^2dY\right)^{1/2}.
\end{align*}

We can cover \(\frac{3}{2}\Delta_j\setminus \frac{3}{2}\Delta_{j-2}\) with at most \(N\) balls \(B(Q_i,2^{j-5}\delta(X))\) such that \(|Q_i-Q_j|<2^{j-5}\delta(X),\, Q_i\in \frac{3}{2}\Delta_j\setminus \frac{3}{2}\Delta_{j-2}\) and \(I_\alpha^k\subset\bigcup_i B(Q_i,2^{j-5}\delta(X))\) when \(Q_\alpha^k\subset \frac{3}{2}\Delta_j\setminus \frac{3}{2}\Delta_{j-2}\). 
Note that \(N\) is again independent of \(X\) and \(\delta(X)\). Let \(\tilde{A}_i=A(Q_i, 2^{j-5}\delta(X))\) and since \(\mathrm{dist}(I_\alpha^k,A_{j-2})\geq 2^{j-3}-2^{-3}-2^{j-5}\geq 2^{j-5}\delta(X)\) the comparison principle applies and gives us for \(Y\in B(Q_i,2^{j-5}\delta(X))\):

\[\frac{|G_0(A_{j-2},Y)|}{|G_0(Y)|}\approx \frac{|G_0(A_{j-2},\tilde{A}_i)|}{|G_0(\tilde{A}_i)|}.\]

Since \(\delta(\tilde{A}_i)=2^{j-5}\delta(X)\) we can use Harnack to obtain 
\[\frac{|G_0(A_{j-2},\tilde{A}_i)|}{|G_0(\tilde{A}_i)|}\approx \frac{|G_0(A_{j-2},A_j)|}{|G_0(A_j)|}.\]
Hence by \refprop{prop:GreenToOmega} we have that
\begin{align}
    \frac{|G_0(A_{j-2},A_j)|}{|G_0(A_j)|}
    \lesssim \frac{\omega^{A_{j}}_0(\Delta_{j-2})}{\omega_0(\Delta_j)}
    \lesssim \frac{1}{\omega_0(\Delta_j)}.
\end{align}

Combined with the above estimate with \eqref{eq:2alphaestimateinF_2^j} we therefore have
\[\frac{G_0(Z,Y)}{G_0(Y)}\approx 2^{-j\beta}\frac{G_0(A_{j-2},Y)}{G_0(Y)}\lesssim 2^{-j\beta}\frac{1}{\omega_0(\Delta_j)},\]
and hence
\[\left(\fint_{I_\alpha^k}|\nabla_Y G_0(Z,Y)|^2dY\right)^{1/2}\lesssim \frac{1}{\omega_0(\Delta_j)}2^{-j\beta}\mathrm{diam}(Q_\alpha^k)^{-1}\left(\fint_{\hat{I}_\alpha^k}|G_0(Y)|^2dY\right)^{1/2}.\]

We again apply the stopping time argument but this time with the sets \(O_j=\{P\in \frac{3}{2}\Delta_j\setminus\frac{3}{2}\Delta_{j-2}; T_\EPS u_1(P)>2^j\}\).  This  leads to  the estimate
\begin{align*}
    &\int_{V_j\setminus (\partial\Omega,\EPS)}|\EPS(Y)||\nabla_Y G_0(Z,Y)||\nabla u_1(Y)|dY
    \\
    &\qquad \lesssim 2^{-j\alpha}\EPS_0\frac{1}{\omega_0(\Delta_j)}\int_{\frac{3}{2}\Delta_j\setminus\frac{3}{2}\Delta_{j-2}}
    \left(\int_{D_{M,\EPS,R_0}}|\nabla u_1(Z)|^2\delta(Z)^{2-n}dZ\right)^{1/2}d\omega_0(P),
\end{align*}
    where \( D_{M,\EPS,R_0} = (\Gamma_M(P)\setminus B_\varepsilon(P))\cap (\partial\Omega,4R_0) \).
Again, taking \(\varepsilon\to 0\) yields:
\begin{align*}
    \int_{V_j}|\EPS||\nabla_Y G_0||\nabla u_1|dY
    &\lesssim \EPS_0\frac{1}{\omega_0(\Delta_j)}\int_{\frac{3}{2}\Delta_j\setminus\frac{3}{2}\Delta_{j-2}} S_M u_1(P)d\omega_0(P) \\[1 ex]
    &\leq M_{\omega_0}[S_M(u)](Q_0).
\end{align*}


This takes care of the subregion $V_j$. When \(Y\in W_j\) we cover \(W_j\) by at most \(N \) balls \(B_{jl}\coloneqq B(X^j_l, 2^{j-9}\delta(X))\) with \(X_l^j\in W_j\). Again, \(N\) is independent of \(j\). Using the comparison principle for the Green's function, the doubling of the elliptic measure and the fact that \(G_0(A_{j-2})\approx G_0(Y)\) using Harnack we have
\begin{align*}
    G_0(A_{j-2}, Y)
    &=\frac{G_0(A_{j-2}, Y)}{G_0(A_{j-2})}G_0(Y)
    =\frac{G_0^*(Y, A_{j-2})}{G_0(A_{j-2})}G_0(Y)
    \\
    &\lesssim \frac{{\omega^*}^Y_0(\Delta(X^*, 2^{j-3}\delta(X)))}{\omega_0(\Delta(X^*, 2^{j-3}\delta(X)))}G_0(Y)
    \lesssim \frac{1}{\omega_0(\Delta_j)}G_0(Y).
\end{align*}
Since \eqref{eq:2alphaestimateinF_2^j} still applies we therefore have \(G_0(Z,Y)\lesssim 2^{-j\alpha}\frac{1}{\omega_0(\Delta_j)}G_0(Y)\). By \refprop{prop:BoundaryHarnack} we have \(G_0(Y)\lesssim G_0(A_{j+1})\) and therefore

\begin{align*}\left(\fint_{B_{jl}}|\nabla_Y G_0(Z,Y)|^2dY\right)^{1/2}
&\lesssim \frac{2^{-j\beta}2^{-j+7}\delta(X)}{\omega_0(\Delta_j)}\left(\fint_{B_{jl}}|G_0(Y)|^2dY\right)^{1/2}
\\
&\lesssim \frac{2^{-j\beta}}{\omega_0(\Delta_j)}2^{-j}\delta(X)G_0(A_{j+1}).\end{align*}

For every \(Y\in B_{jl}\) we have estimates \(\delta(Y)\geq (2^{j-6}-2^{j-8})\delta(X)\geq 2^{j-7}\delta(X)\) and \(|Y-X^*|\leq (2^{j}+2^{j-8})\delta(X)\leq 2^{j+1}\delta(X)\) therefore have for \(\bar{M}=2^8\)
\[B_{jl}\subset \tilde{B}_{jl}\subset \Gamma_{\bar{M}}(X^*)\subset \Gamma_{\bar{M}}(Q_0),\]
where \(\tilde{B}_{jl}=B(X_l^j,2^{j-8}\delta(X))\) is an enlarged ball.
Since the finite ball covers $(B_{jl})$ and $\tilde{B}_{jl}$ can be chosen to have finite overlap. This implies an estimate similar to \eqref{eq:StoppingTimeArgument_FirstStep}, namely that
\begin{align*}
    &\int_{W_j}|\EPS||\nabla_Y G_0||\nabla u_1|dY
    \leq \sum_l\int_{B_{jl}}|\EPS||\nabla_Y G_0||\nabla u_1|dY
    \\
    &\leq \sum_l (2^{j-7}\delta(X))^n\left(\fint_{B_{jl}}|\EPS(Y)|^r dY\right)^{1/r}\left(\fint_{B_{jl}}|\nabla_Y G_0(Z,Y)|^2dY\right)^{1/2}
        \\
        & \hspace{8 em} \cdot\left(\fint_{B_{jl}}|\nabla u_1(Y)|^{\frac{2r}{r-2}} dY\right)^{\frac{r-2}{2r}}
\\    &\lesssim \sum_l (2^{j}\delta(X))^n\left(\fint_{B_{jl}}|\EPS(Y)|^r dY\right)^{1/r}\left(\fint_{B_{jl}}|\nabla_Y G_0(Z,Y)|^2dY\right)^{1/2}
        \\
        &\hspace{7 em} \cdot \left(\fint_{B_{jl}}|\nabla u_1(Y)|^{\frac{2r}{r-2}} dY\right)^{\frac{r-2}{2r}}.
\end{align*}

By \refprop{prop:GradRevHol} we get a statement analogous to \eqref{eq:StoppingTimeArgument_RevHoelderGRadu}: 
$$\left(\fint_{B_{jl}}|\nabla u_1(Y)|^{\frac{2r}{r-2}} dY\right)^{\frac{r-2}{2r}}\lesssim \left(\fint_{\tilde{B}_{jl}}|\nabla u_1(Y)|^2 dY\right)^{1/2}.$$

Next we consider term containing the function $\varepsilon$. We see that \( |B_{jl}|\approx \delta(Y)^n\approx (2^{j}\delta(X))^n\) and \(B_{jl}\subset B(Y,\delta(Y)/2)\) for \(Y\in B_{jl}\). Hence by the Harnack  \(G_0(Y)\approx G_0(A_{j+1})\) and therefore

\begin{align*}
    \left(\fint_{B_{jl}}|\EPS(Y)|^r dY\right)^{1/r}&
    \lesssim \fint_{B_{jl}} \left(\fint_{B_{jl}}|\EPS(Y)|^r dY\right)^{1/r}dZ
    \\
    &\lesssim \fint_{B_{jl}} \left(\fint_{B(Z,\delta(Z)/2)}|\EPS(Y)|^r dY\right)^{1/r}dZ
    \\
    &\lesssim \left(\frac{1}{(2^{j}\delta(X))^{n-2}G_0(A_{j+1})}\int_{B_{jl}} \frac{\beta_r(Y)^2}{\delta(Y)^2}G_0(Y)dY\right)^{1/2}.
\end{align*}
and therefore
\begin{align*}
    &\sum_l (2^{j}\delta(X))^n\left(\fint_{B_{jl}}|\EPS(Y)|^r dY\right)^{1/r}\left(\fint_{B_{jl}}|\nabla_Y G_0(Z,Y)|^2dY\right)^{1/2}
        \\
        & \hspace{7 em} \cdot\left(\fint_{B_{jl}}|\nabla u_1(Y)|^{\frac{2r}{r-2}} dY\right)^{\frac{r-2}{2r}}
    \\
    &\lesssim \sum_l \frac{2^{-j\beta}}{\omega_0(\Delta_j)}(2^{j}\delta(X))^{-n/2}G_0(A_{j+1})^{1/2}\left(\int_{B_{jl}} \frac{\beta_r(Y)^2}{\delta(Y)^2}G_0(Y)dY\right)^{1/2}
        \\
        & \hspace{7 em} \cdot \left(\fint_{\tilde{B}_{jl}}|\nabla u_1(Y)|^2 dY\right)^{1/2}.
\end{align*}
Next, by the comparison principle, \refprop{prop:GreenToOmega}, Cauchy-Schwarz inequality, and the assumptions of \refthm{thm:theorem2.3} we have
\begin{align*}
    &\sum_l \frac{2^{-j\beta}}{\omega_0(\Delta_j)}(2^{j}\delta(X))^{n/2}G_0(A_{j+1})^{1/2}\left(\int_{B_{jl}} \frac{\beta_r(Y)^2}{\delta(Y)^2}G_0(Y)dY\right)^{1/2}
        \\
        & \hspace{15 em} \cdot \left(\fint_{\tilde{B}_{jl}}|\nabla u_1(Y)|^2 dY\right)^{1/2}
    \\
    &\lesssim \sum_l \frac{2^{-j\beta}}{\omega_0(\Delta_j)}\left((2^{j}\delta(X))^{n-2}G_0(A_{j+1})\right)^{1/2}\left(\int_{B_{jl}} \frac{\beta_r(Y)^2}{\delta(Y)^2}G_0(Y)dY\right)^{1/2}
        \\
        & \hspace{17 em} \cdot\left(\int_{\tilde{B}_{jl}}|\nabla u_1(Y)|^2\delta(Y)^{2-n} dY\right)^{1/2}
\\    &\lesssim \sum_l 2^{-j\beta}\left(\frac{1}{\omega_0(\Delta_{j+1})}\int_{B_{jl}} \frac{\beta_r(Y)^2}{\delta(Y)^2}G_0(Y)dY\right)^{1/2}
        \\
        & \hspace{4 em} \cdot\left(\int_{\tilde{B}_{jl}}|\nabla u_1(Y)|^2\delta(Y)^{2-n} dY\right)^{1/2}
    \\
    &\lesssim 2^{-j\beta}\left(\sum_l\frac{1}{\omega_0(\Delta_{j+1})}\int_{B_{jl}} \frac{\beta_r(Y)^2}{\delta(Y)^2}G_0(Y)dY\right)^{1/2}
        \\
        & \hspace{3 em} \cdot\left(\sum_l\int_{\tilde{B}_{jl}}|\nabla u_1(Y)|^2\delta(Y)^{2-n} dY\right)^{1/2}
\\    &\lesssim 2^{-j\beta}\left(\frac{1}{\omega_0(\Delta_{j+1})}\int_{\Omega_{j+1}} \frac{\beta_r(Y)^2}{\delta(Y)^2}G_0(Y)dY\right)^{1/2}
        \\
        & \hspace{3 em} \cdot\left(\int_{\Omega_{j+1}\setminus (\partial\Omega,2^{j-7}\delta(X))}|\nabla u_1(Y)|^2\delta(Y)^{2-n} dY\right)^{1/2}
    \\
    &\lesssim 2^{-j\beta}\EPS_0 S_{\bar{M}}(u_1)(Q_0).
\end{align*}

Combining together the estimates for the subregions  \(V_j\) and \(W_j\) and summing over $j$'s we get
\begin{align*}
    \sum_{j=1}^N |F_2^j(Z)|&\lesssim \sum_{j=1}^N \int_{R_j}|\EPS(Y)||\nabla_Y G(Z,Y)||\nabla u_1(Y)|dY
    \\
    &\lesssim \sum_{j=1}^N 2^{-j\beta}\EPS_0 M_{\omega_0}[S_{\bar{M}}(u_1)](Q_0)\lesssim \EPS_0 M_{\omega_0}[S_{\bar{M}}(u_1)](Q_0).
\end{align*}

We still have one more term $J$ to tackle. 
First, we observe that 
\[(\partial\Omega,4R_0)\setminus (B(X)\cup B(X^*,2^{15}R_0))\subset \bigcup_{\substack{Q_\alpha^k\subset \partial\Omega\setminus \Delta_{2^{14}R_0}\\k_0\leq k}}I_\alpha^k.\]

To see this consider any \(Y\in (\partial\Omega,4R_0)\setminus (B(X)\cup B(X^*,2^{15}R_0))\). Since the collection \(\{I_\alpha^k\}_{\alpha,k}\) covers \((\partial\Omega,4R_0)\) we have that \(Y\in I_\alpha^k\). We know that that there exists \(P_\alpha^k\in Q_\alpha^k\) such that \(8/\lambda \leq |P_\alpha^k-Y|8^k\leq 8\lambda\). For such \(P\in I_\alpha^k\) we have
\begin{align*}
    |P-X^*|&\geq |Y-X^*|-|P-P_\alpha^k|-|P_\alpha^k-Y| \geq 2^{15}R_0- C_08^{-k}-\lambda 8^{-k+1}
    \\
    &\geq 2^{15}R_0-8R_0-8R_0\geq 2^{14}R_0,
\end{align*}
and hence \(Q_\alpha^k\subset \partial\Omega\setminus \Delta_{2^{14}R_0}\). We again start with an estimate analogous to \eqref{eq:StoppingTimeArgument_FirstStep}: 
\begin{align*}
    &\int_{(\partial\Omega,4R_0)\setminus (B(X)\cup B(X^*,2^{15}R_0))}|\EPS(Y)||\nabla_Y G_0(Z,Y)||\nabla u_1(Y)|dY
    \\
    &\quad\leq\sum_{\substack{Q_\alpha^k\subset \partial\Omega\setminus \Delta_{2^{14}R_0} \\ k_0\leq k}}\int_{I_\alpha^k}|\EPS(Y)||\nabla_Y G_0(Z,Y)||\nabla u_1(Y)| dY
    \\
    &\quad\leq\sum_{\substack{Q_\alpha^k\subset \partial\Omega\setminus \Delta_{2^{14}R_0} \\ k_0\leq k}}\mathrm{diam}(Q_\alpha^k)^n\left(\fint_{I_\alpha^k}|\EPS(Y)|^r dY\right)^{1/r}\left(\fint_{I_\alpha^k}|\nabla_Y G_0(Z,Y)|^2dY\right)^{1/2}
        \\[-2 ex]
        & \hspace{13 em} \cdot \left(\fint_{I_\alpha^k}|\nabla u_1(Y)|^{\frac{2r}{r-2}} dY\right)^{\frac{r-2}{2r}}.
\end{align*}

Since \(Y\in I_\alpha^k\) is far away from \(Z\) and \(0\) we can use Harnack's inequality to conclude that \(G_0(Z,Y)\approx G_0(Y)\), where the implicit constants are independent of Y. Again, using the Cacciopolli's inequality,  \refprop{prop:GreenToOmega} and the fact that \(\omega_0\)  is doubling we have as a replacement of \eqref{eq:StoppingTimeArgument_CacciopolliOnGreens}:
\begin{align*}
    \left(\fint_{I_\alpha^k}|\nabla_Y G_0(Z,Y)|^2dY\right)^{1/2}
    &\lesssim \left(\fint_{I_\alpha^k}|\nabla_Y G_0(Y)|^2dY\right)^{1/2}
    \\
    &\lesssim \mathrm{diam}(Q_\alpha^k)^{-1}\left(\fint_{I_\alpha^k}|G_0(Y)|^2dY\right)^{1/2}
    \\
    &\lesssim \frac{\omega_0(3Q_\alpha^k)}{\mathrm{diam}(Q_\alpha^k)^{n-1}}
    \lesssim \frac{\omega_0(Q_\alpha^k)}{\mathrm{diam}(Q_\alpha^k)^{n-1}}.
\end{align*}

Next step is again the stopping time argument analogous to the case \(F_2^0\) with sets \(O_j=\{P\in 
\partial\Omega\setminus \Delta_{2^{14}R_0}; T_\EPS u_1(P)>2^j\}\). This gives us the final estimate of this section:
\begin{align*}
J&=\int_{(\partial\Omega,4R_0)\setminus (B(X)\cup B(X^*,2^{15}R_0))} |\EPS(Y) ||\nabla_{Y} G_0(Z,Y)|| \nabla u_1(Y)| dY
\\
&\lesssim \EPS_0\int_{\partial \Omega\setminus \Delta_{2^{13}R_0}}S_M(u_1)dP
\lesssim \EPS_0\frac{1}{\omega_0(\partial \Omega)}\int_{\partial \Omega}S_M(u_1)dP
\\[1 ex]
&\lesssim \EPS_0 M_{\omega_0}[S_M(u_1)](Q_0).
\end{align*}


\section{Proof of \reflemma{lem:2.10}} \label{S:Lemma2.10Proof}

To prove \reflemma{lem:2.10} we need establish a good \(\lambda\)-inequality. To shorten our notation let
\[ h[F,u_1] \coloneqq 
    N[F]^2 + N[F]\tilde{N}[\delta|\nabla F|] + N[F]S[u_1] + N[\delta|\nabla F|]S[u_1]\]
where the implicit angle is given by \(\alpha=\hat{M}:=8\bar{M}\).

\begin{lemma}\label{lemma:GoodLambdaInequalityLem2.18}
    There exists \(0<\gamma<1\) such that for all \(\lambda>0\) 
    \begin{align*}
        \omega_0 \left(\left\{ S_{\bar{M}}[F]>2\lambda, \; h[F,u_1]\leq (\lambda\gamma)^2 \right\}\right) 
            \lesssim \gamma^2\omega_0 \left(\{S_{\bar{M}}[F]>\lambda\} \right).
    \end{align*}
\end{lemma}

Also, because \(\omega_0\in A_\infty(d\sigma) \) 
    a similar good \(\lambda\)-inequality holds for \( \sigma \) as well:

\begin{cor}\label{cor:GoodLambdaForomega_0}
    There exists \(0<\eta<1, C>0\) and \(0<\gamma<1\) such that for all \(\lambda>0\) 
    \begin{align*}
        \sigma \left(\left\{ S_{\bar{M}}[F]>2\lambda, \; h[F,u_1]\leq (\lambda\gamma)^2 \right\}\right)
            \lesssim \gamma^\eta \sigma \left(\{S_{\bar{M}}[F]>\lambda\} \right).
    \end{align*}
\end{cor}

\begin{proof}
Consider \( q \) for which \(\omega_0\in A_q(d\sigma)\). Then
\[ \frac{\sigma(E)}{\sigma(\Delta)} 
    \lesssim \bigg( \frac{\omega_0(E)}{\omega_0(\Delta)} \bigg)^{1/q}, 
    \quad E \subset \Delta \subset \partial\Omega, \quad \Delta \text{ cube}. \]
Next take a Whitney decomposition of \(\{S_{\bar{M}}[F]>\lambda\} = \bigcup_j \Delta_j \),
    where \( \Delta_j \subset \partial\Omega \) are cubes, and set 
\[ E_j :=\Delta_j\cap\left\{ S_{\bar{M}}[F]>2\lambda, \; h[F,u_1]\leq (\lambda\gamma)^2 \right\}. \]
    
Then 
\begin{align*}
    \sigma(E_j ) 
    &=\sigma(\Delta_j)\frac{\sigma(E_j )}{\sigma(\Delta_j)}
    \lesssim \sigma(\Delta_j)\bigg(\frac{\omega_0(E_j )}{\omega_0(\Delta_j)}\bigg)^{1/q}
    \lesssim \sigma(\Delta_j)\bigg(\frac{\gamma\omega_0(\Delta_j)}{\omega_0(\Delta_j)}\bigg)^{1/q}
    \\
    &\lesssim \gamma^{1/q} \sigma(\Delta_j).    
\end{align*}
This proves our corollary.
\end{proof}

\reflemma{lem:2.10} is a consequence of the following lemma.

\begin{lemma}\label{lemma:2.10/2.16}
\begin{align}\label{eq:Lemma2.10}
    \int_{\partial\Omega}S_{\bar{M}}[F]^2d\omega_0\lesssim     
        \int_{\partial\Omega}f^2d\omega_0+\int_{\partial\Omega}N_{\alpha}[F]^2d\omega_0.
\end{align}
Moreover if \(\omega_0\in B_p(d\sigma)\) we have
\begin{align}\label{eq:Lemma2.16}
    \int_{\partial\Omega}S_{\bar{M}}[F]^qd\sigma\lesssim \int_{\partial\Omega}f^qd\sigma+\int_{\partial\Omega}N_{\alpha}[F]^qd\sigma,
\end{align}
    
\end{lemma}

\begin{proof}
First we note that due to \reflemma{lemma:CaccioppoliForF} we have \(\tilde{N}(\delta|\nabla F|)\lesssim N(F)\) with slightly larger aperture \(\alpha\).

Now we take \(\mu\in\{\sigma,\omega_0\}\) since the proof works analogously for both measures.
The good \(\lambda\)-inequality of \reflemma{lemma:GoodLambdaInequalityLem2.18} 
    or \refcor{cor:GoodLambdaForomega_0} implies that
\begin{align*}
    \int_{\partial\Omega}S_{\bar{M}}[F]^qd\mu&=q2^{q-1}\int_0^\infty \lambda^{q-1}\mu(\{S_{\bar{M}}[F]>2\lambda\})d\lambda
    \\
    &\lesssim \int_0^\infty \lambda^{q-1}\mu\left(\left\{ S_{\bar{M}}[F]>2\lambda, \; h[F,u_1]\leq (\lambda\gamma)^2 \right\}\right)d\lambda
    \\
    &\qquad+\int_0^\infty \lambda^{q-1}\mu(\{N_{\hat{M}}[F]^2> (\lambda\gamma)^2\})d\lambda
    \\
    &\qquad+\int_0^\infty \lambda^{q-1}\mu(\{N_{\hat{M}}[F]S_{\hat{M}}[u_1]> (\lambda \gamma)^2\})d\lambda
    \\
    &\qquad+\int_0^\infty \lambda^{q-1}\mu(\{\tilde{N}_{\hat{M}}[F]\tilde{N}_{\hat{M}}[\delta|\nabla F|]> (\lambda\gamma)^2\})d\lambda
    \\
    &\qquad+\int_0^\infty \lambda^{q-1}\mu(\{\tilde{N}_{\hat{M}}[\delta|\nabla F|]S_{\hat{M}}[u_1]> (\lambda\gamma)^2\})d\lambda
    \\
    &\leq  \gamma^\eta\int_{\partial\Omega}S_{\hat{M}}[F]^q d\mu+C_\gamma\bigg( \int_{\partial\Omega}(S_{\hat{M}}[u_1]N_{\hat{M}}[F])^{q/2}d\mu +\int_{\partial\Omega}N_{\hat{M}}[F]^qd\mu \bigg),
\end{align*}
    where in the last step we used \reflemma{lemma:NontanMaxFctWithDiffConesComparable}.
Because \(\Vert S_{\bar{M}}[F]\Vert_{L^q(d\mu)}\approx\Vert S_{\hat{M}}[F]\Vert_{L^q(d\mu)}\) 
    thanks to \refprop{prop:SquareFctWithDiffApertures}, we 
    choose \(\gamma\) sufficiently small  so that the first term of the last line can be absorbed by the lefthand side. 
Next,
\[ \int_{\partial\Omega}(S_{\hat{M}}[F]N_{\hat{M}}[F])^{q/2}d\mu
    \leq \rho\int_{\partial\Omega}S_{\bar{M}}[F]^qd\mu
        +C_\rho \int_{\partial\Omega}N_{\hat{M}}[F]^q d\mu. \]
Hence again for a sufficiently small \(\rho\) the square function term can be hidden on the lefthand side. 
We treat the other terms similarly, 
    using the fact that \( S_{\hat{M}}[u_1] \lesssim S_{\hat{M}}[F] + S_{\hat{M}}[u_0] \),
    and then apply \eqref{eq:NFLeqTNF} to obtain
\begin{align*}
    \int_{\partial\Omega}S_{\bar{M}}[F]^qd\mu
    &\lesssim \int_{\partial\Omega}\tilde{N}_{\hat{M}}[u_0]^qd\mu
        +\int_{\partial\Omega}N_{\hat{M}}[F]^qd\mu
        +\int_{\partial\Omega}S_{\hat{M}}[u_0]^qd\mu.
\end{align*}
Finally, note that \(\omega_0\in B_p(\mu)\) which implies 
\[ \Vert S_{\hat{M}}[u_0]\Vert_{L^q(d\mu)} \approx \Vert \tilde{N}_{\hat{M}}[u_0]\Vert_{L^q(d\mu)}
    \lesssim \Vert f\Vert_{L^q(d\mu)},\]
   therefore with the help of \reflemma{lemma:NontanMaxFctWithDiffConesComparable} we have
\[ \int_{\partial\Omega}S_{\hat{M}}[F]^qd\mu\lesssim     
    \int_{\partial\Omega}f^qd\mu+\int_{\partial\Omega}N_{\hat{M}}[F]^qd\mu\lesssim     
    \int_{\partial\Omega}f^qd\mu+\int_{\partial\Omega}N_{\alpha}[F]^qd\mu. \]
This means that \eqref{eq:Lemma2.16} holds.
Since \(\omega_0\in B_2(\omega_0)\) we also get \eqref{eq:Lemma2.10}.
\end{proof}

It remains to establish \reflemma{lemma:GoodLambdaInequalityLem2.18}.

\subsection{Proof of \reflemma{lemma:GoodLambdaInequalityLem2.18}}
Throughout the whole proof we assume aperture \(\alpha\) implicitly of corresponding cones unless stated otherwise.
\\

Let \(\Delta\) be a Whitney cube of \(\{S[F]>\lambda\}\), i.e. \(l(\Delta)\approx\mathrm{dist}(\Delta,\{S[F]\leq\lambda\})\), and define 
\[E :=\{X\in \Delta; S[F]>2\lambda, h[F,u_1]\leq \eta\lambda\}.\]
We can cover \(\{S[F]>\lambda\}\) by Whitney cubes whose overlap is finite, whence it is enough to establish \(\omega_0(E)\lesssim \gamma^2\omega_0(\Delta)\).

From Lemma 1 in \cite{dahlberg_area_1984} we have that for every \(\tau>0\) there exists a \(\gamma>0\) 
    such that for the truncated square function
\[S_{\tau l(\Delta)}[F]^2(X)\coloneqq \int_{\Gamma^{\tau r}(X)}|\nabla F(Y)|^2\delta(Y)^{2-n}dY>\frac{\lambda^2}{4}\]
    holds for all points \(X\in E\), where \(\Gamma^{\tau r}(X)\coloneqq \Gamma(X)\cap Q(X,\tau r)\).
We set the sawtooth regions \(\tilde{\Omega}_i=\bigcup_{X\in E}\Gamma_{(i+1)\alpha/2}^{(i+1)\tau l(\Delta)/2}(X)\) for \(i=0,1,2,3,4,5\)
    and choose  a smooth cut-off \(\eta\) with the following properties
\begin{enumerate}
    \item \(\eta\in C_0^\infty(\tilde{\Omega}_3)\),
    \item \(0\leq \eta\leq 1\),
    \item \(\eta\equiv 1\) on \(\tilde{\Omega}_2\) and \(\eta\equiv 0\) outside of  \(\tilde{\Omega}_3\),
    \item \(|\nabla \eta(Y) |\lesssim \frac{1}{\delta(Y)}\) for all \(Y\in \tilde{\Omega}_3\).
\end{enumerate}

We note that \(\tilde{\Omega}_i\) are nested sawtooth regions over the set \(E\).

\smallskip

Next, we need a partition of \(\tilde{\Omega}_3\setminus\tilde{\Omega}_2\). 
Therefore, more generally, 
    let us define 
    \(\tilde{\mathcal{D}}_{i,j}:=\{I_\alpha^k; I_\alpha^k\cap \tilde{\Omega}_j\setminus \tilde{\Omega}_i\neq \emptyset \}\) 
    as the set of Whitney cubes that cover the area between two sawtooth regions 
    \(\tilde{\Omega}_j\) and \(\tilde{\Omega}_i\) where \(j>i\).
Since the chosen aperture was large enough, we can assume that for all \(I_\alpha^k\in \tilde{\mathcal{D}}_{i,j}\) we have \(\tilde{I}_\alpha^k\subset\tilde{\Omega}_{j+1}\setminus \tilde{\Omega}_{i-1}\). 
We claim the existence of a family \((\eta^{k}_\alpha)_{k,\alpha}\) with the following properties
\begin{enumerate}
    \item \(\eta_\alpha^{k}\in C_0^\infty(\tilde{I}_\alpha^k)\)
    \item \(0\leq \eta^{k}_\alpha\leq 1\) 
    \item \(\eta^{k}_\alpha\equiv 1\) on \(I_\alpha^k\) and \(\eta^{k}_\alpha\equiv 0\) outside of \(\tilde{I}_\alpha^k\)
    \item \(\Vert\nabla \eta^{k}_\alpha\Vert_{L^\infty}\approx\frac{1}{\mathrm{diam}(Q_\alpha^k)}\)
    \item \[\sum_{I_\alpha^k\in \tilde{\mathcal{D}}_{i,j}}\eta_\alpha^{k}\equiv 1\qquad \textrm{on } \tilde{\Omega}_j\setminus\tilde{\Omega}_i
        \qquad \textrm{and }\sum_{I_\alpha^k\in \tilde{\mathcal{D}}_{i,j}}\eta_\alpha^{k}\equiv 0 
        \qquad\textrm{on }\Omega\setminus(\tilde{\Omega}_{j+1}\setminus\tilde{\Omega}_{i-1}). \]
\end{enumerate}

Furthermore, we define 
\begin{align*}
    &B_1:=\tilde{\Omega}_4\setminus\tilde{\Omega}_1\cap \{Y\in \Omega;\delta(Y)<\tau l(\Delta)\}, 
    \\
    &B_2:=\tilde{\Omega}_4\setminus\tilde{\Omega}_1\cap \{Y\in \Omega;\delta(Y)\geq\tau l(\Delta)\}.
\end{align*}

Let \(G_0\) be the Green's function of the adjoint PDE with pole in \(0\) as before. We start with
\begin{align*}
   \omega_0(E)&\lesssim \frac{1}{\lambda^2}\int_E S_{\tau l(\Delta)}[F]^2d\omega_0
    =\frac{1}{\lambda^2}\int_E\int_{\tilde{\Omega}_1}|\nabla F|^2\delta^{-n+1}\chi_{\Gamma^{\tau l(\Delta)}(X)}dYd\omega_0(X)
    \\
    &=\frac{1}{\lambda^2}\int_{\tilde{\Omega}_1}|\nabla F|^2\delta^{-n+1}\Big(\int_E\chi_{\Gamma^{\tau l(\Delta)}(X)}d\omega_0(X)\Big)dY
    \\
    &\lesssim \frac{1}{\lambda^2}\int_{\tilde{\Omega}_1}|\nabla F|^2\delta^{-n+1}\omega_0(\Delta(X, \delta(X)))dY
    \\
    &\lesssim \frac{1}{\lambda^2}\int_{\tilde{\Omega}_1}|\nabla F|^2G_0dY.
\end{align*}

In the last line we used \refprop{prop:GreenToOmega}. Furthermore, we have using the ellipticity condition and cutoff $\eta$:

\begin{align*}
    \frac{1}{\lambda^2}\int_{\tilde{\Omega}_1}|\nabla F|^2G_0 dY
    &\lesssim \frac{1}{\lambda^2}\int_{\Omega} A_0\nabla F\nabla F G_0\eta^2 dY.
\end{align*}

Since \(G_0\eta, FG_0\eta\in W^{1,2}_0(\Omega)\) we have
\begin{align*}
    \int_{\Omega}A_0\nabla F \cdot \nabla F (G_0\eta^2) 
    =&\int_{\Omega}\Div(A_0\nabla(F^2)) G_0\eta^2-F\Div(A_0\nabla u_0) G_0\eta^2
    \\
        &\qquad+F\Div(A_1\nabla u_1) G_0\eta^2+\Div(\EPS\nabla u_1)FG_0\eta^2 
    \\
    =&\int_{\Omega}\Div(A_0\nabla(F^2)) G_0\eta^2
        +\int_{\Omega}\Div(\EPS\nabla u_1)FG_0\eta^2
    \\
    =&\int_{\Omega}A_0\nabla(F^2) \nabla(G_0\eta^2)
        +\int_{\Omega}\EPS\nabla u_1 \nabla(FG_0\eta^2)
    \\
    :=I+II.
\end{align*}

Set \(\Psi:=\eta-\sum_{I_\alpha^k\in \tilde{\mathcal{D}}_{2,3}}\eta_\alpha^{k}\) which satisfies \(\Psi=1\) on \(\tilde{\Omega}_1\) and \(\Psi=0\) on \(\Omega\setminus\tilde{\Omega}_4\). Note also that \(|\nabla (\eta\Psi)|\lesssim \frac{\chi_{\tilde{\Omega}_4\setminus\tilde{\Omega}_1}}{\delta}\).
For \(I\) we have by \refprop{Prop:ContantAntiSymmMatrix}
\begin{align*}
    I&= \sum_{I_\alpha^k\in \tilde{\mathcal{D}}_{2,3}}\int_{\Omega}A_0 \nabla(F^2) \nabla (G_0 \eta \eta_\alpha^{k}) + \int_{\Omega}A_0\nabla(F^2) \nabla (G_0 \eta \Psi)
    \\
    &= \sum_{I_\alpha^k\in \tilde{\mathcal{D}}_{2,3}}\int_{\Omega}(A_0-(A_0)_{I_\alpha^k})\nabla(F^2) \nabla (G_0 \eta \eta_\alpha^{k}) + \int_{\Omega}A_0\nabla(F^2) \nabla (G_0 \eta \Psi)
    &=I_1+I_2.
\end{align*}

Further, we use that \(F^2\eta\Psi\in W_0^{1,2}(\Omega\setminus B(0,R_0/2))\) and that \(G_0\) is a solution to the adjoint \(L_0^*\) away from the pole to get
\begin{align*}
    I_2&=\int_{\Omega}A_0\nabla(F^2) \nabla G_0 \eta \Psi + A_0\nabla(F^2) G_0 \nabla(\eta\Psi)
    \\
    &=\int_{\Omega}A_0 \nabla(F^2\eta\Psi) \nabla G_0 -A_0  \nabla (\eta\Psi) \nabla G_0 F^2 + A_0\nabla(F^2) G_0 \nabla(\eta \Psi)
    \\
    &=\int_{\Omega}  -A_0  \nabla (\eta\Psi) \nabla G_0 F^2 + A_0\nabla(F^2) G_0 \nabla(\eta\Psi).
\end{align*}
Since both integrands are supported on \(\tilde{\Omega}_4\setminus \tilde{\Omega}_1\), we can insert the partition of unity given by the \(\eta^k_\alpha\). We obtain with \refprop{Prop:ContantAntiSymmMatrix}

\begin{align*}
    &\sum_{I_\alpha^k\in \tilde{\mathcal{D}}_{1,4}}\int_{\Omega}  -A_0 \nabla (\eta\Psi \eta_\alpha^{k}) \nabla G_0 F^2 + A_0\nabla(F^2) G_0 \nabla(\eta\Psi\eta_\alpha^{k})
    \\
    &=\sum_{I_\alpha^k\in \tilde{\mathcal{D}}_{1,4}}\int_{\Omega}  A_0 \nabla F^2 \nabla (G_0\eta\Psi \eta_\alpha^{k})
    \\
    &=\sum_{I_\alpha^k\in \tilde{\mathcal{D}}_{1,4}}\int_{\Omega}  (A_0 - (A_0)_{\tilde{I}_\alpha^k}) \nabla F^2 \nabla (G_0\eta\Psi \eta_\alpha^{k})
    \\
    &=\sum_{I_\alpha^k\in \tilde{\mathcal{D}}_{1,4}}\int_{\Omega}(A_0 - (A_0)_{\tilde{I}_\alpha^k})\nabla(F^2) \nabla G_0 \eta \Psi \eta_\alpha^{k} + (A_0 - (A_0)_{\tilde{I}_\alpha^k})\nabla(F^2) G_0 \nabla(\eta\Psi \eta_\alpha^{k})
    \\
    &=\sum_{I_\alpha^k\in \tilde{\mathcal{D}}_{1,4}}\int_{\Omega}(A_0 - (A_0)_{\tilde{I}_\alpha^k})  \nabla(F^2\eta\Psi\eta_\alpha^{k}) \nabla G_0 -(A_0 - (A_0)_{\tilde{I}_\alpha^k}) \nabla (\eta\Psi\eta_\alpha^{k})\nabla G_0 F^2
    \\
    &\qquad\qquad\qquad+ (A_0 - (A_0)_{\tilde{I}_\alpha^k})\nabla(F^2) G_0 \nabla(\eta \Psi\eta_\alpha^{k})
\end{align*}
Since \(F^2\eta\Psi\eta_\alpha^k\in W_0^{1,2}(\Omega\setminus B(0,R_0/2))\), we further obtain
\begin{align*}
    &\sum_{I_\alpha^k\in \tilde{\mathcal{D}}_{1,4}}\int_{\Omega}A_0^* \nabla G_0 \nabla(F^2\eta\Psi\eta_\alpha^{k}) -(A_0 - (A_0)_{\tilde{I}_\alpha^k}) \nabla (\eta\Psi\eta_\alpha^{k}) \nabla G_0 F^2
    \\
    &\qquad\qquad\qquad+ (A_0 - (A_0)_{\tilde{I}_\alpha^k})\nabla(F^2) G_0 \nabla(\eta \Psi\eta_\alpha^{k})
    \\
    &=\sum_{I_\alpha^k\in \tilde{\mathcal{D}}_{1,4}}\int_{\Omega}  -(A_0 - (A_0)_{\tilde{I}_\alpha^k}) \nabla (\eta\Psi\eta_\alpha^{k}) \nabla G_0 F^2 + (A_0 - (A_0)_{\tilde{I}_\alpha^k})\nabla(F^2) G_0 \nabla(\eta\Psi\eta_\alpha^{k})
    \\
    &=: I_{21} + I_{22}.
\end{align*}
We can easily see that the penultimate line in the previous calculation equals to \(I_1\) with the only exception that the cutoff function consists of \(\eta\Psi\eta_\alpha^{k}\) for \(I_\alpha^{k}\in\tilde{\mathcal{D}}_{1,4}\) instead of \(\eta\eta_\alpha^{k}\) for \(I_\alpha^{k}\in\tilde{\mathcal{D}}_{2,3}\). However, these term can be dealt with completely analogously, and we are only going to show the calculations for \(I_{21}\) and \(I_{22}\).

\smallskip


Before we are able to calculate \(I_{21}\) and \(I_{22}\), we need to introduce new quasi-annuli. We define
\[\Ann_i:=\{Y\in\partial\Omega; c2^{i-1}\leq\mathrm{dist}(Y,E)\leq c2^{i}\}\qquad\textrm{ for all }{i\leq N},\]
where \(N\in \mathbb{Z}\) is chosen such that \(5\tau l(\Delta)\leq c2^N\leq 6\tau l(\Delta)\).
Next, we cover \(\Ann_i\) by \(N_i\) boundary cubes \((\Delta_m^i)_{1\leq m\leq N_i}\), which have diameter comparable to \(c2^{i}\).
Since \(\tau<\frac{1}{6}\), we have that \(E\cup\bigcup_{i=-\infty}^N \Ann_i\subset 3\Delta\). Next, we observe that if a Whitney cube \(I^i_\alpha \in \mathcal{D}(3\Delta):=\{I_\beta^l; Q_\beta^l\subset 3\Delta\}\) with \(l(I)\approx c2^i\) and \(Y\in I_\alpha^i\cap B_1\neq \emptyset\) then there exists a vertex of a cone \(P\in E\) with
\[\alpha c 2^{i-1}\leq |Y-P|\leq 2\alpha c2^i.\]
Due to the definition of \(\Ann_i\), we know that there exists a \(\Delta_i^m\) such that \(\mathrm{dist}(P,\Delta^m_i)\approx 2^i\). Hence combined, we get that
\begin{align}c2^{i-2}\leq\mathrm{dist}(\Delta_i^m,I^i_\alpha)\approx c2^{i+2}.\label{eq:DistanceDelta^mtoI^+}\end{align}

\smallskip

Hence we showed that for every Whitney cube \(I_\alpha^i\in \mathcal{D}(3\Delta)\) that is part of the covering of \(B_1\), there exists \(m=m(I)\leq N_i\) with \(\mathrm{dist}(\Delta_i^m,I^i_\alpha)\approx c2^i\). We call this boundary cube the \textit{corresponding} boundary cube. 

\smallskip
Conversely, for every such \(\Delta_i^m\) there is at most a finite number \(\tilde{N}\) of Whitney cubes \(I^i_\alpha\) of above form that satisfy (\ref{eq:DistanceDelta^mtoI^+}). The constant \(\tilde{N}\) only depends on the dimension of the ambient space and the constants coming from the uniform domain. To see this, we note that the volume of \(\{Y\in \mathbb{R}^{n};\mathrm{dist}(\Delta_i^m,Y)\lesssim 2^i\}\) is comparable to \(2^{i(n+2)}\), and hence can only be covered by at most \(\tilde{N}\) disjoint Whitney cubes \(I_\alpha^i\) that have volume comparable to \(2^{i(n+2)}\).

\smallskip

In total we have found that for every \(I_\alpha^i\), where \(I_\alpha^i\in \mathcal{D}(3\Delta)\) 
    with \(l(I_\alpha^i)=c2^i\) and \(I_\alpha^i\cap B_1\neq \emptyset\), 
    there exists a corresponding \(\Delta_i^m\) so that \eqref{eq:DistanceDelta^mtoI^+} holds, 
    but there cannot be more than \(\tilde{N}\) of such \(I_\alpha^k\) that correspond to the same \(\Delta_i^m\).
This observation allows the following calculation for the Green's function evaluated at \(X\in I_\alpha^i\) and its corresponding cube \(\Delta_i^m\) that satisfies \eqref{eq:DistanceDelta^mtoI^+}. We have 
\begin{align*}
    &\mathrm{dist}(X^*, \Delta_i^m)
    \lesssim |X^*-X| + |X-Q| +\mathrm{dist}(Q, \Delta_i^m)\lesssim 2^i,
\end{align*}
and hence with the doubling property of the parabolic measure and the Comparison Principle (\refprop{prop:CompPrinc}) and \refprop{prop:GreenToOmega} for all \((Y,s)\in \Delta_i^m\)
\begin{align}
    \frac{G_0(X)}{\delta(X)}&\approx \delta(X)^{-n-1}\omega_0(\Delta(X^*,\delta(X)))\nonumber
    \\
    &\approx \delta(X)^{-n-1}\omega_0(\Delta(Y,3\delta(X)))\nonumber
    \\
    &=\fint_{\Delta(Y,3\delta(X))} k(z)dz\leq M[k\chi_{3\Delta}](Y).\label{eq:G^*overdeltalesssimM[k]}
\end{align}

Furthermore, above observation can be continued to observe that for a boundary Whitney cube \(Q_\alpha^k\in \Ann_i\) there is only a finite number of cubes in \(I_\beta^l\in \tilde{\mathcal{D}}_{1,4}\) such that \(Q_\beta^l\subset Q_\alpha^k\), and hence we obtain  
$\sum_{I_\alpha^k\in \tilde{\mathcal{D}}_{1,4}}\chi_{Q_\alpha^k}\lesssim \chi_{3\Delta}$.
Now, we can estimate \(I_1\) and \(I_2\):
\begin{align*}
    I_{21}
    &=\sum_{I_\alpha^k\in\tilde{\mathcal{D}}_{1,4}}\int_{\tilde{I}_\alpha^k}(A_0-(A_0^a)_{\tilde{I}_\alpha^k}) \nabla (\eta \Psi\eta_\alpha^{k}) F^2 \nabla G_0
    \\
    &\lesssim\sum_{I_\alpha^k\in\tilde{\mathcal{D}}_{1,4}}\sup_{\tilde{I}_\alpha^{k}}F^2\int_{\tilde{I}_\alpha^k}|(A_0-(A_0^a)_{\tilde{I}_\alpha^k}) \nabla (\eta \Psi\eta_\alpha^{k})\nabla G_0|
    \\
    &\lesssim \sum_{I^k_\alpha\in \tilde{\mathcal{D}}_{1,4}} \inf_{Q\in Q_\alpha^k} N(F)^2(Q)\Big(\int_{\tilde{I}_\alpha^k} |A_0-(A_0^a)_{\tilde{I}_\alpha^k}|^r\Big)^{1/r}\Big(\int_{\tilde{I}_\alpha^k}|\nabla G_0|^{2+\EPS}\Big)^{1/(2+\EPS)}
    \\
    &\hspace{80mm}\cdot\Big(\int_{\tilde{I}_\alpha^k}|\nabla (\eta\Psi\eta_\alpha^{k})|^2\Big)^{1/2}
    \\
    &\lesssim \sum_{I^k_\alpha\in \tilde{\mathcal{D}}_{1,4}} \inf_{Q\in Q_\alpha^k} N(F)^2(Q)\Big(\int_{\tilde{I}_\alpha^k}|\nabla G_0|^2\Big)^{1/2}\Big(\int_{\tilde{I}_\alpha^k}|\nabla (\eta\Psi\eta_\alpha^{k})|^2\Big)^{1/2}
    \\
    &\lesssim \sum_{I^k_\alpha\in \tilde{\mathcal{D}}_{1,4}} \inf_{Q\in Q_\alpha^k} N(F)^2(Q)\Big(\fint_{\hat{I}_\alpha^k}\frac{| G_0|^2}{\delta^2}\Big)^{1/2}(c2^k)^{n-1}
    \\
    &\lesssim \sum_{I^k_\alpha\in \tilde{\mathcal{D}}_{1,4}} \inf_{Q\in Q_\alpha^k}N(F)^2(Q)\inf_{Q\in Q_\alpha^k}M[k](Q)|Q_\alpha^k|
    \\
    &\lesssim \int_{3\Delta} N(F)^2(Q) M[k](Q) \lesssim (\lambda\gamma)^2\int_{3\Delta} M[k](Q) d\sigma.
\end{align*}

Similarly, we get
\begin{align*}
    I_{22}
    &\lesssim \sum_{I_\alpha^k\in\tilde{\mathcal{D}}_{1,4}} 
        \inf_{Q\in Q_\alpha^k} N(F)(Q) \inf_{Q\in Q_\alpha^k}M[k](Q)\Big(\fint_{I_\alpha^k}|\nabla F|^2\delta^2\Big)^{1/2}
    \\
    &\lesssim \sum_{I_\alpha^k\in\tilde{\mathcal{D}}_{1,4}} 
        \inf_{Q\in Q_\alpha^k}N(F)(Q)\inf_{Q\in Q_\alpha^k}M[k](Q)\inf_{Q\in Q_\alpha^k}\tilde{N}(|\delta\nabla F|)(Q)|Q_\alpha^k|
    \\
    &\lesssim \int_{3\Delta} N(F)(Q)\tilde{N}(|\delta \nabla F|)(Q) M[k](Q)\lesssim (\lambda\gamma)^2\int_{3\Delta} M[k](Q) d\sigma.
\end{align*}

Next we have for \(II\)
\begin{align*}
    II&=\int_\Omega \EPS\nabla u_1 \nabla F G_0 \eta^2 + \int_\Omega \EPS\nabla u_1  \nabla( G_0 \eta^2) F
=II_1+II_2.
\end{align*}

For the first term, we obtain
\begin{align*}
    II_1&=\int_E \int_{\Gamma_\alpha(P)}\delta^{1-n}\EPS\nabla u_1 \nabla F G_0 \eta^2 dX d\sigma(P)
    \\
    &=\int_E \Big(\sum_{I_\alpha^k\in \mathcal{D}(3\Delta), I_\alpha^k\cap \Gamma(P)\neq \emptyset} \int_{I_\alpha^k}\delta^{1-n}\EPS\nabla u_1 \nabla F G_0 \eta^2 dX \Big)d\sigma(P)
    \\
    &\lesssim\int_E \Big(\sum_{I_\alpha^k\in \mathcal{D}(3\Delta), I_\alpha^k\cap \Gamma(P)\neq \emptyset} \Big(\fint_{I_\alpha^k}|\EPS|^r\Big)^{1/r}\Big(\fint_{I_\alpha^k} |\nabla u_1|^{2+\EPS}\Big)^{1/(2+\EPS)} 
    \\
    &\hspace{50mm}\cdot\Big(\fint_{I_\alpha^k}|\nabla F|^2\Big)^{1/2} l(Q_\alpha^k) \sup_{I_\alpha^k} G_0  \Big)d\sigma(P)
\end{align*}

We can observe that by definition of \(\beta_r\) we have \[\Big(\fint_{I_\alpha^k}|\EPS|^r\Big)^{1/r}\lesssim \Big(\fint_{B(X,\delta(X/2)}|\EPS|^r\Big)^{1/r}=\beta_r(X)\qquad \textrm{for all } X\in I_\alpha^k.\]
Hence we can continue with this observation, reverse H\"{o}lder inequality, 
    and \eqref{eq:G^*overdeltalesssimM[k]} to obtain

\begin{align*}
    II_1&\lesssim\int_E M[k\chi_{3\Delta}]^{1/2}(P) \Bigg(\sum_{I_\alpha^k\in \mathcal{D}(3\Delta), I_\alpha^k\cap \Gamma(P)\neq \emptyset} \Big(\int_{I_\alpha^k} \frac{\beta_r^2G_0}{\delta^{n+1}}\Big)^{1/2}
    \\
    &\hspace{40mm}\cdot\Big(\int_{I_\alpha^k} |\nabla u_1|^2\delta^{2-n}\Big)^{1/2} \Big(\fint_{I_\alpha^k}\delta^2|\nabla F|^2\Big)^{1/2}  \Bigg)d\sigma(P)
    \\
    &=\int_E M[k\chi_{3\Delta}]^{1/2}(P)\tilde{N}[\delta|\nabla F|](P)
    \\
    &\hspace{15mm}\cdot\Bigg(\sum_{I_\alpha^k\in \mathcal{D}(3\Delta), I_\alpha^k\cap \Gamma(P)\neq \emptyset} \Big(\int_{I_\alpha^k} \frac{\beta_r^2G_0}{\delta^{n+1}}\Big)^{1/2}\Big(\int_{I_\alpha^k} |\nabla u_1|^2\delta^{2-n}\Big)^{1/2}  \Bigg)d\sigma(P)
\end{align*}
\begin{align*}    
    &=\int_E M[k\chi_{3\Delta}]^{1/2}(P)\tilde{N}[\delta|\nabla F|](P) S[u_1](P)\Big(\int_{\Gamma^{2l(\Delta)}_{2\alpha}(P)} \frac{\beta_r^2G_0}{\delta^{n+1}}\Big)^{1/2}d\sigma(P)
    \\
    &\lesssim \lambda^2\gamma^2\int_E M[k\chi_{3\Delta}]d\sigma + \lambda^2\gamma^2\int_E\int_{\Gamma^{2l(\Delta)}_{2\alpha}(P)} \frac{\beta_r^2G_0}{\delta^{n+1}}dX d\sigma(P)
    \\
    &\lesssim \lambda^2\gamma^2\int_E M[k\chi_{3\Delta}]d\sigma + \lambda^2\gamma^2\int_{T(3\Delta)} \frac{\beta_r^2G_0}{\delta^{2}}dx 
    \\
    &\lesssim \lambda^2\gamma^2\int_E M[k\chi_{3\Delta}]d\sigma + \lambda^2\gamma^2\omega_0(\Delta),
\end{align*}
where the last line follows from the assumption of \refthm{thm:theorem2.3}.

Similarly, we can estimate for \(II_2\)
\begin{align*}
    II_2&=\int_E N(F)(P)
        \Bigg(\sum_{I_\alpha^k\in \mathcal{D}(3\Delta), I_\alpha^k\cap \Gamma(P)\neq \emptyset} \int_{I_\alpha^k}\delta^{1-n}\EPS\nabla u_1 \nabla (G_0 \eta^2) dX \Bigg)d\sigma(P)
    \\
    &\lesssim\int_E N(F)(P)
        \Bigg(\sum_{I_\alpha^k\in \mathcal{D}(3\Delta), I_\alpha^k\cap \Gamma(P)\neq \emptyset} l(Q_\alpha^k)\Big(\fint_{I_\alpha^k}|\EPS|^r\Big)^{1/r}\Big(\fint_{I_\alpha^k} |\nabla u_1|^{2+\EPS}\Big)^{1/(2+\EPS)} 
    \\
    &\hspace{70mm}\cdot\Big(\fint_{I_\alpha^k}|\nabla G_0|^2 + |\frac{G_0}{\delta}|^2\Big)^{1/2} \Bigg)d\sigma(P)
    \\
    &\lesssim\int_E M[k\chi_{3\Delta}]^{1/2}(P)N(F)(P) 
    \\
    &\hspace{20mm}\cdot\Bigg(\sum_{I_\alpha^k\in \mathcal{D}(3\Delta), I_\alpha^k\cap \Gamma(P)\neq \emptyset} 
    \Big(\int_{I_\alpha^k} \frac{\beta_r^2G_0}{\delta^{n+1}}\Big)^{1/2}\Big(\int_{I_\alpha^k} |\nabla u_1|^2\delta^{2-n}\Big)^{1/2} \Bigg) d\sigma(P)
    \\
    &\lesssim\int_E M[k\chi_{3\Delta}]^{1/2}(P)N[F](P) S[u_1](P)\Big(\int_{\Gamma^{2l(\Delta)}_{2\alpha}(P)} \frac{\beta_r^2G_0}{\delta^{n+1}}\Big)^{1/2}d\sigma(P)
    \\
    &\lesssim \lambda^2\gamma^2\int_E M[k\chi_{3\Delta}]d\sigma + \lambda^2\gamma^2\int_E\int_{\Gamma^{2l(\Delta)}_{2\alpha}(P)} \frac{\beta_r^2G_0}{\delta^{n+1}}d\sigma(P)
    \\
    &\lesssim \lambda^2\gamma^2\int_E M[k\chi_{3\Delta}]d\sigma + \lambda^2\gamma^2\int_{T(5\Delta)} \frac{\beta_r^2G_0}{\delta^{2}}dx 
    \\
    &\lesssim \lambda^2\gamma^2\int_E M[k\chi_{3\Delta}]d\sigma + \lambda^2\gamma^2\omega_0(\Delta).
\end{align*}

Since \(\omega_0\in B_p(d\sigma)\) for some \(p>1\) (as solvability of the Dirichlet problem for $L_0$ is assumed for some $p'>1$), we have using \refdef{DefRH}, boundedness of the maximal function and doubling of $\Omega$:
\begin{align*}
    \int_{3\Delta} M[k\chi_{3\Delta}] d\sigma&\lesssim \sigma(\Delta)\Big(\fint_{3\Delta} M[k\chi_{3\Delta}]^p d\sigma\Big)^{1/p}
    \\
    &\lesssim \sigma(\Delta)\Big(\fint_{3\Delta} k^p d\sigma\Big)^{1/p}
    \\
    &\lesssim \sigma(\Delta)\fint_{3\Delta} k d\sigma = \omega_0(3\Delta)\lesssim \omega_0(\Delta),
\end{align*}
which finishes the proof of the good \(\lambda\)-inequality.


\section{Proof of \refthm{thm:NormSmall}} \label{S:SmallNormProof}

Most of the work to prove \refthm{thm:NormSmall} is already done.
Recall that we want to show that \(\omega_1\in B_p(d\sigma)\) which is equivalent to 
\[\Vert\tilde{N}_\alpha(u_1)\Vert_{L^q(\partial\Omega,d\sigma)}
    \lesssim \Vert f\Vert_{L^q(\partial\Omega,d\sigma)},\qquad\mbox{for }
    \quad \frac{1}{q} + \frac{1}{p} = 1. \]
We assume that \(\omega_0 \in B_p(d\sigma) \) which is equivalent to \( \sigma \in A_{q}(d\omega) \). 
Using this, \reflemma{lem:2.9} and \reflemma{lemma:2.10/2.16} imply:
\begin{align*}
    \int_{\partial\Omega} \tilde{N}_{\alpha}[F]^q d\sigma
    &\lesssim \int_{\partial\Omega} \EPS_0^q M_{\omega_0}[S_{\bar{M}}u_1]^q d\sigma
    \\
    &\lesssim \EPS_0^q \int_{\partial\Omega}  S_{\bar{M}}[u_1]^q d\sigma
\end{align*}
\begin{align*}
    & \lesssim \EPS_0^q\int_{\partial\Omega}  S_{\bar{M}}[F]^q d\sigma 
        + \EPS_0^q \int_{\partial\Omega} S_{\bar{M}}[u_0]^q d\sigma 
\\    &\lesssim \EPS_0^q \int_{\partial\Omega} S_{\bar{M}}[F]^q d\sigma 
        + \int_{\partial\Omega} f^q d\sigma 
    \\
    &\lesssim \EPS_0^q\int_{\partial\Omega} \Tilde{N}_{\alpha}[F]^q  d\sigma +     
        \int_{\partial\Omega}\tilde{N}_{\alpha}[u_0]^qd\sigma +  \int_{\partial\Omega} f^q d\sigma.
    \\
    &\lesssim \EPS_0^q\int_{\partial\Omega} \Tilde{N}_{\alpha}[F]^q  d\sigma  
        + \int_{\partial\Omega} f^q d\sigma.
\end{align*}
By \reflemma{lemma:NontanMaxFctWithDiffConesComparable}, and with \( \EPS_0 \) sufficiently small, 
    we can hide the first term of the righthand side by moving it to the lefthand side. Hence
\[ \| \tilde{N}_{\alpha}[F] \|_{L^q(d\sigma)} \lesssim \| f \|_{L^q(d\sigma)}. \]

Moreover we also have \( \| \tilde{N}_{\alpha}[u_0] \|_{L^q(d\sigma)} \lesssim \| f \|_{L^q(d\sigma)} \) because \(\omega_0\in B_p(d\sigma)\).
Thus 
\begin{align*}
    \int_{\partial\Omega} \tilde{N}_{\alpha}[u_1]^q d\sigma 
    &\lesssim \int_{\partial\Omega} \tilde{N}_{\alpha}[F]^q d\sigma + \int_{\partial\Omega} \tilde{N}_{\alpha}[u_0]^q d\sigma
    \lesssim \int_{\partial\Omega} \tilde{N}_{\alpha}[F]^q d\sigma +  \int_{\partial\Omega} f^q d\sigma
    \\
    &\lesssim \int_{\partial\Omega} f^q d\sigma. 
\end{align*}
  From this \( \omega_1 \in B_p(d\sigma) \). \qed\\

\section{Operators with coefficients satisfying Carleson condition} \label{S:Application}

In this section  \( \Omega \) will be a bounded Lipschitz domain with Lipschitz constant \( K \). 
We consider the operator \(L=\mathrm{div}(A\nabla\cdot)\),
    where \( A \) is \( \lambda_0 \)-elliptic with \( \| A^a \|_{\BMO(\Omega)} \leq \Lambda_0 \) and recall that
\[ \alpha_r(Z) \coloneqq \left( \fint_{B(Z,\delta(Z)/2)} |A - (A)_{B(Z,\delta(Z)/2)}|^r\right)^{1/r}. \]

The aim of this section is to prove \refthm{thm:ApplicationBigNorm} and \refthm{thm:ApplicationSmallNorm}.
But first we prove a similar slightly weaker result:

\begin{thm}\label{thm:ApplicationClassic}
Let \( \Omega \) be a bounded Lipschitz domain with Lipschitz constant \( K \) and suppose that \( \hat{A} \) is elliptic with BMO antisymmetric part.
Moreover suppose that the weak derivative of coefficients exists and consider
\[ \hat{\alpha}^{\eta}(Z) \coloneqq \delta(Z) \sup_{X\in B(Z,\eta \delta(Z))}|\nabla \hat{A}(X)|, \quad 0<\eta<1/2. \]
\begin{itemize}
    \item [(i)] If 
        \[ \| \hat{\alpha}^{\eta}(Z)^2\delta(Z)^{-1}dZ \|_{\CO} < \infty \]
    then the \( L^p \) Dirichlet problem is solvable for \emph{some} \( 1 < p < \infty \),\\
	     
	\item[(ii)] For each \( 1 < p < \infty \), there exists an \( \EPS = \EPS(p) > 0 \), such that if
	\[ \| \hat{\alpha}^{\eta}(Z)^2 \delta(Z)^{-1}dZ \|_{\CO} < \EPS \WORD{and} K < \EPS, \] 
	   then the \( L^p \) Dirichlet problem is solvable.
\end{itemize}
\end{thm}
Statement (ii) follows from \cite[Theorem 2.2]{dindos_lp_2007};
    though it is stated there for bounded matrices it holds in our case as well.
To prove (i) we apply the following theorem 
    (see \cite[Theorem 1.3]{dindos_bmo_2017} or \cite[Theorem 4.1]{kenig_square_2014}):
\begin{thm}\label{thm:BMODirichletProblem}
	If 
	\begin{align}\label{eq:BMODirichletProblem}
	\| |\nabla u|^2 \delta(X) dX \|_{\CO} 
	    \lesssim \|f\|_{L^\infty(\partial\Omega)}^2, \quad \forall f \in C(\partial \Omega), 
	\end{align}
	then \( \omega \in A_{\infty}(d\sigma) \).
\end{thm}

It remains to show \eqref{eq:BMODirichletProblem}.
Let \( \Delta \subset \partial \Omega \) be a boundary ball with \( \Diam \Delta \leq \gamma \),
    where \( \gamma \) is taken small enough that \( T(\Delta) \) lies above a Lipschitz graph.
Note that by  \cite[Cor 5.2]{dindos_regularity_2018} we have
\[ \int_{T(\Delta)} |\nabla u(X)|^2\delta(X) dX \lesssim \int_{2\Delta} (|f|^2 + N_{\alpha}(u)^2)d\sigma. \]

(Again although \( \hat{A} \) is assumed to be bounded in \cite{dindos_regularity_2018} 
    this assumption is not necessary this Corollary to hold as it only uses ellipticity and boundedness of the symmetric part of the matrix).
Thus  \( f \in C(\partial \Omega) \) and 
	the maximum principle imply
\[ \int_{T(\Delta)} |\nabla u(X)|^2\delta(X) dX 
    \lesssim \|f\|_{L^\infty}^2 \sigma(2\Delta) 
    \lesssim \|f\|_{L^\infty}^2 \sigma(\Delta). \]

Dividing both sides by \( \sigma(\Delta) \) and taking supremum over \( \Delta \)
    yields \eqref{eq:BMODirichletProblem}.\qed

\subsection{Proofs of \refthm{thm:ApplicationBigNorm} and \refthm{thm:ApplicationSmallNorm}}
    
In order to prove \refthm{thm:ApplicationBigNorm} and \refthm{thm:ApplicationSmallNorm} we have the following strategy.

First we construct a matrix \( \hat{A} \) from \( A \). The objective is to improve regularity of coefficients in order to 
use \refthm{thm:ApplicationClassic} for \( \hat{A}\). We then deduce solvability for the original matrix $A$ by applying  our perturbation results.\\

To begin with, let \( B(X) = B(X,\delta(X)/2) \)
	and set 
\[  \hat{B}(X) \coloneqq B(X,\tfrac{2}{5} \tfrac{\delta(X)}{2}), \]
	so that 
\[ \bigcup_{X\in \hat{B}(Z)} \hat{B}(X) \subset B(Z). \]

In order to apply \refthm{thm:ApplicationClassic}  our matrix 
    needs to be differentiable.
Thus we define \( \hat{A} \) from \( A \) using a mollification procedure.

Consider \( \phi \in C_c^\infty(\frac{1}{5}\mathbb{B}^n) \) to be nonnegative with \( \int_{\R^n} \phi = 1 \),
    and \( \phi_t(X) = t^{-n}\phi(X/t) \). Let $\delta(X)$ be a smooth version of the distance function and
\begin{align} \hat{A}(X) \coloneqq (\phi_{\delta(X)} * A)(X). \label{eq:DefofhatA}\end{align}

Clearly \( \hat{A}(X) \) is differentiable with
\begin{align}\label{eq:HatAderivative}
    \nabla \hat{A}(X)=\int_\Omega (A(Y)-b) \nabla_X \phi_{\delta(X)}(X-Y)dY, \quad\mbox{for any } b \in \mathbb{R}^{n\times n}.     
\end{align}

If we can show that
\begin{align}\label{eq:HatAisCarleson}
    \| \hat{\alpha}(Z)^2 \delta(Z)^{-1} dZ \|_{\CO} \lesssim \| \alpha_r(Z)^2 \delta(Z)^{-1} dZ \|_{\CO},
\end{align}
holds 
    for \(\hat{\alpha}:=\hat{\alpha}^{\frac{2}{5}}\) clearly \refthm{thm:ApplicationClassic} implies that:
    
\begin{lemma}\label{lemma:AhatIsCarleson}
	Let \(\Omega\) be a bounded Lipschitz domain with Lipschitz constant \(K>0\) 
	Let \(\alpha_r\) be defined like in \eqref{eq:defofalpha_r} 
	    and let \(\omega\) be the elliptic measure of the operator \(L=\mathrm{div}(\hat{A}\nabla\cdot)\).
	Then there exists \(1<r=r(n,\lambda_0,\Lambda_0)<\infty\) such that
\begin{itemize}
    \item[(i)] If 
	\[  \| \alpha_r(Z)^2 \delta(Z)^{-1} dZ \|_{\CO} < \infty \] 
	     then \(\omega\in A_\infty(\sigma)\), i.e. the \( L^p \) Dirichlet problem \emph{for} \( \hat{A} \) 
	     is solvable for \emph{some} \( 1 < p < \infty \).\\
	     
	\item[(ii)] For every \( 1 < p < \infty \) there exists an \( \EPS = \EPS(p) > 0 \), such that if
	\[ \| \alpha_r(Z)^2 \delta(Z)^{-1} dZ \|_{\CO} < \EPS \WORD{and} K < \EPS, \] 
	    then \(\omega\in B_p(\sigma)\), i.e., the \( L^p \) Dirichlet problem \emph{for} \( \hat{A} \) is solvable.
\end{itemize}
\end{lemma}

To prove \eqref{eq:HatAisCarleson} it suffices to show that
\begin{align}\label{eq:AlphaHattoAlpha}
    \hat{\alpha}(Z) \lesssim \alpha_r(Z).    
\end{align}

Take \( b = (A)_{B(Z)} \) in \eqref{eq:HatAderivative}. Then
\[ |\nabla \hat{A}(X)| 
    \leq  \int_{\hat{B}(X)} |A(Y)-(A)_{B(Z)}| |\nabla\phi_{\delta(X)}(X-Y)|dY. \]
Let us estimate the gradient term inside the integral. 
\begin{align*}
    \delta(X)^{n+1} |\nabla_X \phi_{\delta(X)}|
    &\lesssim |\nabla \delta(X)| \phi\left(\frac{X-Y}{\delta(X)}\right) 
        + \bigg|\nabla \phi\left(\frac{X-Y}{\delta(X)}\right)\bigg| 
            \bigg(  |\nabla \delta(X)| \frac{|X-Y|}{\delta(X)} + 1 \bigg) 
    \\
    &\leq \| \phi \|_{L^\infty} + 2\| \nabla \phi \|_{L^\infty} \lesssim 1,  
\end{align*}
    since \( |X-Y| \leq \delta(X) \) and \( |\nabla \delta | \leq 1 \).
It follows that
\[ |\nabla_X \phi_{\delta(X)}| \lesssim \delta(X)^{-(n+1)}. \]

This implies that for any \(X\in \hat{B}(Z)\) we have that
\begin{align*}
    |\nabla \hat{A}(X)| &\lesssim  \frac{1}{\delta(X)^{n+1}} \int_{\hat{B}(X)}|A(Y)-(A)_{B(Z)}|dY
    \\
    &\leq \frac{1}{\delta(X)^{n+1}} \int_{B(Z)}|A(Y)-(A)_{B(Z)}|dY
    \\
    &\approx \frac{1}{\delta(Z)} \fint_{B(Z)}|A(Y)-(A)_{B(Z)}| dY
    \\
    &\leq \frac{1}{\delta(Z)} \left(\fint_{B(Z)}|A(Y)-(A)_{B(Z)}|^rdY\right)^{1/r}
    = \frac{1}{\delta(Z)}\alpha_r(Z).
\end{align*}

From this
\[ \hat{\alpha}(Z) = \delta(Z) \sup_{X \in \hat{B}(Z)} |\nabla \hat{A}|  \lesssim \alpha_r(Z), \]
    as desired.\\
    
It remains to apply our two perturbation results  for \( A_0 = \hat{A} \) and \( A_1 = A \).
Clearly \( \hat{A} \) is \( \lambda_0 \)-elliptic.
Moreover we can see that \(\Vert \hat{A}\Vert_{BMO(\Omega)}\lesssim \Lambda_0\). 

To see this we distinguish two cases. 
First, consider a ball \(B\subset \Omega\) such that  \(B\not\subset\hat{B}(X)\) is true for all $X\in\Omega$. 
Then we can find a cover with balls \((\hat{B}(X_i))_{i}\) such that the balls \(\hat{B}(X_i)\) have finite overlap, 
    and \(|\bigcup_i \hat{B}(X_i)|\lesssim |B|\).
The constants in the last inequality are independent of \(B\). By Lemma 2.1 of \cite{jones_extension_1980} we know that \(|(A)_{\hat{B}(X)}-(A)_{B(Z)}|\lesssim \Lambda_0\) for all \(X\in \hat{B}(Z)\). Hence

\begin{align*}
    \fint_B &|\hat{A}-(A)_B|dX\leq \fint_B|\hat{A}-A|dX+\fint_B|A-(A)_B|dX
    \\
    &\leq \fint_B\left|\int_{\hat{B}(X)}(A(Y)-A(X))\frac{\phi\left(\frac{X-Y}{\delta(X)}\right)}{\delta(X)^n}dY\right|dX+\Lambda_0
    \\
    &\leq \fint_B\fint_{\hat{B}(X)}|A(Y)-(A)_{\hat{B}(X)}|dY +|A(X) - (A)_{\hat{B}(X)}|dX+\Lambda_0
    \\
    &\leq \fint_B |A(X)-(A)_{\hat{B}(X)}|dX + 2\Lambda_0
    \\
    &\lesssim \frac{1}{|B|}\sum_i\int_{\hat{B}(X_i)}|A(X)-(A)_{B(X_i)}|+|(A)_{B(X_i)}-(A)_{\hat{B}(X)}|dX+2\Lambda_0
    \\
    &\lesssim \frac{1}{|B|}\sum_i|B(X_i)|\fint_{B(X_i)}|A(X)-(A)_{B(X_i)}|+\Lambda_0 dX+2\Lambda_0
    \\
    &\lesssim\lambda_0\left(2+\frac{1}{|B|}\sum_i|B(X_i)|\right)\lesssim\Lambda_0.
\end{align*}
The second case is if \(B\subset \hat{B}(X_1)\) for some \(X_1\in \Omega\). Then we have
\begin{align*}
    \fint_B |\hat{A}-(A)_{B(X_1)}|dX&\leq \fint_B\left|\int_{\hat{B}(X)}(A(Y)-(A)_{B(X_1)})\frac{\phi\left(\frac{X-Y}{\delta(X)}\right)}{\delta(X)^n}dY\right|dX
    \\
    &\leq \fint_B\fint_{\hat{B}(X)}|A(Y)-(A)_{B(X_1)}|dYdX
    \\
    &\lesssim \fint_B\fint_{B(X_1)}|A(Y)-(A)_{B(X_1)}|dYdX\lesssim\Lambda_0.
\end{align*}

Thus we can conclude that
\[\inf_{M\in \mathbb{R}^{n\times n} \textrm{ constant}}\fint_B|\hat{A}(X)-M|dX\lesssim \Lambda_0, \]
    which implies \(\Vert \hat{A}\Vert_{BMO(\Omega)}\lesssim\Lambda_0\).
It follows that we indeed may apply our perturbation results.    
    
Let
\[\beta_r(Z) \coloneqq \left(\fint_{B(Z)} |\hat{A}(Y)-A(Y)|^rdY\right)^{1/r}.\]

Our next objective is to show that
\begin{align}\label{eq:ApplicationPerturbation}
    \| \beta_r(Z)^2 \delta^{-1}(Z) dZ \|_{\CO} \lesssim \| \alpha_r(Z)^2 \delta(Z)^{-1} dZ \|_{\CO}. 
\end{align}
Assume for the moment this is indeed true. Then \reflemma{lemma:AhatIsCarleson} and \refthm{thm:NormBig}
imply \refthm{thm:ApplicationBigNorm}. Similarly, \reflemma{lemma:AhatIsCarleson}  and \refthm{thm:NormSmall} imply  \refthm{thm:ApplicationSmallNorm}. Thus if we establish \eqref{eq:ApplicationPerturbation} we are done.

We start by observing that the following estimate holds:
\begin{align*}
    \left(\fint_{B(Z)}|\hat{A}-A|^r\right)^{1/r}
    &\leq \left(\fint_{B(Z)}|\hat{A}-(A)_{B(Z)}|^r\right)^{1/r}+\left(\fint_{B(Z)}|A-(A)_{B(Z)}|^r\right)^{1/r}
    \\
    &= \left(\fint_{B(Z)}|\hat{A}-(A)_{B(Z)}|^r\right)^{1/r} + \alpha_r(Z).
\end{align*}

The last term already has the required form. For the first term we see that
\begin{align*}
    \left(\fint_{B(Z)}|\hat{A}-(A)_{B(Z)}|^r\right)^{1/r} 
    &\leq \bigg(\fint_{B(Z)} \left| 
        \int_{\hat{B}(X)} |A(Y)-(A)_{B(Z)}| \phi_{\delta(X)}(X-Y) dY \right|^r dX\bigg)^{1/r}
    \\
    &\leq \| \phi \|_{L^\infty} \bigg(\fint_{B(Z)} \left| 
        \delta(X)^{-n} \int_{\hat{B}(X)} |A(Y)-(A)_{B(Z)}|  dY \right|^r dX \bigg)^{1/r}
    \\
    &\lesssim  \bigg(\fint_{B(Z)} \left| 
        \fint_{\hat{B}(X)} |A(Y)-(A)_{B(Z)}| dY \right|^r dX\bigg)^{1/r}
    \\
    &\lesssim \bigg(\fint_{B(Z)} |A(Y)-(A)_{B(Z)}|^r dX\bigg)^{1/r} = \alpha_r(Z).
\end{align*}
\qed

\bibliographystyle{alpha}

\end{document}